\documentclass[a4paper,11pt]{article}
\usepackage[utf8]{inputenc}
 \usepackage[normalem]{ulem}
 \usepackage[english]{babel}
 \usepackage{color}

\usepackage{hyperref,float}
\usepackage[usenames,dvipsnames,svgnames,table]{xcolor}

\usepackage[citestyle =numeric-comp,bibstyle = numeric, maxnames= 12,firstinits=true,abbreviate=true,isbn = false, url = false,doi=false,eprint=false, backend=bibtex]{biblatex}
\bibliography{WKB_GAUSSIEN}
\AtEveryBibitem{\clearfield{month}}
\AtEveryCitekey{\clearfield{month}}
\AtEveryBibitem{\clearlist{language}}
\AtEveryBibitem{\clearlist{location}}
\renewbibmacro{in:}{}

\usepackage{amsmath}
\usepackage{amssymb}
\usepackage{amsthm}
\usepackage{amsfonts}
\usepackage{appendix}
\usepackage{subcaption}
\usepackage{graphicx}
\usepackage{enumerate}
\usepackage{fullpage}

\newcommand{\ds}{\displaystyle}
\newcommand{\bt}{\hat t}
\newcommand{\bx}{\hat x}
\newcommand{\bs}{\hat s}
\newcommand{\by}{\hat y}
\renewcommand{\tt}{\tilde t}
\newcommand{\tx}{\tilde x}
\newcommand{\ts}{\tilde s}
\newcommand{\ty}{\tilde y}
\newcommand{\tv}{\tilde v}
\newcommand{\tw}{\tilde w}
\newcommand{\wtildechi}{\widetilde \chi}
\newcommand{\overlinechi}{\widehat \chi}

\newcommand{\LambdaB}{B}

\numberwithin{equation}{section}

\newcommand{\supp}{\mathrm{supp}\,}

\newcommand{\ind}[1]{\mathbf{1}_{#1}\,}

\newcommand{\R}{\mathbb{R}}

\newcommand{\mA}{\mathcal{A}}
\newcommand{\mB}{\mathcal{B}}
\newcommand{\mU}{\mathcal{U}}
\newcommand{\mmA}{\mathfrak{A}}
\newcommand{\e}{\varepsilon}
\renewcommand{\epsilon}{\varepsilon}
\newcommand{\eps}{\varepsilon}

\renewcommand{\u}{u}
\DeclareMathOperator{\spn}{span}

\newcommand{\argmin}{\mathrm{argmin}\;}

\newcommand{\minp}{\mathrm{min}^\prime\,}
\newcommand{\aminp}{\mathrm{argmin}^\prime\,}

\newtheorem{thm}{\bf Theorem}[section]
\newtheorem{prop}[thm]{\bf Proposition}
\newtheorem{rmq}[thm]{\bf Remark}
\newtheorem{lem}[thm]{\bf Lemma}
\newtheorem{defi}[thm]{Definition} 

\newcommand{\tg}{\textcolor{black}}

\newcommand{\Leb}{\mathrm{Leb}}

\date\today
\author{Emeric Bouin 
\footnote{
CEREMADE, UMR 7534 CNRS \& PSL University, Universit\'e Paris-Dauphine, Place de Lattre de Tassigny, Paris, France.\newline
\texttt{bouin@ceremade.dauphine.fr}}
\and
Vincent Calvez \footnote{Institut Camille Jordan, UMR 5208 CNRS \& Universit\'e Claude Bernard Lyon 1, and Inria, project team Dracula, Lyon, France.\newline 
\texttt{vincent.calvez@math.cnrs.fr}}\and
Emmanuel Grenier \footnote{Unit\'e de Math\'ematiques Pures et Appliqu\'ees, UMR 5669 CNRS \& Ecole Normale Sup\'erieure de Lyon, and Inria, project team NUMED, Lyon, France.\newline
\texttt{emmanuel.grenier@ens-lyon.fr}}\and
Gr\'egoire Nadin \footnote{Laboratoire Jacques-Louis Lions, UMR 7598 CNRS \& Sorbonne Universit\'e, Paris, France.\newline 
\texttt{nadin@ann.jussieu.fr}}}

\title{Large scale asymptotics of velocity-jump processes and non-local Hamilton-Jacobi equations}
\begin{document}
\maketitle

\begin{abstract}
We investigate a simple velocity jump process in the regime of large deviation asymptotics. New velocities are taken randomly  at a constant, large, rate from a Gaussian distribution with vanishing variance. The Kolmogorov forward equation associated with this process is the linear BGK kinetic transport equation.
We derive a new type of Hamilton-Jacobi equation which is nonlocal with respect to the velocity variable. We introduce a suitable notion of viscosity solution, and we prove well-posedness in the viscosity sense. We also prove convergence of the logarithmic transformation towards this limit problem.  
Furthermore, we identify the variational formulation of the solution by means of an action functional supported on piecewise linear curves.  
As an application of this theory, we compute the exact rate of acceleration in a kinetic version of the celebrated Fisher-KPP equation in the one-dimensional case. 
\end{abstract}

\pagestyle{plain}
\pagenumbering{arabic}

\section{Introduction}

This paper is mainly concerned with the asymptotic limit of the following linear kinetic transport equation as $\eps\to 0$,
\begin{equation} \label{eq:main}
\partial_t f^{\eps}(t,x,v) + v\cdot \nabla_{x} f^{\eps}(t,x,v) = \frac{1}{\epsilon} \left( M_\eps(v) \rho^{\eps}(t,x)- f^{\eps}(t,x,v) \right ), \quad t>0\, , \, x\in \R^n\, , \, v\in \R^n\, .
\end{equation}
Here, $f^\eps(t,x,v)$ denotes the density of particles at time $t>0$ in the phase space $\R^{n}\times \R^n$ (position$\times$velocity), and $\rho^\eps(t,x)$ is the spatial density: 
$\rho^\eps(t,x) = \int_{\R^n} f^\eps(t,x,v') dv'$.
Particles move with velocity $v$. Reorientation occurs at random exponential times with rate $1/\eps$. 
The velocity distribution of reorientation events is given, and is denoted by  $M_\eps(v)$. We opt here for the Gaussian distribution with variance $\eps$:
$M_\eps(v) = \frac{1}{(2\pi\eps)^{n/2}} \exp\left( - \frac{|v|^2}{2\eps} \right)$.
However, we believe that our methodology could be applied to a broader range of distributions.

As such, the underlying velocity-jump process belongs to the class of Piecewise Deterministic Markov Processes (PDMP).  

Equation \eqref{eq:main} is obtained from the unscaled problem  ($\eps = 1$) in the scaling regime $\left(\frac{ t}{\eps},\frac{ x}{\eps^{3/2}},\frac{ v}{\eps^{1/2}} \right)$ which is appropriate to establish a Large Deviation Principle. Indeed, we prove under some conditions that the rate function $u^\eps(t,x,v) = -\eps \log f^\eps(t,x,v)$ converges, as $\eps \to 0$, towards the solution of the following problem with appropriate initial condition:
\begin{equation}\label{eq:limit}
\begin{cases}
\displaystyle\max\left( \partial_t \u(t,x,v) + v \cdot \nabla_x \u(t,x,v) - 1 , \u(t,x,v) - \min_{v'\in \R^n} \u(t,x,v') - \dfrac{\vert v\vert^2}2 \right) = 0\, ,  
\medskip \\
\displaystyle\partial_t \left( \min_{v'\in \R^n}\u(t,x,v') \right) \leq 0\, , 
\quad \text{and}\quad \displaystyle\partial_t \left( \min_{v'\in \R^n}\u(t,x,v') \right) = 0  \; \text{ if }\; \underset{v'\in \R^n}{\argmin}\u(t,x,v') = \left\lbrace 0 \right\rbrace\,.
\end{cases}
\end{equation}
%
About the initial data, we assume for simplicity that it is well-prepared, in the form $f^\eps(0,x,v) = \exp\left( - \frac{ \u_0(x,v)}\eps\right)$. This sets  $u(0+,x,v) = \min\left (u_0(x,v) ,\min_{v'}  u_0(x,v') + |v|^2/2\right )$ as the initial condition for the limit problem \eqref{eq:limit}.

The notion of viscosity solution of \eqref{eq:limit} is made precised by means of sub- and super-solutions in the pair of Definitions \ref{def:subsol intro} and \ref{def:supersol intro} below.

We derive the representation formula for \eqref{eq:limit}:
\begin{equation}\label{eq:kinHopf Lax}
u(t,x,v) = \inf_{\footnotesize \begin{array}{c}
\{\gamma: \gamma(t) = x, \dot\gamma(t) = v\}
\end{array}} \left \{ \mA_0^t[\dot \gamma] + u_0(\gamma(0),\dot\gamma(0))  \right  \}
\end{equation}
where the action of a piecewise linear curve $\gamma$ over the time interval $(0,t]$ is given by:
\begin{equation}\label{eq:action-intro}
\mA_0^t[\dot \gamma] = \frac12 \sum_{\sigma\in \dot\Gamma_*} |\sigma|^2 + \Leb \left\{s\in (0,t]:\dot\gamma(s) \neq 0\right\}\,, 
\end{equation}
where $\dot \Gamma_*$ denotes the finite list of velocities $(\dot \gamma(s))_{s\in (0,t]}$ but the initial one. Alternatively speaking, each non-zero velocity $\sigma$ after the first velocity jump contributes to a single cost of $\frac12|\sigma|^2 $ and a running cost of one per unit of time. We refer to Definition \ref{def:action} for a precise definition of the action $\mA$.

Our approach is similar to the analysis of the heat equation in the scaling regime $\left(\frac{ t}{\eps},\frac{ x}{\eps}\right)$, that is, with a vanishing viscosity: $\partial_t \rho^\eps(t,x) - \frac\eps2 \Delta \rho^\eps(t,x) = 0$. Indeed, it is well-known (see \textit{e.g.} \cite{fleming_exit_1977,evans_pde_1985,
evans_pde_1989,barles_wavefront_1990}) that, under appropriate conditions, $u^\eps(t,x) = -\eps \log \rho^\eps(t,x)$ converges  locally uniformly towards the viscosity solution of the following Hamilton-Jacobi equation
\begin{equation} 
\partial_t u(t,x) + \frac12|\nabla_x u(t,x)|^2 = 0\, ,  \label{eq:HJ-heat}
\end{equation}
with appropriate initial condition. Moreover, $u$ is given by the Hopf-Lax variational formula $u(t,x) = \inf_y \left \{ \frac{|x-y|^2}{2t} + u_0(y) \right \}$. The latter formulation is in correspondence with \eqref{eq:kinHopf Lax}, whereas we claim that \eqref{eq:HJ-heat} and \eqref{eq:limit} are analogous. 


To the best of our knowledge, system \eqref{eq:limit} is original. We refer to it as a Hamilton-Jacobi problem by analogy with \eqref{eq:HJ-heat} which is obtained via a similar procedure. Also, the first equation in \eqref{eq:limit}, and   the first part of the action \eqref{eq:action-intro} are similar to quasi-variational inequalities that can be found in the classical formulation of impulse control problems \cite{bensoussan_impulse_1984,
barles_deterministic_1985,
perthame_remarks_1985}. Indeed, sudden velocity changes that persist in the limit $\eps\to 0$ can be viewed as impulses with a cost $|\sigma|^2/2$.  
In the case of a compactly supported velocity distribution $M(v)$, the same procedure leads to a standard Hamilton-Jacobi equation in the space variable in the scaling regime $\left(\frac{ t}{\eps},\frac{ x}{\eps}\right)$ \cite{bouin_kinetic_2012,bouin_hamilton-jacobi_2015, caillerie_large_2017}. However, there are no similar impulses persisting in the limit $\eps\to 0$, and so the two cases (compactly supported {\em vs.} Gaussian distribution) behave quite differently (see further discussion below in the next paragraph and Section \ref{sec:intro compact}).

\paragraph{Connection with large deviations.}
\color{black}
Our work can be viewed as a prefiguration of a  large deviation result for the underlying velocity jump process beneath \eqref{eq:main}. 
We follow the methodology of Fleming's logarithmic transformation \cite{fleming_exit_1977},  Evans and Ishii \cite{evans_pde_1985}, Evans and Souganidis,\cite{evans_pde_1989} and Barles, Evans and Souganidis  \cite{barles_wavefront_1990}, in which PDE techniques were successfully applied to derive large deviations results  in the context of viscosity solutions of Hamilton-Jacobi equations. We also refer the reader to the monograph by Feng and Kurtz for a comprehensive presentation of this approach in the context of stochastic processes \cite[Chapter 6]{feng_large_2006}.

In particular, consider, in dimension $n=1$, the Markov process $(X'_t,V'_t)$ whose law is described by \eqref{eq:main}, starting from the origin with zero velocity, that is, $f(0,x',v') = \delta_{(0,0)}$. Then, we have
\begin{align*}
\frac{1}{t}\log\mathbb{P}\left ( X'_{t}\geq ct^{3/2}\right ) 
&= \frac1t \log \left (  \int_{ x' \geq ct^{3/2}} f(t,x',v') dx'dv' \right )\\
& = \e\log \left ( \int_{x \geq c} \exp\left ( -\frac{u_\e(1,x,v)}{\e}\right )dxdv \right )+o(1)
\end{align*}
after the change of  variables $x=\frac{x'}{t^{3/2}}$ and $v=\frac{v'}{t^{1/2}}$, and the notation  $ t= \frac1\e$, in accordance with the scaling regime considered in \eqref{eq:main}. 
We hereby develop analytical tools that would contribute to  show that the right-hand-side converges to $-\min_{\{x\geq c, v\in \R\}}u(1,x,v)$, where $u$ is the unique viscosity solution of \eqref{eq:limit}. Therefore, we shall say that $(X'_t)$ satisfy a large deviation principle with rate $\frac1{t^{3/2}}$ and rate function $\min_{v\in \R} u(1,\cdot,v)$. Crucial steps are the convergence of $u^\eps$ towards $u$ (Section \ref{sec:Conv}), and the handling of singular initial data, such as in Section \ref{sec:acc}.

One salient feature of this large deviation principle is the nature of the action\eqref{eq:action-intro}. While large deviations of random paths are measured typically by minimizing the classical action $\int_0^t L(\gamma(s),\dot \gamma(s))\, ds$, for an appropriate Lagrangian function $L$,  over smooth curves $\gamma$ (see the original theorem by Schilder \cite{schilder_asymptotic_1966} for the Brownian motion, and the introduction in \cite{feng_large_2006} for further examples), here the action is of a different nature, as it is supported by piecewise linear trajectories instead of smooth curves. 

The case of bounded velocities, however, falls into the standard framework with an action deriving from a Lagrangian function $L(x,v)$ as usual, see Section \ref{sec:intro compact}. Moreover, the rate function $u$ does not depend on the variable $v$ due to an averaging property only valid for bounded velocities \cite{bouin_kinetic_2012,bouin_hamilton-jacobi_2015, caillerie_large_2017}, see also \cite{faggionato_averaging_2008}. Hence, the possibility of arbitrarily large velocities reveals a new scaling regime calling for an extended notion of Lagrangian functionals.  

After the completion of this work, the connection with the notion of {\em big jump} in large deviations was notified to the authors by Jara and Mallein \cite{Jara-Mallein-perso,jara_mallein_note}. Roughly speaking, the distribution of spatial increments of the underlying velocity-jump process falls into the subexponential regime, for which Cram\'er’s condition does not hold. In that case, the large deviation events are dominated by a single large increment, see \cite{denisov_large_2008} and references therein, and also \cite{dyszewski_large_2020} in the context of branching random walks. This single large increment will be illustrated in Proposition \ref{prop:sigma1} where the minimizing curves are described to some extent.


\paragraph{Non local Hamilton-Jacobi equations.}
The analysis of some non-local Hamilton-Jacobi equations can be found in the literature, see for instance \cite{soner_optimal_1988,
alvarez_viscosity_1996,
awatif_equations_1991}, and the series of papers about non-local eikonal equations in the context of dislocation dynamics in crystals, see \textit{e.g.} \cite{cardaliaguet_front_2000,
cardaliaguet_front_2001,
alvarez_existence_2005,
barles_uniqueness_2009} and the references therein; and also \cite{barles_geometrical_2003,da_lio_nonlocal_2004} in the context of geometric motions. However, there seems to be no link between these works and ours.

\paragraph{Organization of the paper.}
Section \ref{subsec:heuristics} contains an informal discussion about the limit system \eqref{eq:limit} and some heuristics about the asymptotics leading from \eqref{eq:main} to \eqref{eq:limit}. 

In order to state our results, we need a proper definition of viscosity solutions of \eqref{eq:limit}, as it does not fit apparently in the standard theory of Hamilton-Jacobi equations. This is the content of Section \ref{sec:intro viscosity}. The results of uniqueness of viscosity solutions and convergence of $u^\eps =-\eps \log f^\eps$ as $\eps\to 0$, as well as the representation formula \eqref{eq:kinHopf Lax} are presented in Section \ref{sec:intro results}. The proof of uniqueness is contained in Section \ref{sec:Comp} for the case of bounded solutions (along the spatial variable), then extended in Appendix \ref{sec:unbded} to the case of unbounded solutions assuming quadratic bounds. Section \ref{sec:Conv} contains the proof of convergence. Section \ref{sec:variational} is devoted to establishing the variational representation formula.

We discuss some qualitative properties of the solutions in Section \ref{sec:intro bdd vel}, related to some explicit computations of the action $\mA$ contained in Section \ref{sec:kernel}. Finally, an application to front acceleration in reaction-transport equations is presented in Section \ref{sec:intro accel}, which is a summary of Section \ref{sec:acc}. 
%

\paragraph{Notations.} We introduce the following notations for the sake of conciseness: 
\begin{equation*}
\begin{cases}
\displaystyle\minp u(t,x) = \min_{v'\in \R^n} \u(t,x,v')\\
\displaystyle\aminp u(t,x) = \underset{v'\in \R^n}{\argmin}\u(t,x,v')
\end{cases}
\end{equation*}
If $u$ is  lower semi-continuous ({\em resp.} upper semi-continuous), we denote by $u(0+,x,v)$ its lower limit ({\em resp.} upper  limit) at time $t=0$: 
\[
u(0+,x,v) = \liminf_{\footnotesize\begin{array}{c}
\tau\to 0\vspace{-2pt}\\ \tau>0
\end{array}}u(\tau,x,v) \quad \left ( \text{\em resp.}\; \limsup_{\footnotesize\begin{array}{c}
\tau\to 0\vspace{-2pt}\\ \tau>0
\end{array}}u(\tau,x,v)\right )
\]

\paragraph{Acknowledgement.} \emph{The authors are indebted to Guy Barles who kindly shared a number of comments on some preliminary version of this work. He should be credited for the simplification of the proof of the comparison principle, the judicious use of half-relaxed limits, and the appropriate formulation of the initial data. Further discussion with him encouraged the authors to seek a representation formula, leading to the completion of the approximation of geometric optics in Section \ref{sec:acc}. The authors also thank Milton Jara and Bastien Mallein for pointing the insightful connection with the regime of big jump in large deviations. This article has benefited from a careful reading by two anonymous reviewers, whose comments have lead to an improvement of the manuscript. \\
This project has received funding from the European Research Council (ERC) under the European Union’s Horizon 2020 research and innovation programme (grant agreement No 639638 and grant agreement No 865711).}

\subsection{Informal description of the dynamics and heuristics}
\label{subsec:heuristics}

The system  \eqref{eq:limit}  is not  a standard Hamilton-Jacobi equation. The first equation of \eqref{eq:limit} does not contain enough information due to the occurrence of $\minp \u$ for which extra dynamics are required. Although it seems somehow sparse, the two additional (in)equations $\partial_t \left(\minp u\right) \leq 0\; (=0)$ are sufficient to determine a unique solution of the Cauchy problem, as stated in the comparison principle below (Theorem \ref{theo:comp}).

In order to get some insight about the well-posedness of  \eqref{eq:limit}, we propose the following description of the typical dynamics of its solution $u$. 
The first condition in \eqref{eq:limit} guarantees that the following constraint must be satisfied everywhere: 
\begin{equation}\label{eq:parabolic constraint}
u(t,x,v) \leq  \minp u(t,x) + \dfrac{\vert v\vert^2}2 \, .
\end{equation}  
Consequently, the solution reaches its global minimum with respect to the velocity variable at $v = 0$. 
Furthermore, the following dichotomy holds: 
\begin{enumerate}[(i)]\item either the constraint is saturated: $u = \minp u + \frac{|v|^2}2$, 
\item or the solution is driven by free transport: $\partial_t u + v \cdot \nabla_x \u = 1$.
\end{enumerate}
Then, two more cases must be distinguished: if $v=0$ is the only global minimal point with respect to velocity ($\aminp\u(t,x) = \left\lbrace 0 \right\rbrace$), then the minimal value does not change, see Figure \ref{fig:Sua}. Hence, the parabolic constraint \eqref{eq:parabolic constraint} does not change as well. Nevertheless, the solution in the unsaturated area can still evolve by free transport and decay. If it touches the minimal value somewhere else, then the condition $\aminp\u(t,x) = \left\lbrace 0 \right\rbrace$ is not satisfied anymore, and the minimal value can possibly decrease, together with the parabolic constraint, see Figure \ref{fig:Sub}. It is the decay in the free zone that drives the global decay of the solution.

\begin{figure}[t]
\begin{center}
\begin{subfigure}{0.45\linewidth}
\includegraphics[width=\linewidth]{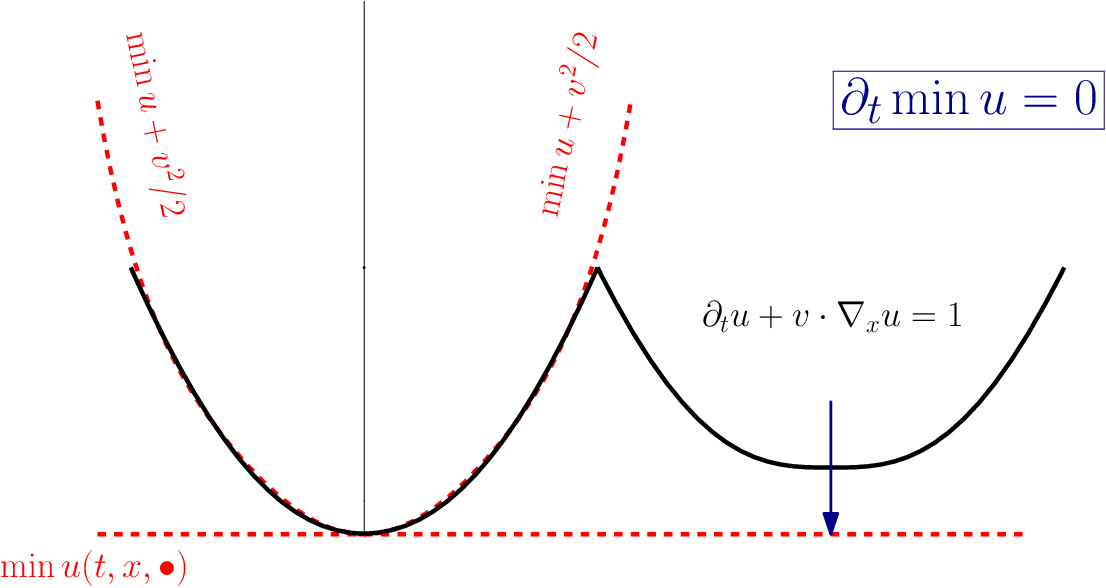}
\caption{The case $\aminp u(t,x) = \{0\}$}
\label{fig:Sua}
\end{subfigure}
\begin{subfigure}{0.45\linewidth}
\includegraphics[width=\linewidth]{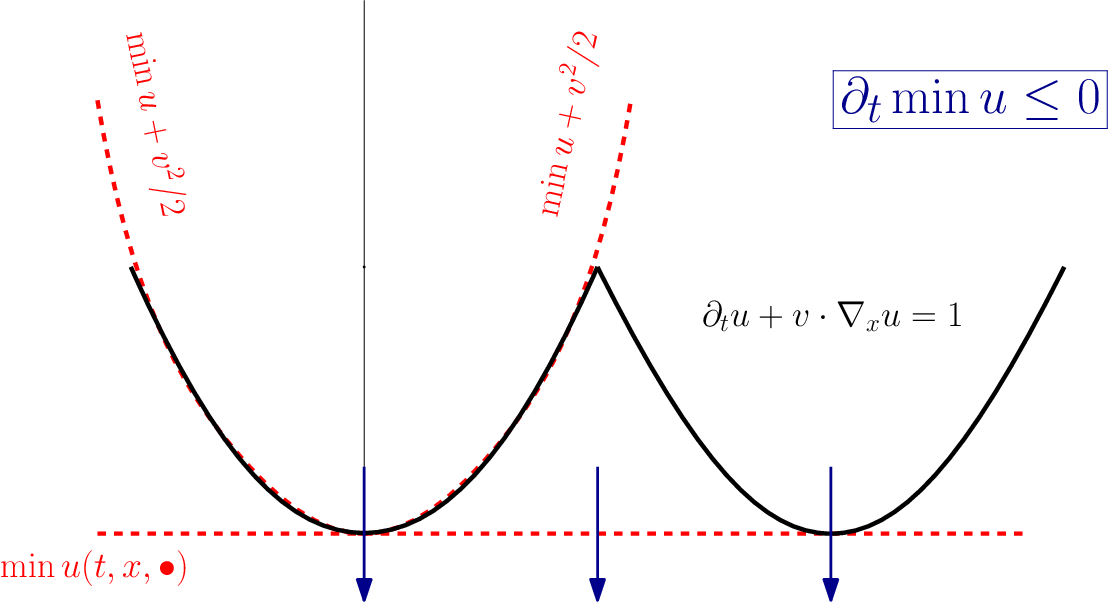}
\caption{The case $\aminp u(t,x) \neq \{0\}$}
\label{fig:Sub}
\end{subfigure}
\caption{Typical dynamics of solutions to \eqref{eq:limit}.}
\end{center}
\end{figure}

We also propose the following heuristics to describe the link between \eqref{eq:main} and \eqref{eq:limit}. Firstly, \eqref{eq:main} is equivalent to the following equation on $u^\eps$: 
\begin{equation}\label{WKB1}
\partial_t \u^{\eps}(t,x,v) + v \cdot \nabla_x  \u^{\eps}(t,x,v) - 1 = - \frac{1}{(2 \pi\eps)^{n/2} }
\int_{\R^n} \exp \left( \frac{\u^{\eps}(t,x,v) - \u^{\eps}(t,x,v') - \vert v\vert^2/2 }{\e} \right) dv'\, .
\end{equation}
On the one hand, it is immediate that the constraint \eqref{eq:parabolic constraint} is fulfilled in the limit $\eps\to 0$ provided that the left-hand-side is locally uniformly bounded. On the other hand, the continuity equation $\partial_t \int f^\eps\, dv + \nabla_x\cdot \int v f^\eps\, dv = 0$ is equivalent to 
\begin{equation}  
\int_{\R^n} \left(\partial_t u^\eps + v \cdot \nabla_x u^\eps \right)\, d\mu^\eps(v) = 0\, ,\quad  d\mu^\eps(v)   =  \dfrac{f^\e(t,x,v)}{\rho^\e(t,x)}\, dv\, .
\label{eq:heuristics 1}
\end{equation}


The probability measure $d\mu^\eps$ is expected to concentrate on the minimum points of $u$ with respect to $v$ as $\eps \to 0$. Let assume that we do have in some sense, 
\begin{equation} 
d\mu^\eps \rightharpoonup  \sum_{w\in \aminp u(t,x)} p_w \delta(v - w)\;  =\; p_0 \delta(v) + \sum_{w'\in \aminp u(t,x)\setminus\{0\}} p_{w'} \delta(v - w') \, ,
\label{eq:heuristics 2} \end{equation}
where the weights satisfy $\sum p_{w}= 1$. 
The constraint \eqref{eq:parabolic constraint} at each $w'\in \aminp u(t,x)\setminus\{0\}$ is clearly unsaturated, in the sense that $u(t,x,w') < \minp u+ {\vert w' \vert^2}/2$. There, we expect to see the right-hand-side contribution of \eqref{WKB1} vanish. This would lead to $\partial_t u(t,x,w') + w' \cdot \nabla_x u (t,x,w') = 1$ for each  $w'\in \aminp u(t,x)\setminus\{0\}$. Plugging this into \eqref{eq:heuristics 1}, and using \eqref{eq:heuristics 2}, we obtain successively,
\begin{multline*}
0=\sum_{w\in \aminp u(t,x)} p_{w}\left( \partial_t u + w \cdot \nabla_x u  \right) = p_0 \partial_t u(t,x,0) + \sum_{w'\in \aminp u(t,x)\setminus\{0\}} p_{w'} \\
=
 p_0 \partial_t u(t,x,0) + 1 - p_0\, .
 \end{multline*}
As we have formally $\partial_t u(t,x,0) = \partial_t \left( \minp u \right)(t,x)$ by the chain rule, we expect eventually that $\partial_t ( \minp u )  \leq 0$ and even $\partial_t (\minp u)  = 0$ if $p_0 = 1$, that is, somehow $\aminp u(t,x) = \left\lbrace 0 \right\rbrace$. All this reasoning is purely formal, but we shall make it rigorous in Section \ref{sec:Conv} with a different approach. 
It is indeed not necessary to describe accurately the limit of the probability measure $d\mu_\eps$ to establish the connection between \eqref{eq:main} and \eqref{eq:limit}. Besides, we believe that the characterization of the limiting measure would require a higher order description involving correctors, simply because an informal application of Laplace's method to decipher the weights (when $u_\eps$ does not depend on $\eps$, say), would require the knowledge of the Hessian at each mimimum point, a quantity which is far beyond our level of description.


\subsection{Informal connection between the variational formulation  and the Hamilton-Jacobi equation}

It is possible to recover, at least partially and informally, the Hamilton-Jacobi formulation  \eqref{eq:limit} from the variational formulation \eqref{eq:kinHopf Lax}. The action \eqref{eq:action-intro} is composed of a running cost which is absolutely continuous with respect to the Lebesgue measure, with rate $\mathbf{1}_{\{\dot \gamma \neq 0\}}$, and a discrete part, with a cost $|\sigma^2|/2$ at each new "impulsion". As in \cite{barles_deterministic_1985}, the impulse part leads to an implicit obstacle problem, which can be derived informally by the following calculation issued from \eqref{eq:kinHopf Lax},
\begin{equation}
\max \left \{ \lim_{s\to 0} \left ( \dfrac{u(t,x,v) - u(t-s,x-sv,v)}s -  \mathbf{1}_{v \neq 0}\right ) , u(t+,x,v) - u(t-,x,v') - \frac{|v|^2}{2}\right \} = 0
\end{equation}
where each occurrence in the supremum  corresponds to the alternative between a free run before $t$, or an impulse at time $t$. In contrast with the usual formulation, see {\em e.g.} \cite{barles_deterministic_1985}, the supremum is not taken over $v$ as it is a variable and not an unknown associated with the minimizing curves in the standard Lagrange variational formulation. Another difference with \cite{barles_deterministic_1985} is that the impulse control problem comes usually with a cost depending on $v'-v$, which is bounded below by a positive constant. Here, the cost depends only upon the posterior velocity $v$, and it is obviously not uniformly positive. Nevertheless, this formal procedure leads naturally to the following quasi-variational inequality:
\begin{equation}
\max \left \{ \partial_t u(t,x,v) + v\cdot \nabla_x u(t,x,v) -  \mathbf{1}_{v \neq 0} , u(t,x,v) - \min_{v'} u(t,x,v') - \frac{|v|^2}2 \right \} = 0\,. \label{eq:informal HJ}
\end{equation}  
The second line of \eqref{eq:limit} comes from the examination of this relationship at $v=0$. Indeed, at this point, we see that the second inequality in \eqref{eq:informal HJ} is an equality: $u(t,x,0) = \min' u(t,x)$. Therefore the first inequality brings the following information: $\partial_t u(t,x,0) \leq 0$, which can be reformulated as $\partial_t (\min' u)\leq 0$. The last condition in \eqref{eq:limit} seems more delicate to interpret. At this stage, we rely on the heuristics of the Figure \ref{fig:Sua}/\ref{fig:Sub} to explain this additional piece of information which is essential for the well-posedness of \eqref{eq:limit}.

The rigorous link between the Hamilton-Jacobi formulation  \eqref{eq:limit} and the variational formulation \eqref{eq:kinHopf Lax} is the purpose of Section \ref{sec:variational}.

\subsection{The notion of viscosity solution}
\label{sec:intro viscosity}

Equation \eqref{eq:limit} can be viewed as a coupled system of Hamilton-Jacobi equations on $u$ and $\minp\u$. Accordingly, we define viscosity solutions of \eqref{eq:limit} using a pair of  test functions as {\em e.g.} in \cite{lenhart_viscosity_1988,
engler_viscosity_1991}. 

\begin{defi}[Sub-solution]\label{def:subsol intro}
Let $\underline{u}_0$ be a continuous function, and $T>0$. An upper semi-continuous function $\underline{u}$ is a \textbf{viscosity sub-solution} of \eqref{eq:limit} on $(0,T) \times \R^{2n}$ with initial data $\underline{u}_0$ if the following conditions are fulfilled:\medskip \\
\noindent{\em (i)} $\underline{u}(0+,\cdot,\cdot) \leq  \underline{u}_0 $.\medskip \\
\noindent{\em (ii)} It satisfies the constraint
\[ \forall (t,x,v)\in (0,T)\times\R^{2n}\quad \underline{u}(t,x,v) -  \minp \underline{u}(t,x) - \dfrac{\vert v\vert^2}{2} \leq 0\, . \]
\\
{\noindent{\em (iii)}}
For all pair of test functions $(\phi,\psi) \in \mathcal{C}^1\left( (0,T) \times \R^{2n} \right) \times \mathcal{C}^1\left( (0,T) \times \R^{n} \right) $, if $(t_0,x_0,v_0)$ is such that both $\underline{u}(\cdot,\cdot,v_0) - \phi(\cdot,\cdot,v_0)$ and $  \minp \underline{u}  - \psi$ have a local maximum at $(t_0,x_0)$ with $t_0>0$, then
\begin{equation*}
\begin{cases}
\partial_t \phi(t_0,x_0,v_0) + v_0 \cdot \nabla_x \phi(t_0,x_0,v_0) - 1 \leq 0, \medskip \\
\partial_t \psi(t_0,x_0) \leq 0 .
\end{cases}
\end{equation*}
\end{defi}

\begin{defi}[Super-solution]\label{def:supersol intro}
Let $\overline{u}_0$ be a continuous function, and $T>0$. A lower semi-continuous function $\overline{u}$ is a \textbf{viscosity super-solution} of \eqref{eq:limit} on $(0,T) \times \R^{2n}$  with initial data $\overline{u}_0$ if the following conditions are fulfilled:
\medskip\\
{\noindent{\em (i)}} 
$ \overline{u}(0+,\cdot,\cdot) \geq \overline{u}_0$.\medskip \\
{\noindent{\em (ii)}}
For all pair of test functions $(\phi,\psi) \in \mathcal{C}^1\left( (0,T) \times \R^{2n} \right) \times \mathcal{C}^1\left( (0,T) \times \R^{n} \right) $, if $(t_0,x_0,v_0)$ is such that both $\overline{u}(\cdot,\cdot,v_0) - \phi(\cdot,\cdot,v_0)$ and $  \minp u  - \psi$ have a local minimum  at $(t_0,x_0)$ with $t_0>0$, then 
\begin{equation}\label{eq:S2 intro}
\begin{cases}
\displaystyle\partial_t \phi(t_0,x_0,v_0)  + v_0 \cdot \nabla_x \phi(t_0,x_0,v_0) - 1 \geq 0 & \displaystyle \text{if} \quad \overline{u}(t_0,x_0,v_0) - \minp  \overline{u}(t_0,x_0) - \dfrac{\vert v_0\vert^2}{2} < 0, \medskip \\
\partial_t  \psi(t_0,x_0) \geq 0 , &  \text{if} \quad \aminp \overline{u}(t_0,x_0) = \left\lbrace 0 \right\rbrace .
\end{cases}
\end{equation}
%
%
%
%
\end{defi}

Let us mention that the mimimality (\textit{resp.} maximality) condition in the definition of the super- (\textit{resp.} sub-) solution arises with respect to variables $(t,x)$ only. This is consistent with the fact that there is no derivative in the velocity variable in \eqref{eq:limit}.

\begin{defi}[Solution]
Let $u_0$ be a continuous function, and $T>0$. A function $u$ is a \textbf{viscosity solution} of \eqref{eq:limit} on $(0,T) \times \R^{2n}$  with initial data $u_0$ if its upper (\textit{resp.} lower) semi-continuous envelope is a sub- (\textit{resp.} super-) solution in the sense of definitions \ref{def:subsol intro} and \ref{def:supersol intro}.
\end{defi} 


\subsection{Statement of the main results}
\label{sec:intro results}

The following theorem states a comparison principle for  viscosity  (sub/super-)solutions of the  system  \eqref{eq:limit}. This establishes uniqueness of viscosity solutions as a corollary. The proof is contained in Section \ref{sec:Comp}.

\begin{thm}[Comparison principle]\label{theo:comp}
%
%
Let $\underline{u}$ (resp. $\overline{u}$) be a viscosity sub-solution (resp. super-solution) of \eqref{eq:limit} on $(0,T)\times\R^{2n}$  with continuous initial data $\underline{u}_0 \leq \overline{u}_0$. Assume that $\underline{u}$ and $\overline{u}$ are such that 
\begin{equation} \label{eq:v2 plus borne}
 \overline{u} - \frac{\vert v\vert^2}2 \in L^\infty\left((0,T)\times\R^{2n}\right)\,, \quad   \underline{u} - \frac{\vert v\vert^2}2 \in L^\infty\left((0,T)\times\R^{2n}\right)\, .
\end{equation}
Then $\underline{u} \leq \overline{u}$ on $(0,T) \times \R^{2n}$.
\end{thm}

This result is extended for  solutions with at most quadratic  growth  in Appendix \ref{sec:unbded}. This growth condition is compatible with the kernel of the variational representation formula, as we shall see below.

In Section  \ref{sec:Conv}, we prove the convergence of the family $\left(u^\eps\right)$  towards the unique viscosity solution of \eqref{eq:limit} as $\eps\to 0$.

\begin{thm}[Convergence]\label{HJlimit}
Assume that $\u_0$  is continuous and satisfies the following property:
\begin{equation}
\label{eq:initial condition 1}
\u_0 -  \frac{\vert v\vert^2}{2} \in L^{\infty}(\R^{2n}), 
\end{equation}
Let $\u^\eps$ be the solution of \eqref{WKB1}, with the initial data $\u^\eps(0,\cdot) = u_0$. Then,  $\u^\eps$ converges locally uniformly towards $u$  as $\eps\to 0$, where $u$ is the unique viscosity solution of \eqref{eq:limit}
with initial data $\min\left(u_0,\minp u_{0} + |v|^{2}/2\right)$.
\end{thm}



Finally, we also establish the variational formulation \eqref{eq:kinHopf Lax} of the viscosity solution. 
This enables connecting the PDE (Eulerian) point of view and the trajectory (Lagrangian) point of view. The standard connection goes through the convex duality $H\leftrightarrow L$, where $L$ is the Legendre tranform of the Hamiltonian (and vice-versa), and minimizing trajectories as usually smooth curves. However, in the present study, minimizing trajectories are piecewise linear in space (and piecewise constant in velocity, accordingly). This striking behavior is connected to the concept of {\em big events} in large deviations of subexponential processes, see discussion above and \cite{jara_mallein_note}. This motivates the following definition.

\begin{defi}[Notations for the minimizing curves]
Let $\Sigma_s^t$ be the space of piecewise constant, {\em c\'a{}dl\'a{}g} functions defined over the time interval $(s,t]$ taking values in $\R^n$. For any $y\in \R^n$ and $\sigma \in \Sigma_s^t$, let $(t_i)_{1\leq i\leq N}$ be the times of discontinuity of $\sigma$ in $(s,t]$, such that 
\begin{equation*}
\sigma = \sigma_0 {\bf 1}_{(s,t_1)} +  \sum_{i =1}^{N-1} \sigma_i {\bf 1}_{[t_{i},t_{i+1})} + \sigma_{N} {\bf 1}_{[t_{N},t]} \,. 
\end{equation*}
Then, the piecewise linear curve $\gamma_\sigma$ is defined naturally as 
\begin{equation*}
\gamma_\sigma(\tau) = y + \int_s^\tau \sigma(\tau') d\tau'\, .
\end{equation*}
\end{defi}

\begin{defi}[The action of a piecewise linear curve]
The action of $\sigma$ on  $(s,t]$ is defined as follows:
\begin{equation}\label{eq:A intro}
\mA_s^t[\sigma] = \frac12 \sum_{i=1}^N |\sigma_i|^2 + \sum_{i=0}^{N} (t_{i+1} - t_{i}) {\bf 1}_{ \sigma_i\neq 0}\, ,  
\end{equation}
with the convention $t_0 = s$ and $t_{N+1} = t$.
\label{def:action}
\end{defi}
%

The action can be expressed in an informal way: each discontinuity after the initial time $s$ (not included), with a posterior non-zero velocity $\sigma_i$,  contributes to a punctuated cost $|\sigma_i|^2/2$ and a running cost of 1 per unit of time. Zero velocities come at no additional cost. Notice that the case of a constant velocity $\sigma\equiv \sigma_0$ does not involve any punctuated cost, only the running cost if $\sigma_0\neq 0$. Adding a punctuated cost would contradict the additivity of the action of curves on sub-intervals, simply because the punctuated cost would be accounted for twice on two sub-intervals. In fact, this apparently missing punctuated cost is accounted for in the initial data, which involves $|v|^2/2$, see Theorem \ref{HJlimit}.

The action can be interpreted in the following way: in the large scale regime $\left(\frac{ t}{\eps},\frac{ x}{\eps^{3/2}},\frac{ v}{\eps^{1/2}} \right)$, the non-zero velocities are indeed huge. It is unlikely to draw one such at a reorientation event, and the action precisely measures how unlikely it is: this is the first contribution in \eqref{eq:A intro} where one can recognize the rate function of the Gaussian distribution $-\eps \log M_\eps(v)$ (up to a negligible constant). Then, in order to make a significant move in the appropriate spatial scale $\mathcal{O}(\eps^{-3/2})$, this huge (non-zero) velocity of order $\mathcal{O}(\eps^{-1/2})$ should be kept for a time of order $\mathcal{O}(\eps^{-1})$. This is the second contribution   in \eqref{eq:A intro} where one can recognize the rate function of an exponential random time. However, the latter cost is restricted to non-zero velocities. In fact, zero velocities (meaning $o(\eps^{-1/2})$ in the original variables) are constantly drawn from the bulk of the Gaussian velocity distribution during repeated reorientation events, and it does not make sense to count for their cumulated time of persistence. Only the non-zero velocities matter. We also refer to \cite{jara_mallein_note} for a complementary viewpoint on this topic.

The following property is established in Section \ref{sec:variational}.

\begin{thm}[Kinetic Hopf-Lax formula]
\label{th:hopf lax intro}
Let $u_0$ be a continuous function verifying \eqref{eq:initial condition 1}. Then, the following representation formula
\begin{equation}\label{eq:kin HL th}
U(t,x,v) = \inf_{\footnotesize \begin{array}{c}
\{(y,\sigma)\in \R^n\times \Sigma_0^t:\\ \gamma_\sigma(t) = x, \sigma_N = v\}
\end{array}} \left \{  \mA_0^t[\sigma] + u_0(y,\sigma_0) \right  \} 
\end{equation}  
is the viscosity solution of \eqref{eq:limit}  with initial data $\min\left(u_0,\minp u_{0} + |v|^{2}/2\right)$. 
\end{thm}

\begin{rmq}
The final velocity $\sigma_N = v$ is assigned in \eqref{eq:kin HL th}. Hence, the last contribution in the punctuated costs $|\sigma_N|^2/2$ \eqref{eq:A intro} could be put outside the infimum in \eqref{eq:kin HL th} except if $N=0$, that is, the velocity is constant $\sigma\equiv \sigma_0$.
\end{rmq}

\subsection{Discussion}
\label{sec:intro bdd vel}

\subsubsection{Qualitative behaviour (in the long term)}

We provide in Section \ref{sec:kernel} several expressions for the minimal value of the action $\mA_0^t[\sigma]$ with prescribed endpoints $(y,\sigma_0)$ and $(x,\sigma_N)$ (we can set $y = 0$ without loss of generality by translation invariance). In fact, the minimal path contains at most one intermediate non trivial velocity, as in Proposition \ref{prop:sigma1}, so that $N \leq 2$. Consequently, the infimum in \eqref{eq:kin HL th} is attained at a minimizing trajectory.
This reduction of complexity enables computing  $\mmA_0^t(x,\sigma_0,\sigma_2) = \min \mA_0^t[\sigma]$ in the one-dimensional case $n=1$, see Proposition \ref{prop:1D}:
\begin{equation}\label{eq:kernel1D}
\mmA_0^t(x,\sigma_0,\sigma_2) = \begin{cases}
t{\bf 1}_{\sigma_0\neq 0} & \displaystyle \text{if}\; \sigma_0 = \sigma_2 = \frac xt,
\medskip\\
\displaystyle \frac{|\sigma_2|^2}{2}  + \min(L(x,\sigma_0),L(x,\sigma_2)) & \text{otherwise}.
\end{cases}
\end{equation}
The formula for $L(x,\sigma)$ is expressed below \eqref{eq:candidates all}, and its values  are depicted in Figure \ref{fig:fund}. Specializing $\sigma_0 = \sigma_2 = 0$, we find that the action between position $x$ and the origin, with both initial and final velocities at rest is:
\begin{equation}
\mmA_0^t(x,0,0) = 
\begin{cases}
\dfrac32 |x|^{2/3}  & \text{ if } |x|\leq t^{3/2},
\medskip\\
\dfrac{\vert x \vert^2}{2 t^2} + t  & \text{ if } |x|\geq t^{3/2}.
\end{cases}
\label{eq:candidates intro}
\end{equation}

\begin{figure}
\begin{center}
\includegraphics[width = .90\linewidth]{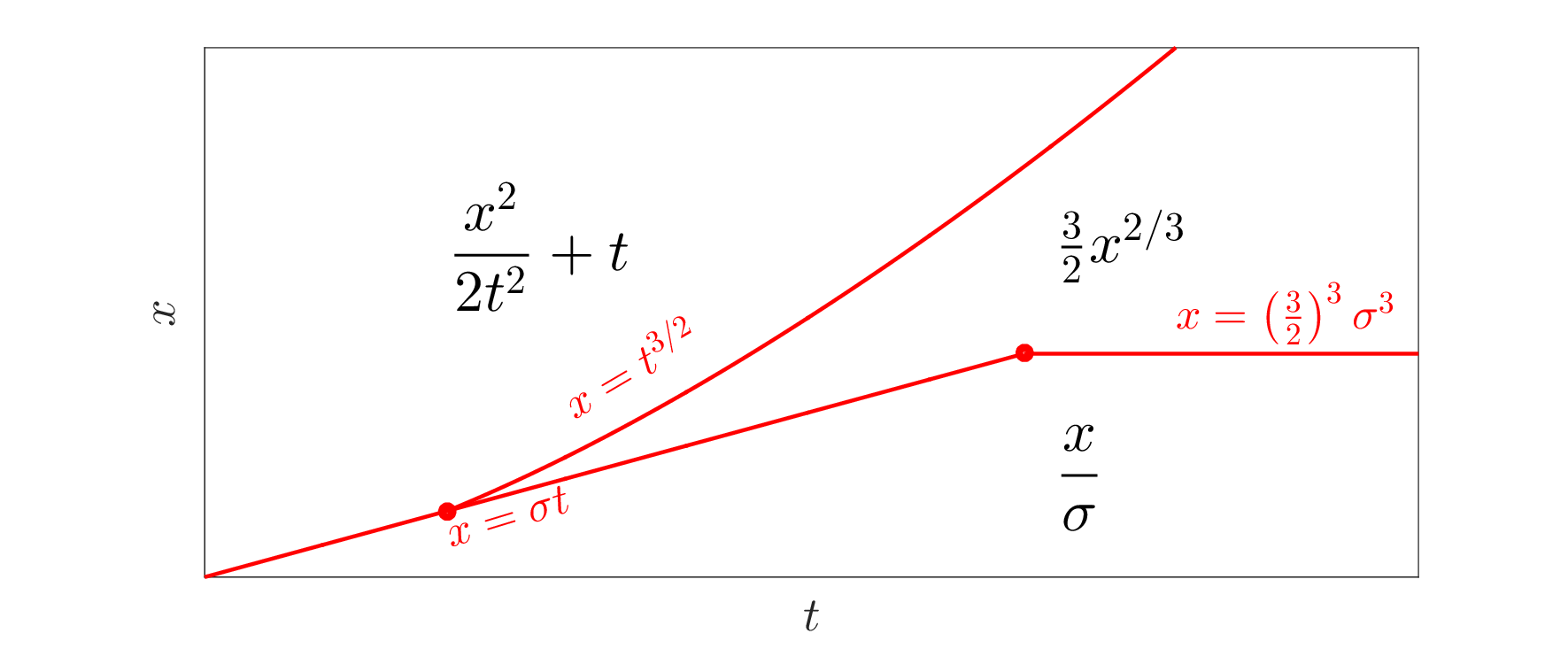}
\end{center}
\caption{Values of the contribution $L(x,\sigma)$ to the kernel of the variational formulation in the one-dimensional case \eqref{eq:kernel1D}.}\label{fig:fund}
\end{figure}

This spatial behaviour is illustrated in 
Figure \ref{fig:min_introA} in comparison with the kernel $\frac{x^2}{2t}$ associated with the Hamilton-Jacobi equation \eqref{eq:HJ-heat} coming from the heat equation with vanishing viscosity (Figure \ref{fig:min_introB}). In the latter case, the kernel is a family of parabola converging to zero as $t\to \infty$, uniformly on compact intervals. It means that, despite the rarity of finding a Brownian particle far from its origin, the small probability is not uniformly exponentially small. Contrarily, the kernel \eqref{eq:candidates intro} converges towards its envelope $\frac32 |x|^{2/3}$, which is obviously uniformly positive on closed intervals that do not contain the origin. Alternatively speaking, the probability of finding a particle far from its origin remains uniformly exponentially small in the velocity-jump process under study.

\begin{figure}
\begin{center}
\begin{subfigure}{0.45\linewidth}
\includegraphics[width = \linewidth]{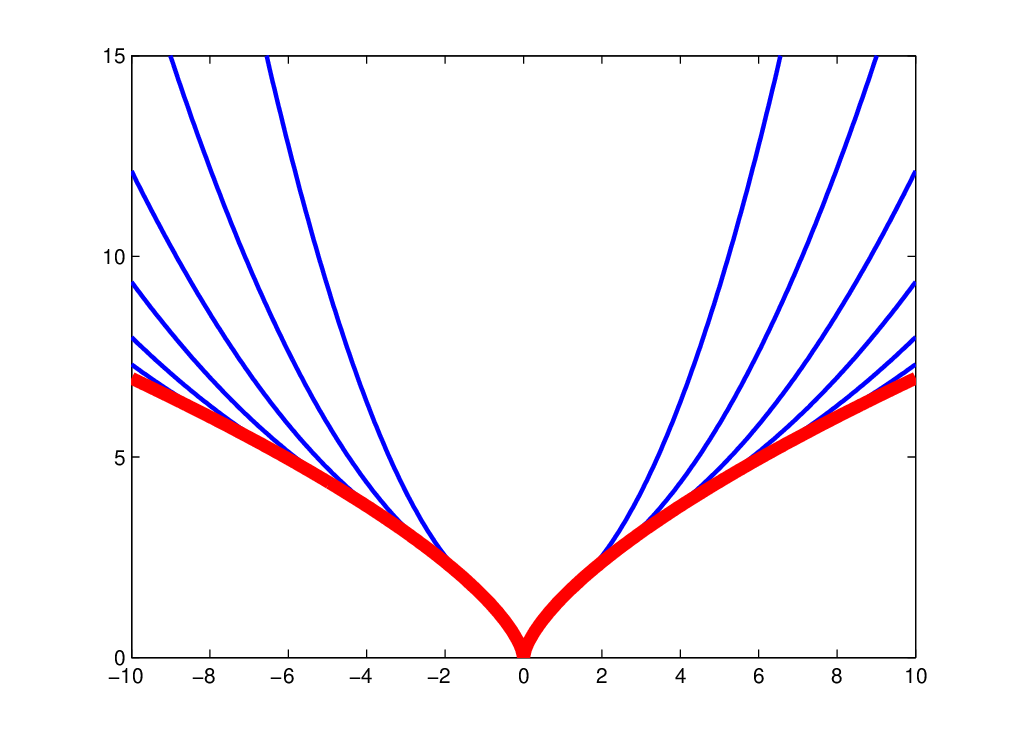} \; 
\caption{Plot of the kernel \eqref{eq:candidates intro} with initial and final velocities at rest. The red curve is the envelope as $t\to +\infty$.}
\label{fig:min_introA}
\end{subfigure}\qquad
\begin{subfigure}{0.45\linewidth}
\includegraphics[width = \linewidth]{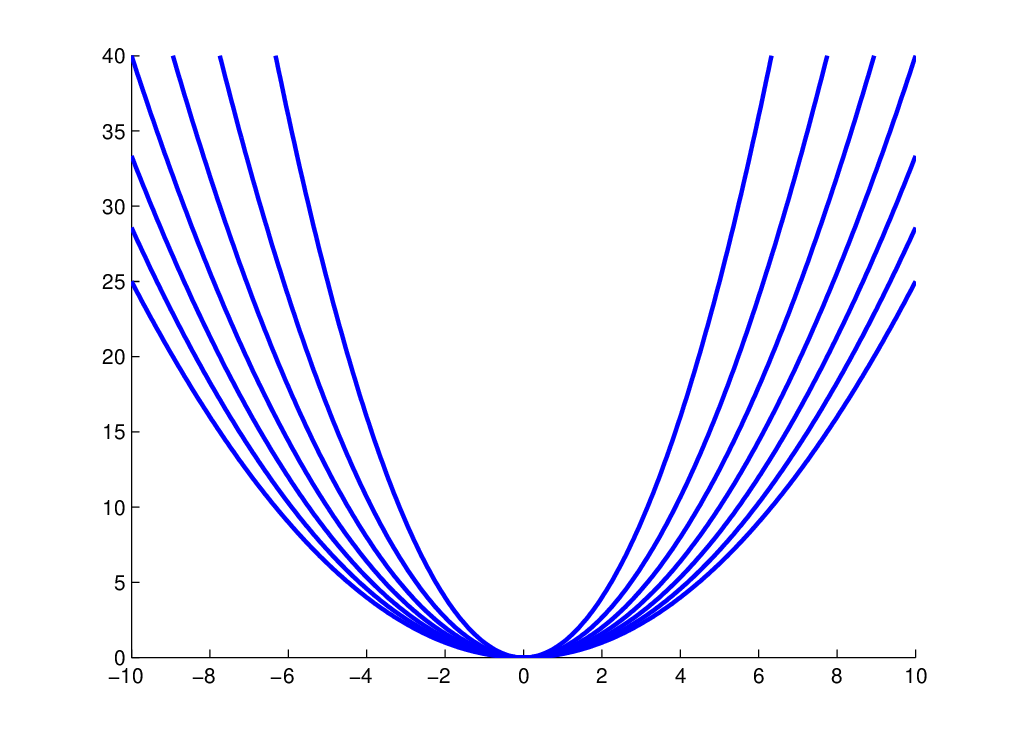}
\caption{Plot of the kernel $\frac{x^2}{2t}$ associated with the heat equation with vanishing viscosity.}
\label{fig:min_introB}
\end{subfigure}
\end{center}
\label{fig:min_intro}
\caption{Comparison of different qualitative behaviours between the velocity-jump process and the Brownian motion.}
\end{figure}

\subsubsection{Comparison with the case of bounded velocities}
\label{sec:intro compact}

Suppose that $M(v)$ is a compactly supported probability distribution, and denote $V = \supp M$.  It was shown in \cite{bouin_kinetic_2012}  that the appropriate scaling regime is different. Indeed, the velocity rescaling $\frac v{\eps^{1/2}}$ is clearly not possible in the case of bounded velocities. The scaling regime is rather $\left(\frac{ t}{\eps},\frac{ x}{\eps}\right)$, and $v$ unchanged. 

It was shown in \cite{bouin_kinetic_2012,bouin_hamilton-jacobi_2015, caillerie_large_2017} that, under such scaling regime, 
$ u^\eps (t,x,v)= -\eps \log f^\eps(t,x,v)$ converges to a function $u(t,x)$ which does not depend on $v$, and which is the viscosity solution of a standard Hamilton-Jacobi equation $\partial_t u + H(\nabla_x u) = 0$, where the Hamiltonian function $H(p)$ is defined implicitly as 
\begin{equation}\label{eq:HJbdd}
 \int_{\R^n} \frac{M(v)}{1 + H(p) - p\cdot v}\, dv = 1\, ,
 \end{equation}
 provided that this equation has an admissible solution satisfying $H(p) \leq \max_{v\in V} (p\cdot v) - 1$, and $H(p) = \max_{v\in V} (p\cdot v) - 1$ otherwise (see Caillerie \cite{caillerie_large_2017} for more details). 

The averaging process occurring in the case of bounded velocities is 
similar  to large deviation principles for slow-fast systems as in \cite{faggionato_averaging_2008,
perthame_asymmetric_2009,
kifer_large_2009,
bressloff_path_2014,
bressloff_hamiltonian_2017}, and references therein. In this case, the role of the fast variable is played by velocity, whereas the space variable is the slow one. 

Our methodology follows the Hamiltonian viewpoint. We refer to \cite{faggionato_averaging_2008,
kifer_large_2009,
bressloff_path_2014} for the dual viewpoint focusing on the trajectories of the underlying PDMP,  and to \cite{jara_mallein_note} for the case of unbounded velocities, as discussed above. We present briefly the results of the former works for the sake of comparison. Let $\Sigma$ be the finite set of possible velocities (note that their analysis is restricted to a finite number of velocities). A curve is expressed in terms of a time-varying measure $\nu = (\nu_\sigma)$ on $\Sigma$ such that $\sum_{\sigma\in \Sigma}  \nu_\sigma (t) = 1$ at any time. Then, the curve is constructed by its averaged velocity $ v= \sum \sigma \nu_\sigma$ in the following way:
\begin{equation*}
\gamma(t) = y + \int_0^t \sum_{\sigma\in \Sigma} \sigma \nu_\sigma(s)\, ds.
\end{equation*}
The action of a curve $\gamma$ is defined as $\mA_0^t[\nu] = \int_0^t L(\nu(s))\, ds$, where $L$ is given by the solution of a cell problem \cite{faggionato_averaging_2008}:
\begin{equation*}
L(\nu) =  \sup_{Z\in (\R_+^*)^\Sigma}  \left ( \sum_{(\sigma,\sigma')\in \Sigma^2} \nu_\sigma M(\sigma') \left ( 1 - \dfrac{Z_{\sigma'}}{Z_\sigma}\right )\right ).
\end{equation*}
Note that here we present the simpler case where the rate of velocity change does not depend on the velocity prior to the jump, and is space homogeneous, but more generality can be handled in \cite{faggionato_averaging_2008}. Straightforward computations yield
\begin{equation*}
L(\nu) = \left (\sum_{\sigma\in \Sigma} \nu_\sigma\right )
\left (\sum_{\sigma\in \Sigma} M(\sigma)\right ) - \left (\sum_{\sigma\in \Sigma} \left( \nu_\sigma M(\sigma) \right )^{1/2}\right )^2.
\end{equation*}
It coincides with the convex conjugate of $H$ \eqref{eq:HJbdd} by the Legendre-Fenchel transformation $H(p) = \sup\left ( p\cdot(\sum \sigma \nu_\sigma) - L(\nu)\right )$ where the supremum is taken over the set of probability measures $\nu_\sigma$ on $\Sigma$ (details omitted).  

The averaging phenomenon which is central in the case of bounded velocities \cite{faggionato_averaging_2008,bouin_kinetic_2012,bouin_hamilton-jacobi_2015, caillerie_large_2017} does not occur in the case of unbounded velocities. One immediate consequence is that the rate function $u$ depends on the velocity variable $v$. More profound consequences are the seemingly new structure of the non-local Hamilton-Jacobi problem \eqref{eq:limit}, and the singular shape of the action supported on piecewise linear curves \eqref{eq:A intro}. Alternatively speaking, the nature of the PDMP persists in the regime of large deviations.

\subsection{Accelerated fronts in reaction-transport equations} \label{sec:intro accel}

As an application of our methodology, we investigate quantitatively front acceleration in reaction-transport equations in Section \ref{sec:acc}. We focus on \eqref{eq:main} with an additional monostable reaction term:
\begin{equation}\label{eq:kinreac}
\partial_t f(t,x,v)  + v \cdot \nabla_x f(t,x,v) =\left( M(v) \rho(t,x)  - f(t,x,v)  \right) + r \rho(t,x)  \left( M(v) - f(t,x,v)  \right)\,.
\end{equation}
This models a population of individuals that change velocity at rate one, pick up a random new velocity following a Gaussian distribution, and divide at rate $r>0$. Moreover, new particles pick up their initial velocity from the same Gaussian distribution. Saturation occurs when the spatial density $\rho(t,x) = \int_{\R^n} f(t,x,v') dv'$ gets too large, so that the space homogeneous problem admits the pair of equilibria zero (trivial) and $M(v)$.  

This  model has been studied in \cite{hadeler_reaction_1999,
schwetlick_travelling_2000,
cuesta_traveling_2012,
bouin_propagation_2015} in the case of bounded velocities, and further in \cite{bouin_propagation_2015} in the case of possible unbounded velocities. 
Equation \eqref{eq:kinreac} can be viewed as a kinetic version of the celebrated Fisher-KPP equation, 
\begin{equation}\label{eq:FKPP}
\partial_t \rho(t,x) - \Delta \rho(t,x) = r \rho(t,x) (1 - \rho(t,x))\, . 
\end{equation}
There exists a true link via the diffusion limit when $t$ and $x$ are scaled in the parabolic regime (provided that the rate of division $r$ is scaled too), see  \cite{cuesta_traveling_2012}. However, we point out that the parabolic scaling is not compatible with front tracking, so we need to follow a direct approach as in \cite{bouin_propagation_2015}. 

In the case of bounded velocities,  there exist traveling waves with constant speed of propagation \cite{schwetlick_travelling_2000,
cuesta_traveling_2012,
bouin_propagation_2015}. Moreover, any solution to the Cauchy problem with sufficiently decaying initial data spreads with the minimal speed, just as for Fisher-KPP \cite{fisher_wave_1937,
kolmogorov_etude_1937,
hadeler_travelling_1975,
aronson_multidimensional_1978}. 

In the case of unbounded velocities, and more precisely for a Gaussian velocity distribution, it was established in \cite{bouin_propagation_2015} that solutions to \eqref{eq:kinreac}, in the one-dimensional case $n=1$, behave in the long-time asymptotics as accelerating fronts due to the (rare) occurrence of high velocities that send particles far from the bulk. Furthermore, the location of the front is of the order of $t^{3/2}$, in accordance with the scaling limit of the linear problem performed in  \eqref{eq:main}. The  location of the front $X(t)$ (such that $\rho(t,X(t)) = 1/2)$ was  determined in \cite{bouin_propagation_2015} via the construction of sub- and super-solutions, with some room in between. More precisely, it was estimated that
\begin{equation} \label{eq:bounds expansion}
\left( \frac{r}{r+2} \right)^{3/2}  \leq\dfrac{X(t)}{t^{3/2}} \leq \sqrt{2r}  \, , \end{equation}
in a weak sense (see \cite[Theorem 1.11]{bouin_propagation_2015} for details). 

Front acceleration has been reported in a number of works in the past decade.   Cabr\'e and Roquejoffre studied the Fisher-KPP equation \eqref{eq:FKPP} where the diffusion operator is replaced with a fractional diffusion operator \cite{cabre_propagation_2009,
cabre_influence_2013}, 
\begin{equation}\label{eq:cabre}  \partial_t \rho(t,x) + (-\Delta )^{\alpha}\rho(t,x) = r \rho(t,x) (1 - \rho(t,x))\, , \end{equation}
for some exponent $\alpha\in (0,1)$.
They described quantitatively the acceleration of the front, which occurs at exponential rate, namely $X(t) = \exp(rt/(n+2\alpha))$ in a weak sense. This seminal work was continued in \cite{coulon_transition_2012,
roquejoffre_gradient_2017}. 

Garnier has investigated integro-differential equations, where the spreading operator is given by the convolution with a fat-tailed kernel \cite{garnier_accelerating_2011}, 
\begin{equation}\label{eq:garnier}  \partial_t \rho(t,x) + \left( - \int_{\R} J(x-y) \rho(t,y)\, dy + \rho(t,x)\right) = r \rho(t,x) (1 - \rho(t,x))\, . \end{equation}
Here, fat-tailed means that the kernel $J$ decays slower than exponentially. There, the level lines of the solution spread super linearly, depending on the decay of the convolution kernel $J$.

Recently, spreading in the so-called cane toads equation has been  studied intensively for unravelling dispersal evolution at the edge of an invasion front,   
\begin{equation}\label{eq:canetoads}  
\partial_t f(t,x,\theta) - \theta \partial^2_{x} f(t,x,\theta) - \partial^2_{\theta} f(t,x,\theta) = r f(t,x,\theta) (1 - \rho(t,x))\, , \quad \rho(t,x) = \int f(t,x,\theta')\, d\theta'\, , 
\end{equation}
When the variable $\theta$ is unbounded, accelerated propagation was conjectured in \cite{bouin_invasion_2012}, then was established independently by Berestycki, Mouhot and Raoul \cite{berestycki_existence_2015}, and by the first author, Henderson and Ryzhik \cite{bouin_super-linear_2017}. There is a formal analogy between  \eqref{eq:canetoads}, and our problem \eqref{eq:kinreac}. Indeed, acceleration also happens due to the influence of a microscopic variable $\theta$, which plays a similar role as the velocity variable in this paper. This is another example of a nonlinear acceleration phenomena appearing in a structured model.

Here, we aim to apply the powerful methodology of the approximation of geometric optics for reaction-diffusion equations \cite{freidlin_functional_1985,
freidlin_geometric_1986, evans_pde_1989, barles_wavefront_1990}. Recently, this method has been applied successfully to the case of the fractional reaction-diffusion equation \eqref{eq:cabre} by M\'el\'eard and Mirrahimi \cite{meleard_singular_2015}, to the integro-differential equation \eqref{eq:garnier} by the first author, Garnier, Henderson and Patout \cite{bouin_thin_2018}, to the cane toads equation \eqref{eq:canetoads} by the second author, Henderson, Mirrahimi and Turanova \cite{calvez_non-local_2022}, and also to the reaction-transport equation \eqref{eq:kinreac} with bounded velocities \cite{bouin_hamilton-jacobi_2015,
bouin_spreading_2019}. 
It amounts to perform the right scaling of variables, here $\left(\frac{ t}{\eps},\frac{ x}{\eps^{3/2}},\frac{ v}{\eps^{1/2}} \right)$, and to derive an equation for $u^\eps = -\eps \log f^\eps$ in order to track the level sets of the density, and to localize the place where the population is emerging.

\begin{figure}
\begin{center}
\includegraphics[width = .90\linewidth]{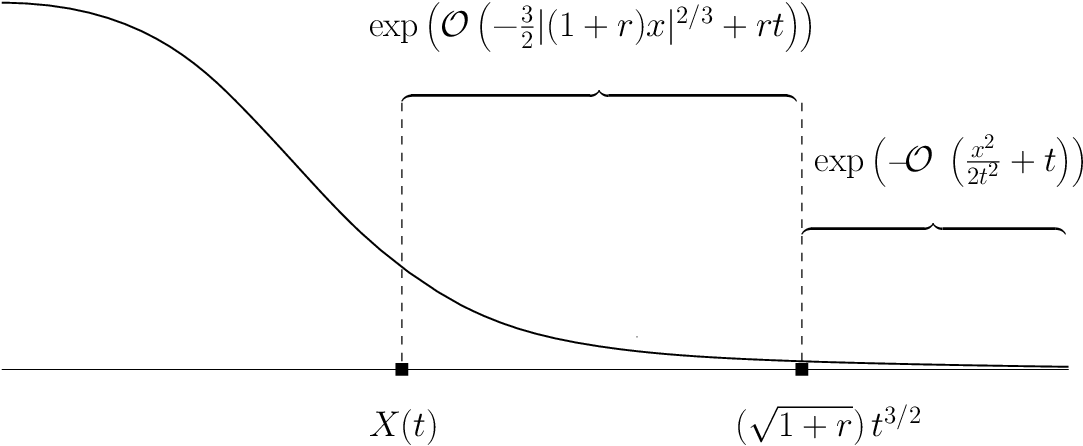}
\end{center}
\caption{Loose description of the dynamics of front acceleration based on the computations of Section \ref{sec:acc}.}\label{fig:front dynamics}
\end{figure}

We restrict to  the one-dimensional case $n=1$ for simplicity.
After identification of the limit problem, which is an obstacle version of \eqref{eq:limit} due to the saturation term, we are able to refine the estimate \eqref{eq:bounds expansion} to get that the front  is located around
\begin{equation*}
X(t) = \left( \frac{\left( (2/3) r \right)^{3/2}}{1+r} \right)  t^{{3}/{2} }\, , 
\end{equation*}
in a weak sense. To decipher the dynamics of acceleration, we show in Figure \ref{fig:front dynamics} a cartoon of the tails of the spatial density resulting from the approximation of geometric optics performed in Section \ref{sec:acc}. There is a first zone far ahead where the density is uniformly exponentially small, independently of $r$. It is followed by a region where the density is approximately of separable variables: a sub-exponential anomalous spatial profile multiplied by a growing exponential. It is where the front is actually emerging at $X(t)$ (roughly).

\section{The  comparison principle}\label{sec:Comp}

This section is devoted to the proof of Theorem \ref{theo:comp}. We  perform a classical  doubling of variables argument in $(t,x)$. However, much attention has to be paid to the velocity variable. This is the main concern of this proof. In particular, the velocity variable is not doubled, which is consistent with the fact that there is no gradient with respect to velocity in the limit system \eqref{eq:limit}. 


%
We define some auxiliary functions as follows,
\begin{equation*}
\underline{b}(t,x,v) =  \underline{u}(t,x,v) - \frac{\vert v \vert^2}{2}, \qquad \overline{b}(t,x,v) =  \overline{u}(t,x,v) - \frac{\vert v \vert^2}{2}.
\end{equation*}
They are bounded by assumption. We denote $B = \max (\|\underline{b}\|_\infty,\|\overline{b}\|_\infty)$. As a consequence, there exists $R_0$ such that for all $(t,x)$, $\aminp \overline{u} (t,x) \subset \mB(0,R_0)$ the ball of radius $R_0$ around the origin.

We introduce the parameter $\kappa \in (0,1)$ for the sake of comparing $\kappa \underline{b}$ and $\overline{b}$.
We also introduce $\eps>0$, $\alpha>0$, $R>R_0$, and some additional $\delta>0$ to be suitably chosen below, depending on $\kappa$, $B$ and $R$.

To perform a doubling of variables argument, we define another auxiliary function with twice the number of variables, except for the velocity, as follows
\begin{multline}\label{eq:chiapp}
\wtildechi(t,x,s,y,v) = \kappa\underline{b}(t,x,v) - \overline{b}(s,y,v) -  \frac\delta2 \left( |x|^2 + \vert y \vert^2 \right) - \alpha\left( \dfrac1{T- t} + \dfrac1{T- s}\right) 
\\  - \frac1{2\eps} \left(|t -s|^2 + |x-y|^2\right) - \left(  |v|^2 - R^2\right)_+ \, .
\end{multline}
Let $(\tt,\tx,\ts,\ty,\tv)$ which realizes the maximum of $\wtildechi$ (we omit the dependency with respect to the parameters for the sake of conciseness). It exists by upper semi-continuity and confinement. In the same spirit, we introduce the following auxiliary function
\begin{multline*}
\overlinechi(t,x,s,y) = \kappa \left (\minp \underline{u}(t,x)\right ) - \minp \overline{u}(s,y) -  \frac\delta2 \left( |x|^2 + \vert y \vert^2 \right) -  \alpha \left( \dfrac1{T- t} + \dfrac1{T- s}\right)
\\   - \frac1{2\eps} \left(|t -s|^2 + |x-y|^2\right)\, .
\end{multline*}
Finally, we define the  maximum values:
\begin{equation*} 
\omega = \max_{((0,T) \times \R^n)^2}\overlinechi \,, \quad \Omega = \max_{((0,T) \times \R^n)^2 \times \R^n} \wtildechi \,.
\end{equation*}

\begin{lem}\label{lem:Oomega}
We have $\omega \leq \Omega$.
\end{lem}

\begin{proof}
Let $(\bt,\bs,\bx,\by)$ 
be a maximum point of  $\overlinechi$. Let $\hat v \in \aminp \overline{u} (\bs,\by) \subset \mB(0,R)$. The following sequence of inequalities holds true,
\begin{align*}
\Omega & \geq  \wtildechi(\bt,\bs,\bx,\by,\hat v) \\
&=
\ds \kappa\underline{b}(\bt,\bx,\hat v) - \overline{b}(\bs,\by,\hat v)
\\
&\qquad \qquad \qquad \qquad-  \frac\delta2 \left( |\bx|^2 + \vert \by \vert^2 \right) - \alpha\left( \dfrac1{T- \bt} + \dfrac1{T- \bs}\right) - \frac1{2\eps} \left(|\bt -\bs|^2 + |\bx-\by|^2\right)\, ,\\
&=
\ds \kappa\underline{u}(\bt,\bx,\hat v) - \overline{u}(\bs,\by,\hat v) + (1-\kappa)\frac{\vert \hat v \vert^2}{2} 
\\
&\qquad \qquad \qquad \qquad-  \frac\delta2 \left( |\bx|^2 + \vert \by \vert^2 \right) - \alpha\left( \dfrac1{T- \bt} + \dfrac1{T- \bs}\right) - \frac1{2\eps} \left(|\bt -\bs|^2 + |\bx-\by|^2\right)\, ,\\
& \geq  \ds \kappa \minp \underline{u}(\bt,\bx) - \minp \overline{u}(\bs,\by)
\\
&\qquad \qquad \qquad \qquad -  \frac\delta2 \left( |\bx|^2 + \vert \by \vert^2 \right) - \alpha\left( \dfrac1{T- \bt} + \dfrac1{T- \bs}\right) - \frac1{2\eps} \left(|\bt -\bs|^2 + |\bx-\by|^2\right)\, ,\\
& = \overlinechi(\bt,\bs,\bx,\by) =   \omega \, .
\end{align*}
\end{proof}

%

\begin{lem}\label{lem:limit t0}
The point $(\tx,\ty,\tv)$ satisfies the following   estimates:
\begin{equation*}
 \delta\max\left ( |\tx| ,  |\ty|\right )  \leq 4B^{1/2} \delta^{1/2}\quad , \quad |\tilde v|^2\leq 4B + R^2 \,.
\end{equation*}
Moreover, the following limit holds true,
\[ \lim_{\eps\to 0} |\tt - \ts | + |\tx - \ty|  = 0\, . \]
\end{lem}
\begin{proof}
The evaluation $\wtildechi(0,0,0,0,0) \leq \wtildechi(\tt,\tx,\ts,\ty,\tilde v)$ yields
\begin{equation*}
\dfrac\delta2\left ( |\tx|^2 + |\ty|^2\right ) + \alpha \left ( \dfrac1{T-\tt}+  \dfrac1{T-\ts}\right )  +  \frac1{2\eps} \left(|\tt -\ts|^2 + |\tx-\ty|^2\right) +   \left(  |\tilde v|^2 - R^2\right)_+ 
\leq  4\LambdaB   \, .
\end{equation*}
We deduce the following estimates:
$|\tx| ,  |\ty|  \leq 4B^{1/2} \delta^{-1/2}$, and $|\tt - \ts|,\, |\tx - \ty| \leq C\eps^{1/2}$.

The inequality on $|\tilde{v}|^{2}$ is obvious.
\end{proof}


We continue with the comparison argument. On the one hand, assume that $(0,0)$ is an accumulation point of $(\tt,\ts)$ as $\eps \to 0$. Let $(x_0,x_0,v_0)$ be an associated accumulation point of $(\tx,\ty,\tilde v)$. For any $(t,x,v)$, we surely have $\wtildechi(t,x,t,x,v) \leq \wtildechi(\tt,\tx,\ts,\ty,\tilde v)$, hence 
\begin{equation*}
\kappa\underline{b}(t,x,v) - \overline{b}(t,x,v) -  \delta   |x|^2   - \frac{2\alpha}{T- t} - \left(  |v|^2 - R^2\right)_+  \leq  \kappa\underline{b}(\tt,\tx,\tilde v) - \overline{b}(\ts,\ty,\tilde v).
\end{equation*}
Passing to the limit along a subsequence $\eps_n \to 0$, we have by upper semi-continuity:
\begin{align*}
\kappa\underline{b}(t,x,v) - \overline{b}(t,x,v) -   \delta  |x|^2   - \frac{2\alpha}{T- t} - \left(  |v|^2 - R^2\right)_+  &\leq  \kappa\underline{b}(0+,x_0, v_0) - \overline{b}(0+,x_0,v_0) 
\\
& \leq \sup  \left (  \kappa\underline{b}_0 - \overline{b}_0 \right ) \,.
\end{align*}
Using the boundedness of $\underline{b}$, $\kappa\underline{b}_0 - \overline{b}_0$ converges uniformly towards $\underline{b}_0 - \overline{b}_0$ as $\kappa\to 1$. Finally, passing to the limit $\kappa\to 1, \delta\to 0, \alpha \to 0$, and $R\to +\infty$, we find 
\begin{align*}
\underline{b}(t,x,v) - \overline{b}(t,x,v)  \leq \sup \left (  \underline{b}_0 - \overline{b}_0 \right )  & = \sup \left (  \underline{u}_0 - \overline{u}_0\right )  \leq 0.
\end{align*}



On the other hand,  assume that $(0,0)$ is {\em not} an accumulation point of $(\tt,\ts)$ as $\eps \to 0$. Then, we distinguish between two cases:

\medskip

\noindent{\bf \# Case 1:  $\overline{b}(\ts,\ty,\tilde v) < \underset{w \in \R^n}{\min}\left(\overline{b}(\ts,\ty,w) + \frac{\vert w \vert^2}{2} \right)  = \minp \overline{u}(\ts,\ty)$.} 

\medskip

In this case, $\partial_t \underline{b} + v\cdot \nabla_x \underline{b} - 1 \leq 0$ and $\partial_t \overline{b} + v\cdot \nabla_x \overline{b} - 1 \geq 0$ in the viscosity sense.
We first use the test function
\begin{multline}\label{def:phi2}
\phi_2(s,y,v) = \kappa\underline{b}(\tt,\tx,v) -  \frac\delta2 \left( |\tx|^2 + \vert y \vert^2 \right) - \alpha\left( \dfrac1{T- \tt} + \dfrac1{T- s}\right) \\ - \frac1{2\eps} \left(|\tt-s|^2 + |\tx-y|^2\right) - \left(  |v|^2 - R^2\right)_+,
\end{multline}
associated to the supersolution $\overline{b}$ at the point $(\ts,\ty,\tilde v)$. Notice that the condition $\ts>0$ is verified for $\eps$ small enough. By using  Definition  \ref{def:supersol} of a super-solution, this yields
\begin{equation}
-\dfrac\alpha{(T - \ts)^2}  + \frac1\eps (\tt -\ts) + \tilde  v \cdot \left( - \delta \ty - \frac1\eps ( \ty- \tx)\right) - 1 \geq 0\, . \label{eq:chain rule 1.1}
\end{equation}
On the other hand, using the test function
\begin{multline}\label{def:phi1}
\phi_1(t,x,v) = \overline{b}(\ts,\ty,v) + \frac\delta2 \left( |x|^2 + \vert \ty \vert^2 \right) + \alpha \left( \dfrac1{T- \ts} + \dfrac1{T- t}\right) \\  + \frac1{2\eps} \left(|t-\ts|^2 + |x-\ty|^2\right) + \left(  |v |^2 - R^2\right)_+\, ,
\end{multline}
associated to the subsolution $\kappa \underline{b}$ at the point $(\tt,\tx,\tilde v)$, we obtain
\begin{equation}
\dfrac\alpha{(T - \tt)^2} + \frac1\eps (\tt - \ts) +\tilde  v \cdot \left(\delta \tx +  \frac1\eps ( \tx - \ty) \right) - \kappa \leq 0\,,
\label{eq:chain rule 2.1}
\end{equation}
by using Definition  \ref{def:subsol} of a sub-solution. By substracting \eqref{eq:chain rule 2.1} to \eqref{eq:chain rule 1.1}, we obtain
\begin{equation*}
 - 8B^{1/2}( 4B + R^2 )^{1/2} \delta^{1/2} + (1-\kappa)  \leq \alpha \left( \dfrac1{(T - \tt)^2} + \dfrac1{(T - \ts)^2} \right) + \delta \tilde  v \cdot \left( \tx  + \ty  \right) + (1-\kappa) \leq 0,
\end{equation*}
where we have used  Lemma \ref{lem:limit t0} in order to bound $\delta \tilde  v \cdot \left( \tx  + \ty  \right)$ from below. By choosing $\delta$ sufficiently small as compared to $1 - \kappa$, $B$ and $R$, we obtain a contradiction. 
%
%
%
%

\medskip

\noindent{\bf \# Case 2:  $\overline{b}(\ts,\ty,\tilde v) \geq \underset{w \in \R^n}{\min}\left(\overline{b}(\ts,\ty,w) + \frac{\vert w \vert^2}{2} \right) = \minp \overline{u}(\ts,\ty)$.} 

\medskip

Let $\tw \in \aminp \overline{u}(\ts,\ty)\subset \mB(0,R)$ such that $\overline{b}(\ts,\ty,\tilde v) \geq \overline{b}(\ts,\ty,\tw) + \frac{\vert \tw \vert^2}{2}$. Since the constraint \eqref{eq:parabolic constraint} is satisfied for the sub-solution $\underline{u}$, 
we have
\begin{multline*}
\wtildechi(\tt,\tx,\ts,\ty,\tilde v) \leq \kappa \left( \underline{b}(\tt,\tx,\tw) + \frac{\vert \tw \vert^2}{2} \right) - \left(\overline{b}(\ts,\ty,\tw) + \frac{\vert \tw \vert^2}{2} \right)  \\ -  \frac\delta2 \left( |\tx|^2 + \vert \ty \vert^2 \right) - \alpha\left( \dfrac1{T- \tt} + \dfrac1{T- \ts}\right) \\- \frac1{2\eps} \left(|\tt -\ts|^2 + |\tx-\ty|^2\right) - (|\tilde  v|^2 - R^2)_+,
\end{multline*}
which can be reformulated as
\begin{equation*}
\wtildechi(\tt,\tx,\ts,\ty,\tilde v) \leq \wtildechi(\tt,\tx,\ts,\ty,\tw) + \frac12 (\kappa -1)  \vert \tw \vert^2 + (|\tw|^2 - R^2)_+ -  (|\tilde  v|^2 - R^2)_+.
\end{equation*}
Since $\kappa < 1$ and $\tw \in \mB(0,R)$ we have both that $\tilde v \in \mB(0,R)$, and that $\tw = 0$, by the maximality of $(\tt,\tx,\ts,\ty,\tilde v)$. 
As a consequence, we find that $\aminp \overline{u}(\ts,\ty)$ is reduced to the singleton $\lbrace 0 \rbrace$ and that $\wtildechi(\tt,\tx,\ts,\ty,\tilde v) = \wtildechi(\tt,\tx,\ts,\ty,0)$.

Next, we find a point that maximizes $\overlinechi$. By Lemma \ref{lem:Oomega}, we find:
\begin{align*}
\omega \leq \Omega = \wtildechi(\tt,\tx,\ts,\ty,0) &= \kappa\underline{b}(\tt,\tx,0) - \overline{b}(\ts,\ty,0) -  \frac\delta2 \left( |\tx|^2 + \vert \ty \vert^2 \right)\\& \qquad - \alpha \left( \dfrac1{T-\tt} + \dfrac1{T- \ts}\right) - \frac1{2\eps} \left(|\tt -\ts|^2 + |\tx-\ty|^2\right) \\
&= \kappa\minp \underline{u}(\tt,\tx) - \minp \overline{u}(\ts,\ty) -  \frac\delta2 \left( |\tx|^2 + \vert \ty \vert^2 \right)\\& \qquad - \alpha\left( \dfrac1{T-\tt} + \dfrac1{T- \ts}\right) - \frac1{2\eps} \left(|\tt -\ts|^2 + |\tx-\ty|^2\right) \\
&= \overlinechi(\tt,\tx,\ts,\ty) \leq \omega.
\end{align*}
%
%
Therefore, $(\tt,\tx,\ts,\ty)$ maximizes $\overlinechi$. We are now ready to perform the last step. 
%
%
%
First, we introduce  after \eqref{def:phi1}
\[\psi_1(t,x) = \phi_1(t,x,0) = \overline{u}(\ts, \ty,0) + \frac\delta2 \left( |x|^2 + \vert \ty \vert^2 \right) + \alpha \left( \dfrac1{T- \ts} + \dfrac1{T- t}\right) \\  + \frac1{2\eps} \left(|t-\ts|^2 + |x-\ty|^2\right)  
\,,\]
such that $\kappa \minp \underline{u} - \psi_1$ has a maximum at $(\tt,\tx)$. We can thus use the subsolution criterion and write
\begin{equation}
 \dfrac\alpha{(T - \tt)^2} + \frac1\eps (\tt - \ts)  \leq 0\, .
\label{eq:chain rule 2}
\end{equation}
Then, we introduce  after \eqref{def:phi2}:
\[\psi_2(s,y) =  \phi_2(s,y,0) =  \kappa\underline{u}(\tt,\tx,0) -  \frac\delta2 \left( |\tx|^2 + \vert y \vert^2 \right) - \alpha\left( \dfrac1{T- \tt} + \dfrac1{T- s}\right) - \frac1{2\eps} \left(|\tt-s|^2 + |\tx-y|^2\right) \,,\] 
such that $ \minp\overline{u} - \psi_2$ has a minimum at $(\ts,\ty)$.
Notice that the condition $\ts>0$ is verified for $\eps$ small enough, and that precisely $\aminp \overline{u}(\ts,\ty)=\lbrace 0 \rbrace$. The second criterion in \eqref{eq:S2 intro} can be applied, and we obtain,
\begin{equation}
-  \dfrac\alpha{(T- \ts)^2} - \frac1\eps (\ts -\tt) \geq 0\, . \label{eq:chain rule 1}
\end{equation}

By substracting \eqref{eq:chain rule 2} to \eqref{eq:chain rule 1}, we obtain a contradiction as $\alpha >0$, concluding the proof of the comparison principle.

\section{Convergence of $\u^\eps$ when $\eps \to 0$.}\label{sec:Conv}

This Section is devoted to the proof of Theorem \ref{HJlimit}. We follow the method of half-relaxed limits of Barles and Perthame \cite{barles_discontinuous_1987}. We define accordingly the upper semi-continuous limit $u^*$ and the lower semi-continuous limit $u_*$ as follows:
\begin{equation}\label{eq:semilimit}
u^*(t,x,v) = \limsup_{\footnotesize
\begin{array}{c} \eps \to 0\\
(s,y,w)\to (t,x,v) 
\end{array}} u^\eps(s,y,w)\, ,\quad 
u_*(t,x,v) = \liminf_{\footnotesize
\begin{array}{c} \eps \to 0\\
(s,y,w)\to (t,x,v) 
\end{array}} u^\eps(s,y,w).
\end{equation}
We establish below that the former is a viscosity sub-solution, and the latter is a viscosity super-solution. Then, the comparison result obtained in the previous section guarantees that $u^* \leq u_*$. Hence $u_* = u^*$, and we get  convergence of $u^\eps$ towards the viscosity solution $u$.

We define the auxiliary function $b^\eps = u^\eps - \vert v \vert^2/2$ as in the previous section. It solves the following equation:
\begin{equation*}
\partial_t b^{\eps}(t,x,v) + v \cdot \nabla_x b^{\eps}(t,x,v) = 1 - \int_{\R^n} M_\eps (v') \exp \left( \frac{b^{\eps}(t,x,v) - b^{\eps}(t,x,v') }{\e} \right) dv'.
\end{equation*}
It results from the maximum principle that $b^\eps$ is uniformly bounded for all $t>0$ provided that the initial condition $b_0$ is bounded which is the assumption \eqref{eq:initial condition 1}. Hence, $u^*$ and $u_*$ verify the hypotheses of Theorem \ref{theo:comp}. 

We split the proof in two steps.  

\medskip
\noindent{\bf \# Step 1: $u^*$ is a viscosity sub-solution.}
\medskip

We begin with the parabolic constraint \eqref{eq:parabolic constraint}. Let $(t_0,x_0,v_0)$, with $t_0>0$, and let $w\in \R^n$. 

The Duhamel formula is expressed as follows, for $0< \tau < t$:
\begin{equation}\label{eq:dufo}
f^\eps(t,x,v) = e^{-\tau/\eps}  f^\eps (t-\tau, x - \tau v, v) + M_\eps(v)\int_0^\tau e^{-s/\eps}   \rho^\eps (t-s,x-sv)\, ds\, . 
\end{equation}
By omitting the first contribution in the right-hand-side, we deduce that 
\begin{equation*}
u^\eps(t,x,v) \leq  - \eps \log\left ( \int_0^\tau e^{-s/\eps} \rho^\eps(t-s,x-sv)\, ds
\right ) +  \frac{|v|^2}{2}  + \eps \log \left ( (2 \pi\eps)^{n/2} \right )\, .
\end{equation*}
Choose any $w\in \R^n$, and let $\delta>0$. By definition of the upper semi-limit, there exist $r>0$ and $\eps_{0}>0$ such that for all $\eps<\eps_{0}$, $u^\eps(t,x,w)\leq u^*(t_0,x_0,w) + \delta$ in a neighbourhood of radius $r$ of $(t_0,x_0,w)$. For $(t,x,v)$ in the neighbourhood of $(t_0,x_0,v_0)$, we have
\begin{align*}
\rho^\eps(t-s,x-sv) &= \int \exp\left ( -\dfrac{u^\eps(t-s,x-sv,v')}\eps \right )\, dv'\\
& \geq   |\mB(w,r)| \exp\left ( -\dfrac{u^*(t_0,x_0,w) + \delta}\eps  \right )\,,
\end{align*}
provided that $|t-s - t_0|< r$ and $|x - s v - x_0|< r$. This holds true if $\max(s,|t-t_0|,s|v|,|x-x_0|) < r/2$. Thus, we find
\begin{equation*}
u^\eps(t,x,v) \leq  \tau - \eps \log\left (  \int_0^{ \min (\tau, r/2,  (r/(2|v|)) )} |\mB(w,r)| \exp\left ( -\frac{u^*(t_0,x_0,w) + \delta}{\eps}  \right ) ds \right )    + \dfrac{|v|^2}{2} + \mathcal O(\eps \log \eps)\, .
\end{equation*}
Taking the $\limsup$ of $u^\eps$ as $\eps \to 0$ and $(t,x,v)\to (t_0,x_0,v_0)$, then letting $\tau,\delta \to 0$, we find that the following inequality holds true for all $w$, 
\begin{equation}\label{eq:constraint u*}
u^*(t_0,x_0,v_0) \leq u^*(t_0,x_0,w) +  \frac{|v_0|^2}{2} \, ,
\end{equation}
hence the constraint is satisfied. 

Then, we consider a pair of $\mathcal{C}^1$ test functions $(\phi,\psi)$   as in Definition \ref{def:subsol intro}, namely $\u^* - \phi$ and $ \minp \u^*  - \psi$ have a local maximum with respect to variables $(t,x)$ at the point $(t_0,x_0,v_0)$, with $t_0 > 0$. Note that the maximum can be supposed global and strict without loss of generality. The following inequalities must be checked:
\begin{enumerate}[(i)]
\item\label{(i)} $\partial_t \phi(t_0,x_0,v_0) + v_0 \cdot \nabla_x \phi(t_0,x_0,v_0) - 1 \leq 0\,,$  
\item\label{(ii)} $\partial_t \psi(t_0,x_0) \leq 0\, .$
\end{enumerate}
On the one hand, the first condition \eqref{(i)} is immediate as the following inequality is always satisfied:
\begin{equation*}
\forall \eps > 0, \qquad \forall (t,x,v) \in \R_+^* \times \R^{2n}, \qquad \partial_t \u^\eps(t,x,v)  + v \cdot \nabla_x  \u^\eps(t,x,v)  \leq  1 \,.
\end{equation*}

On the other hand, we observe that $\u^*$ attains its minimum at $v=0$ due to the constraint \eqref{eq:constraint u*}. Let $\delta >0$. There exist  $(t_\eps,x_\eps,v_\eps) \to (t_0,x_0,0)$ such that $u^\eps - \psi - (1+\delta)\vert v \vert^2/2$ has a local maximum at $(t_\eps,x_\eps,v_\eps)$, which is global with respect to velocity. In particular, for any $w$ we have,
\begin{equation*}
u^\eps(t_\eps, x_\eps, v_\eps) - (1+\delta)\frac{\vert v_\eps \vert^2}2 \geq u^\eps(t_\eps, x_\eps, w) - (1+\delta)\frac{\vert w \vert^2}2.
\end{equation*}
Thus, at $(t_\eps,x_\eps,v_\eps)$ we find, 
\begin{equation*}
\partial_t \psi(t_\eps,x_\eps) + v_\eps \cdot \nabla_x  \psi(t_\eps,x_\eps) \leq 1 - \frac{1}{(2 \pi\eps)^{n/2} }
\int_{\R^n} \exp \left( \frac{- (1+\delta)\vert v' \vert^2}{2\e} \right) dv' = 1 - \frac{1}{(1+\delta)^{n/2}}\, .
\end{equation*}
Passing to the limit $\eps \to 0$, then $\delta \to 0$, we get the second condition \eqref{(ii)}.

Finally, it remains to check the criterion for the initial data. It can be deduced from the Duhamel formulation \eqref{eq:dufo}:
\begin{equation*}
u^\eps(\tau,x,v) \leq \min \left ( \tau + u_0(x - \tau v, v) , -\eps \log \left (  M_\eps(v)\int_0^\tau e^{-s/\eps}   \rho^\eps (\tau-s,x-sv)\, ds\right ) \right )
\end{equation*}
Letting $\eps\to 0$, then $\tau\to 0$, we deduce from the same reasoning as above \eqref{eq:constraint u*}, and the continuity of $u_0$, that
\begin{equation}\label{eq:IC sub}
u^*(0+,x_0,v_0) \leq \min \left ( u_0(x_0,v_0) ,\minp u^*(0+,x_0) +  \frac{|v_0|^2}{2}\right )\, .
\end{equation}
We find in the next iteration of \eqref{eq:IC sub} that $\minp u^*(0+,x_0)\leq \minp u_0(x_0)$, so that we find Definition \ref{def:subsol intro}(i) with the appropriate initial value $\min\left(u_0,\minp u_{0} + |v|^{2}/2\right)$.

\medskip 
\noindent {\bf \# Step 2: $u_*$ is a viscosity super-solution.}
\medskip

We consider a pair of $\mathcal{C}^1$ test functions $(\phi,\psi)$  as in Definition \ref{def:supersol intro}, namely $\u_* - \phi$ and $ \minp  \u_*  - \psi$ have a local minimum with respect to variables $(t,x)$ at the point $(t_0,x_0,v_0)$, with $t_0 > 0$. The following inequalities must be checked:
\begin{enumerate}[(i)]
\item $\partial_t \phi(t_0,x_0,v_0)  + v_0 \cdot \nabla_x \phi(t_0,x_0,v_0) - 1 \geq 0\,,$ if $\u_*(t_0,x_0,v_0) - \minp \u_*(t_0,x_0) - {\vert v_0\vert^2}/{2} < 0\,,$ 
\item $\partial_t \psi(t_0,x_0) \geq 0\,,$ if $ \aminp\u_* (t_0,x_0) = \{0\}\,$.
\end{enumerate}

In the former case, we define $ 2 \delta =  {\vert v_0\vert^2}/{2} +  \minp \u_* (t_0,x_0) - \u_*(t_0,x_0,v_0)  > 0$. There exist $(t_\eps,x_\eps,v_\eps) \to (t_0,x_0,v_0)$ such that $u^\eps - \phi$ has a local minimum at $(t_\eps, x_\eps, v_\eps)$ with respect to $(t,x)$, and $(t_\eps,x_\eps,v_\eps)$ realizes the $\liminf$ in \eqref{eq:semilimit}. For $\eps$ small enough, we have 
\begin{equation}\label{eq:unsaturated constraint supersol}
u^\eps(t_\eps,x_\eps,v_\eps)  - \minp \u^\eps (t_\eps,x_\eps) - \dfrac{|v_\eps|^2}{2} < - \delta \, .
\end{equation}
Otherwise, we would find a subsequence $(t_{\eps'},x_{\eps'},w_{\eps'})$ such that $u^{\eps'} (t_{\eps'},x_{\eps'},v_{\eps'}) - \u^{\eps'} (t_{\eps'},x_{\eps'},w_{\eps'}) - |v_{\eps'}|^2/2  \geq -\delta$. Passing to the $\liminf$ by compactness of $(w_{\eps'})$, we would obtain $\u_*(t_0,x_0,v_0) - |v_0|^2/2 \geq - \delta   + \u_*(t_0,x_0,w_0) \geq - \delta  +   \minp u_*(t_0,x_0)  $. The latter inequality is in contradiction with the definition of $\delta$.  

From \eqref{eq:unsaturated constraint supersol}, we deduce that 
\begin{align*}
\partial_t \phi(t_\eps,x_\eps,v_\eps)  + v_\eps \cdot \nabla_x  \phi(t_\eps,x_\eps,v_\eps) - 1 &= - \frac{1}{(2 \pi\eps)^{n/2} } \int \exp\left ( \frac1\eps \left ( \u^\eps(t_\eps,x_\eps,v_\eps) -\u^\eps(t_\eps,x_\eps,v') - \frac{\vert v_\eps \vert^2}2 \right )  \right )\, dv' \\
& \geq - \frac{e^{-\frac{\delta}{\eps}}}{(2 \pi\eps)^{n/2} } \int \exp\left ( \frac1\eps \left ( \minp \u^\eps(t_\eps,x_\eps) -\u^\eps(t_\eps,x_\eps,v')\right )  \right )\, dv'\, .
\end{align*}
We claim that the last contribution vanishes as $\eps \to 0$. Indeed, we can split the integral over $\mB(0,R)$ and $\R^n\setminus \mB(0,R)$, where $R$ is such that $\minp u^\eps(t_\eps,x_\eps) - u^\eps(t_\eps,x_\eps,v') \leq 2 B - |v'|^2/2 $ is uniformly negative for $|v'|>R$. So the contribution beyond $\mB(0,R)$ vanishes as $\eps\to 0$, and the contribution inside $\mB(0,R)$ is as follows:
\begin{equation*}
\frac{e^{-\frac{\delta}{\eps}}}{(2 \pi\eps)^{n/2} } \int_{\mB(0,R)} \exp\left ( \frac1\eps \left ( \minp \u^\eps(t_\eps,x_\eps) -\u^\eps(t_\eps,x_\eps,v')\right )  \right )\, dv'
\leq \frac{e^{-\frac{\delta}{\eps}}}{(2 \pi\eps)^{n/2} } |\mB(0,R)| \to 0\,, 
\end{equation*}
simply by definition of the $\minp$.
Hence the first condition is fulfilled.

For the other condition, we recall that  $\psi$ is a test function such that $\minp u_* - \psi$ has a global strict minimum at $(t_0,x_0)$ with $t_0>0$.  Suppose in addition that $\aminp u_*(t_0,x_0) = \{0\}$. 
%
%
We define a perturbed test function with localized and small perturbation:
\begin{equation}\label{eq:perturb test fun}
\phi^\eps(t,x,v) = \psi(t,x) + \min\left ( \dfrac{|v|^2}{2} , \eps^{1/2} \right ) \, .
\end{equation}
The method of the perturbed test function is classical in homogenization theory \cite{evans_perturbed_1989}. It was used in this context to deal with compact velocities as discussed in Section \ref{sec:intro compact} \cite{bouin_kinetic_2012,bouin_hamilton-jacobi_2015,caillerie_large_2017}. In the choice of the perturbation in  \eqref{eq:perturb test fun}, it is important that the threshold $\eps^{1/2}$ is such that $\eps \ll \eps^{1/2} \ll 1$. 
 
By uniform convergence of $\phi^\eps$, there exist $( t_{\eps},x_{\eps}, v_{\eps} )\to  (t_0, x_0,0) $ such that $u^\eps - \phi^\eps$ has a global minimum at $( t_{\eps},x_{\eps}, v_{\eps} )$.
The point $v_\eps \to 0$ results from the additional condition  $\aminp u_*(t_0,x_0) = \{0\}$. Otherwise, we could find another minimum point $v_*\neq 0$ by extracting a subsequence of $(v_\eps)$ outside a neighbourhood of the origin. 
%
We deduce from the equation on $u^\eps$ that 
\begin{multline*}
\partial_t \psi (t_{\eps},x_{\eps}) + v_\eps\cdot \nabla \psi (t_{\eps},x_{\eps})\\ \geq  1 -  \exp\left ( - \frac{|v_\eps|^2}{2\eps} + \min\left ( \dfrac{|v_\eps|^2}{2\eps} ,  \eps^{-1/2} \right )   \right ) 
 \frac{1}{(2\pi\eps)^{n/2}}  \int_{\mB(0,R)} \exp\left (  - \min\left ( \dfrac{|v'|^2}{2\eps} ,  \eps^{-1/2} \right )   \right ) \, dv' \\
- \frac{1}{(2 \pi\eps)^{n/2} } \int_{\R^n\setminus  \mB(0,R)} \exp\left ( \frac1\eps \left ( \u^\eps(t_\eps,x_\eps,v_\eps) -\u^\eps(t_\eps,x_\eps,v') - \frac{\vert v_\eps \vert^2}2 \right )  \right )\, dv' \, ,
\end{multline*}
where $R$ is chosen large enough to ensure the smallness of the last integral due to uniform quadratic growth of $u$ with respect to $v$.
Next, we obtain
\begin{multline*}
\partial_t \psi (t_{\eps},x_{\eps}) + v_\eps\cdot \nabla \psi (t_{\eps},x_{\eps})  \geq  1 -   
\frac{1}{(2\pi)^{n/2}} \int_{\mB(0,R\eps^{-1/2})} 
  \exp\left (  - \min\left ( \dfrac{|v'|^2}{2} ,  \eps^{-1/2} \right )  \right ) \, dv'  \\
- \frac{1}{(2\pi\eps)^{n/2}}  \int _{\R^n\setminus  \mB(0,R)} \exp\left ( \frac1\eps \left (2 B - \dfrac{|v'|^2}{2}\right ) \right )\, dv'   \,.
\end{multline*}
The first integral term converges to $1$ as $\eps\to 0$, and the  other one converges to 0. Therefore we get $\partial_t \psi (t_{0},x_{0}) \geq 0$ in the limit as required.


It remains to check the criterion for the initial data. We deduce from the  Duhamel formulation \eqref{eq:dufo} that 
\begin{equation*}
u^\eps(\tau,x,v) \geq \min \left (  \tau + u_0(x-\tau v, v) , -\eps \log \left (  M_\eps(v)\int_0^\tau e^{-s/\eps}   \rho^\eps (\tau-s,x-sv)\, ds\right ) \right ) - \eps\log 2
\end{equation*}
Let $R$ be so large that we can accurately restrict the integration with respect to $v$ on $\mB(0,R)$ as above. Let $\delta>0$. By definition of the lower limit \eqref{eq:semilimit}, there exists $\eps_0>0$ such that $u^\eps(\tau,x,v)\geq \minp u_*(0+,x_0) - \delta$ in the neighbourhood of $(0,x_0)$, and $v\in \mB(0,R)$, for $\eps < \eps_0$. By taking the limit $\eps\to 0$,  $\delta\to 0$, and $\tau\to 0$, we get that  
\begin{equation}\label{eq:IC sup}
u_*(0+,x_0,v_0) \geq \min \left ( u_0(x_0,v_0) ,\minp u_*(0+,x_0) +  \frac{|v_0|^2}{2}\right )\, .
\end{equation}
However, it cannot be deduced immediately that $u_*(0+,x_0,v_0) \geq \min\left(u_0(x_0,v_0),\minp u_{0}(x_0) + |v_0|^{2}/2\right)$ by iteration of this inequality. The latter  amounts to proving that $\minp u_*(0+,x_0) \geq \minp u_{0}(x_0)$. 

We argue by contradiction in order to uncover the boundary layer at initial time.  Suppose that $\minp u_*(0+,x_0) < \minp u_{0}(x_0)$ (strict jump of the minimum value from below at $t=0+$). First, we observe that necessarily, $\aminp u_*(0+,x_0) = \{0\}$. Indeed, if we denote by $v_*$ a minimum point of $u_*(0+,x_0,\cdot)$, then we get from \eqref{eq:IC sup} evaluated at $(x_0,v_*)$ that $u_*(0+,x_0,v_*) \geq \minp u_*(0+,x_0) + |v_*|^{2}/2$, hence $v_* = 0$.
Second, we define the test function 
\begin{equation*}
\psi(t,x) = \minp u_{0}(x_0) - \frac{1}{2}|x - x_0|^2 - t\,.
\end{equation*}
Consider the function $\minp u^\eps - \psi$. It takes value $0$ at $(0,x_0)$ by definition, and uniformly negative values as $\eps\to 0$ in the neighbourhood of $(0,x_0)$ because the $\liminf$ equals $\minp u_*(0+,x_0) - \minp u_{0}(x_0) < 0$ from our supposition. Hence, there exists $(t_\eps,x_\eps,v_\eps)\to (0,x_0,0)$, with $t_\eps >0$, such that $u^\eps - \phi^\eps$ has a global minimum at $(t_\eps,x_\eps,v_\eps)$, where $\phi^\eps$ is defined as in \eqref{eq:perturb test fun}. Then, we can repeat the same arguments as above (since $t_\eps>0$), ending up with $\partial_t\psi(0,x_0) \geq 0$ which is clearly a contradiction.

\section{The variational formula}
\label{sec:variational}

This section is devoted to the proof of Theorem \ref{th:hopf lax intro}. 
As usual, the arguments rely on the following semi-group property, which is straightforward. 
\begin{lem}[Dynamic programming principle]\label{lem:DPP}
For all intermediate $s\in (0,t)$, we have:
\begin{equation}\label{eq:DPP}
U(t,x,v) = \inf_{\footnotesize  \begin{array}{c}
\{(y,\sigma)\in \R^n\times \Sigma_s^t:\\ \gamma(t) = x, \sigma_N = v\}
\end{array}} \left \{ U(s,y,\sigma_0) + \mA_s^t[\sigma] \right  \}
\end{equation}
\end{lem}


\begin{proof}[Proof of Theorem \ref{th:hopf lax intro}]
We begin with the criteria for the viscosity subsolution. Let $U^*$ be the upper semi-continuous envelope of $U$. To prove that the constraint is always satisfied -- Definition \ref{def:subsol intro}(ii) -- we may choose any $\sigma$ such that $\sigma_0 = w$ and $\sigma_1 = v$ (or possibly $\sigma_0 = w = v$), for which we obtain that
\begin{equation*}
U(t,x,v) \leq U(t-\tau,y,w) + \dfrac12 |v|^2 + \tau, \quad  \text{ for some } y \text{ such that }  |x-y| \leq \tau \max (|v|,|w|).  
\end{equation*}
Taking the $\limsup$ as $\tau\to 0$, we get that
\begin{equation*}
U(t,x,v) \leq U^*(t,x,w) + \dfrac12 |v|^2 \,.
\end{equation*}
We deduce that 
\begin{equation}\label{eq:min U*}
U^*(t,x,v) \leq   \minp  U^*(t,x) + \dfrac12 |v|^2 .
\end{equation}

Let $(\phi, \psi)$ be a pair of test functions for $(U^*, \minp U^*)$  as in Definition \ref{def:subsol intro}(iii). We have:
\begin{equation}\label{eq:subsol 1}
U^*(t_0,x_0,v_0) - U^*(t_0-\tau,x_0 - \tau v_0,v_0) \geq   \phi(t_0,x_0,v_0) - \phi(t_0-\tau,x_0 - \tau v_0,v_0). 
\end{equation}
We may choose the constant configuration $\sigma \equiv v$ in \eqref{eq:DPP}, so as to obtain 
\begin{equation*}
U(t,x,v)  \leq  U(t-\tau,x - \tau v,v) + \tau\, ,  
\end{equation*}
in the neighbourhood of $(t_0,x_0,v_0)$. Taking the upper semicontinuous envelope, we obtain that 
\begin{equation}\label{eq:subsol 2}
U^*(t_0,x_0,v_0)  \leq  U^*(t_0-\tau,x_0 - \tau v_0,v_0) + \tau\, .  
\end{equation}
Combining \eqref{eq:subsol 1} and \eqref{eq:subsol 2}, we obtain $\partial_t \phi + v_0\cdot \nabla_x \phi - 1 \leq 0$ at $(t_0,x_0,v_0)$. 

From \eqref{eq:min U*} we obtain that the partial minimum of $U^*$ is attained at $v = 0$. We have
\begin{equation*}
U^*(t_0,x_0,0) - U^*(t_0-\tau,x_0 ,0) \geq   \psi(t_0,x_0) - \psi(t_0-\tau,x_0 ). 
\end{equation*}
Repeating the previous argument, we may choose the step function $\sigma$, such that $\sigma  = 0$ on $(t-\tau,t) $, and $\sigma = v$ on the last time $\{t\}$ -- or $[t-\delta,t]$ for arbitrary small $\delta$. We deduce that \begin{equation*}
U(t,x,v)  \leq  U(t-\tau,x  ,0) + \dfrac12 |v|^2 \, ,  
\end{equation*}
in the neighbourhood of $(t_0,x_0,0)$. Taking the upper semicontinuous envelope, we obtain that 
\begin{equation*}
U^*(t_0,x_0,0)  \leq  U^*(t_0-\tau,x_0  ,0)  \, ,  
\end{equation*}
We conclude as above that $\partial_t \psi(t_0,x_0) \leq 0$. 

The condition on the initial data is easily verified, as we have $U^*(0+,x,v) \leq u_0(x,v)$ from the very definition \eqref{eq:kin HL th} by choosing a free transport trajectory on $(0,\tau]$ for arbitrary small $\tau>0$, as well as $U^*(0+,x,v) \leq u_0(x,w) + |v|^2/2$ for all $w$ by choosing a trajectory with an instantaneous jump from $w$ to $v$. 

\bigskip

We continue with the criteria for the viscosity supersolution. Let $U_*$ be the lower semicontinuous envelope of $U$. Let $(\phi, \psi)$ be a pair of test functions for $(U_*, \minp U_*)$  as in Definition \ref{def:supersol intro}(ii). Let $(t^n, x^n, v^n)$ be a minimizing sequence in the neighbourhood of $(t_0,x_0,v_0)$ such that $U(t^n, x^n, v^n)$ converges to $U_*(t_0,x_0,v_0)$. Let $\tau > 0$, and $\sigma^n$ be a  nearly  optimal trajectory for \eqref{eq:DPP} on $(t^n-\tau,t^n]$: 
\begin{equation}\label{eq:DPP super}
 \frac1 n+ U(t^n, x^n, v^n) \geq  U(t^n-\tau , y^n, \sigma_0^n) + \mA_{t^n-\tau}^{t^n}[\sigma^n]\, ,
\end{equation}
where $x^{n}=y^{n}+\int_{t^{n}-\tau}^{t^{n}}\sigma^{n}(s)ds$.

\begin{lem}\label{eq:constant sigma}
There exists $\tau_0>0$ such that, if $U_*(t_0,x_0,v_0) <  \minp U_*(t_0, x_0)  +  |v_0|^2/2$, and if $\tau< \tau_0$, then $\sigma^n \equiv v^n$ is constant beyond some finite range $n\geq N_0(\tau)$. 
\end{lem}

\begin{proof}
We argue by contradiction. Suppose that we can extract a subsequence such that $\sigma^{n\prime}$ is not constant, then we must have $\mA_{t^n-\tau}^{t^n}[\sigma^{n\prime}] \geq |\sigma_N^{n\prime}|^2/2 = |v^{n\prime}|^2/2$ due to the last time discontinuity. From \eqref{eq:DPP super} we find that 
\begin{equation*}
\dfrac 1{n^\prime}+ U(t^{n\prime}, x^{n\prime}, v^{n\prime}) \geq
\minp U(t^{n\prime}-\tau , y^{n\prime}) +\dfrac12 |v^{n\prime}|^2\, .
\end{equation*}
Taking the lower continuous envelope, we find as $n^\prime\to+ \infty$:
\begin{equation*}
U_*(t_0,x_0,v_0) \geq  \minp U_*(t_0-\tau, x_0 + \mathcal O(\tau))  +\dfrac12 |v_0|^2 \, ,
\end{equation*}
where the displacement is controlled by  $|{y^{n}}' - x_0|\leq \left (\max |{\sigma^n}'|\right ) \tau$, the latter being clearly bounded by the definition of the action.
Hence, we obtain the reverse inequality $U_*(t_0,x_0,v_0) \geq  \minp U_*(t_0, x_0)  +\frac12 |v_0|^2$ as $\tau\to 0$, by lower semi-continuity. 
\end{proof}
  
Suppose that $U_*(t_0,x_0,v_0) <  \minp U_*(t_0, x_0)  + |v_0|^2/2$. On the one hand,  we necessarily have that $v_0\neq 0$. On the other hand,  $\sigma^n\equiv v^n$ is constant beyond some finite range for $\tau$ small enough by Lemma \ref{eq:constant sigma}. Therefore, for $n$ large enough,  we have 
\begin{equation*}
 \frac 1n+U(t^n, x^n, v^n) \geq U(t^n-\tau, x^n - \tau v^n, v^n) + \tau \, ,
\end{equation*}
because $v^n\to v_0\neq 0$ and there is no change in velocity. 
Then, passing to the limit $n\to +\infty$, we find
\begin{equation*}
U_*(t_0, x_0, v_0) \geq  U_*(t_0-\tau, x_0 - \tau v_0, v_0) + \tau\, ,
\end{equation*}
Using that $U_* - \phi$ has a minimum point at $(t_0,x_0,v_0)$, we find eventually that $\partial_t \phi(t_0,x_0,v_0) + v_0\cdot \nabla_x \phi(t_0,x_0,v_0) - 1 \geq 0$. 

\bigskip

It remains to check the inequality $\partial_t \psi(t_0,x_0)\geq 0$ provided that $\aminp U_* (t_0,x_0) = \left\lbrace 0 \right\rbrace$.

Let $r >0$. There exists $\delta>0$ such that $  U_*(t_0,x_0,v) \geq \minp U_*(t_0,x_0) + \delta$ for all $v\notin \mB(0,r)$ (otherwise we could extract a subsequence converging towards a minimal velocity outside $\mB(0,r)$). 

Again, let $(t^n,x^n,v^n)$ be such that $U(t^n,x^n,v^n)$ converges towards $\minp U_*(t_0,x_0)$, and that $(t^n,x^n)\to (t_0,x_0)$. 
Let $\tau > 0$, and $\sigma^n$ be a nearly optimal trajectory  in \eqref{eq:DPP} on $(t-\tau,\tau]$: 
\begin{equation}\label{eq:DPP super 2}
\frac1n + U(t^n, x^n, v^n) \geq  U(t^n-\tau , y^n, \sigma_0^n) + \mA_{t^n-\tau}^{t^n}[\sigma^n]\, .
\end{equation}
\begin{lem}
There exists $\tau_0>0$ such that $|\sigma_0^{n}|<2r$ beyond a certain range $n\geq N_0(\tau)$ if $\tau<\tau_0$.
\end{lem}
\begin{proof}
We argue by contradiction. Suppose that we can extract a subsequence such that $\sigma_0^{n\prime}\to v'_\tau\notin \mB(0,r)$. Then, we find as previously that 
\begin{align*}
\frac1{n'} + U(t^{n\prime},x^{n\prime},v^{n\prime}) & \geq U(t^{n\prime} - \tau ,y^{n\prime},\sigma_0^{n\prime}) \\
\minp U_*(t_0,x_0) & \geq U_*(t_0- \tau,x_0+ \mathcal O(\tau),v'_\tau)\\
\minp U_*(t_0,x_0) &  
\geq U_*(t_0,x_0,v'_\tau) - \omega(\tau) 
\end{align*}
%
where $\omega(\tau)\to 0$ as $\tau\to 0$ by lower semi-continuity. Then, it is sufficient to exhibit $\tau_0>0$ such that $|\omega(\tau)|<\delta/2$ for all $\tau<\tau_0$. 
\end{proof}

At this stage we have gained information on the initial velocity $\sigma_0^n$. However, the function $\sigma^n$ might be quite complicated. Nevertheless, as a direct consequence of the definition of the action $\mA$ \eqref{eq:A intro}, we get that 
\begin{equation*}
\mA_{t^n-\tau}^{t^n}[\sigma^n] \geq \dfrac12 \max_{i=1..N} |\sigma_i^n|^2\,.  
\end{equation*}
This enables   controlling the displacement in the following way:
\begin{equation}\label{eq:estimate displ}
|x^n - y^n| 
\leq \tau \max \left ( |\sigma_0^n|, \max_{i=1..N} |\sigma_i^n| \right ) \leq  
 \tau \left (  2r + \left ( 2 \mA_{t^n-\tau}^{t^n}[\sigma] \right )^{1/2} \right )
\end{equation}
Using \eqref{eq:DPP super 2}, and using the test function $\psi$, we find that 
\begin{align*}
& \frac1n +  U(t^n, x^n, v^n) - \minp U_*(t_0, x_0)\\ 
& \geq \minp U_*(t^n-\tau , y^n)  + \mA_{t^n-\tau}^{t^n}[\sigma^n] - \minp U_*(t_0, x_0) \\
& \geq \psi(t^n-\tau,y^n) -  \psi(t_0,x_0)  + \mA_{t^n-\tau}^{t^n}[\sigma^n] \\
&\geq \psi(t^n,x^n) -  \psi(t_0,x_0)  -\tau \partial_t \psi(t^n,x^n) - |x^n - y^n| \left | \nabla_x\psi(t^n,x^n) \right | + \mA_{t^n-\tau}^{t^n}[\sigma^n] \\
&\hspace{280pt}  - \tau\omega(\tau) - |x^n - y^n|\omega(x^n - y^n)  \,\\
&\geq \psi(t^n,x^n) -  \psi(t_0,x_0)  - \tau \partial_t \psi(t^n,x^n) - \tau \left ( 2r + \left ( 2 \mA_{t^n-\tau}^{t^n}[\sigma] \right )^{1/2} \right ) \left | \nabla_x\psi(t^n,x^n) \right |  + \mA_{t^n-\tau}^{t^n}[\sigma^n] \\
&\hspace{280pt}  - \tau \omega(\tau)  - |x^n - y^n|\omega(x^n - y^n)\, . 
\end{align*}
We focus on the leading order terms, neglecting the higher-order corrections for the sake of clarity (they can be dealt with similarly).
Plugging the following inequality into the preceding estimate,
\begin{equation}\label{eq:Young}
\tau  \left | \nabla_x\psi(t^n,x^n) \right |  \left ( 2 \mA_{t^n-\tau}^{t^n}[\sigma] \right )^{1/2}  \leq  \mA_{t^n-\tau}^{t^n}[\sigma^n] + \frac{\tau^2}2 \left\vert \nabla_x\psi(t^n,x^n) \right\vert^2\,,
\end{equation}
we get the following estimate:
\begin{align*}
& \frac1n + U(t^n, x^n, v^n)  - \minp U_*(t_0, x_0) \\
&\geq \psi(t^n,x^n) - \psi(t_0,x_0)  - \tau \partial_t \psi(t^n,x^n) - 2 r \tau   \left | \nabla_x\psi(t^n,x^n) \right |  - \frac{\tau^2}2 \left\vert \nabla_x\psi(t^n,x^n) \right\vert^2  \\
&\hspace{280pt}  - \tau \omega(\tau)  - |x^n - y^n|\omega(x^n - y^n)\, .
\end{align*}
By letting $n\to +\infty$, and then dividing by $\tau$, this becomes
\begin{equation*}
0 \geq - \partial_t \psi(t^0,x^0) - 2 r  \left | \nabla_x\psi(t^0,x^0) \right |  - \frac{\tau}2 \left\vert \nabla_x\psi(t^0,x^0) \right\vert^2 - \omega(\tau) \, , 
\end{equation*}
where we have used that $x^n - y^n = \mathcal{O}(\tau)$ \eqref{eq:estimate displ}. 
Now taking the limit $\tau\to 0$, then $r\to 0$, we conclude that $\partial_t \psi(t_0,x_0)\geq 0$, as required.  
 
The condition on the initial data is easily verified. Indeed, if we examine the definition \eqref{eq:kin HL th}, we find that piecewise trajectories with no jump $\mA_0^t[\sigma] =  \mathcal O(\tau)$, or with a single jump $\mA_0^t[\sigma] = |v|^2/2 + \mathcal O(\tau)$ are admissible, but any other choice with more than two jumps with a uniformly positive contribution as $\tau\to 0$ would lead to a higher cost than 
\begin{equation*}
\min \left ( u_0(x,v) , \minp u_0(x) + \frac{|v|^2}{2} \right )\, .
\end{equation*}
Hence, the latter is a lower bound for $U_*(0+,x,v)$ (and actually the exact limit of $U(\tau,x,v)$ as $\tau\to 0$). 
\end{proof}

\section{Computation of the kernel}
\label{sec:kernel}
We present below a systematic way to reduce the action of a piecewise linear curve $\mA_0^t[\sigma]$ to a finite number of cases that must be compared to obtain the kernel of the variational representation formula. Then, the one-dimensional case is calculated thoroughly. Some partial features of the two-dimensional case are given to illustrate the geometric complexity of the problem. 

\begin{figure} 
\begin{center}
\begin{subfigure}{.45\linewidth}
\includegraphics[width=\linewidth]{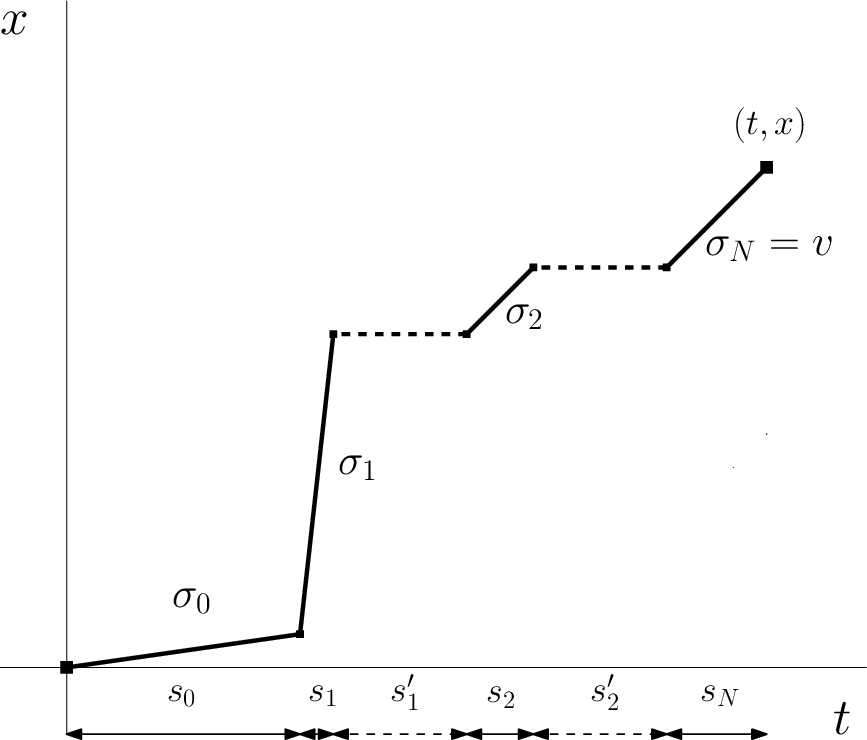} 
\caption{}
\end{subfigure}
\begin{subfigure}{.45\linewidth}
\includegraphics[width=\linewidth]{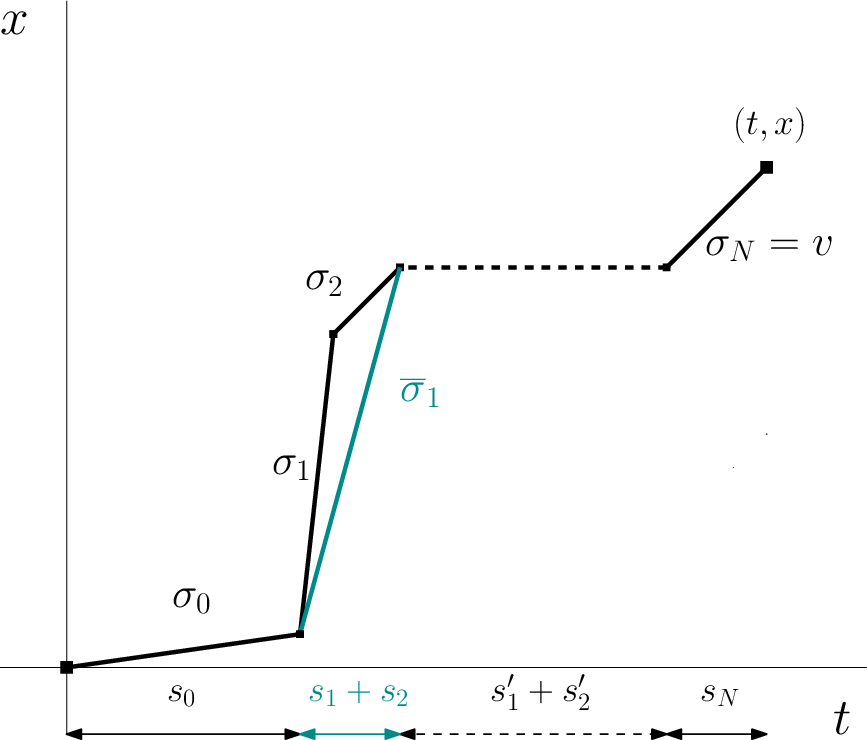} 
\caption{}
\end{subfigure}
\caption{Schematic reduction from  two non-trivial intermediate paths with non-zero velocities to a single one.}
\label{fig:arbitrary}
\end{center}
\end{figure}

\begin{prop}\label{prop:sigma1}
In arbitrary dimension, there is at most one  intermediate non trivial path $(s_1,\sigma_1)$ with non-zero velocity $\sigma_1\neq 0$ and at most one flat path with zero velocity. Hence, the action reduces to the following problem:
\begin{equation*}
\begin{cases}
\displaystyle\text{either }  \sigma \equiv \sigma_0 = \frac x t\,,\quad\text{then}\quad \mA = t{\bf 1}_{\sigma_0\neq 0} \quad \text{(free transport)}\medskip\\
\displaystyle \text{or}:\quad  
\mA = \frac{|\sigma_2|^2}{2}  + L(x,\sigma_0,\sigma_2)\, ,
\end{cases}
\end{equation*}
where 
\begin{equation}\label{eq:minimization}
L(x,\sigma_0,\sigma_2) =   \min_{\footnotesize \begin{array}{c}s_0\sigma_0 + s_1 \sigma_1 + s_2 \sigma_2 = x\\
s_0,s_1,s_2\geq 0\\
s_0 + s_1 + s_2 \leq t \end{array}} \left ( \frac{|\sigma_1|^2}{2} +  s_0 + s_1 + s_2 \right )  \, .
\end{equation}
\end{prop}
\begin{proof}
Take an arbitrary curve with initial velocity $\sigma_0$ and final velocity $\sigma_N$ (not to be changed). By re-arranging the order of intermediary paths $(s_1,\sigma_1),\dots (s_{N-1},\sigma_{N-1})$, we can merge all flat sections with zero velocity into a single one. Then, two adjacent sections with non-zero velocities (say $\sigma_1$ and $\sigma_2$) can be merged into a single one over the same time interval $s_1+s_2$ with average velocity $\overline{\sigma}_1$ by reducing the cost using the following convexity inequality, see also Figure \ref{fig:arbitrary}: 
\begin{align*}
 \left |\dfrac{s_1 \sigma_1 + s_2 \sigma_2}{s_1 + s_2}\right |^2 &\leq  \dfrac{s_1}{s_1 + s_2} |\sigma_1|^2 + \dfrac{s_2}{s_1 + s_2} |\sigma_2|^2\\
 |\overline{\sigma}_1|^2 &\leq |\sigma_1|^2 + |\sigma_2|^2
\end{align*}
\end{proof}

In the sequel we choose not to index the flat section with zero velocity for the sake of clarity.


In the one and two-dimensional cases, we can reduce further the complexity of the previous minimization problem. 

\begin{prop}[One-dimensional case]\label{prop:1D}
Assume that the spatial dimension is $n=1$. Then, 
\begin{equation*}
L(x, \sigma_0, \sigma_2) = \min (L(x, \sigma_0, \sigma_0), L(x, \sigma_2, \sigma_2))\, .
\end{equation*}
We denote in short $L(x, \sigma) = L(x, \sigma, \sigma)$. 
The values of $L(x, \sigma)$ are depicted in Figure \ref{fig:fund}. \\
It is noticeable that the minimization problem reduces to finding at most one non-zero time among $s_0,s_1$ and $s_2$.  
\end{prop}

\begin{prop}[Two-dimensional case]
Assume that the dimension is $n=2$. Then, 
\begin{equation*}
L(x, \sigma_0, \sigma_2) = \min \left (L(x, \sigma_0, \sigma_0), L(x, \sigma_2, \sigma_2), \tau_0 + \tau_2\right )\, ,
\end{equation*}
where the last option is the (unique)  decomposition of $x = \tau_0\sigma_0 + \tau_2 \sigma_2$ in the set $\{ s_0\geq 0, s_2\geq 0, s_0 + s_2\leq t \}$ to be considered only if $\spn(\sigma_0,\sigma_2) = \R^2$ and $x$ is reachable in time less than $t$. \\
It is noticeable that the minimization problem reduces to at most two non-zero times among $s_0,s_1,s_2$. There is no simple general formulation for $L(x,\sigma)$. However, some values are depicted in Figure \ref{fig:fund 2D}.
\end{prop}


%

\subsection{Computation of $L(x,\sigma_0,\sigma_2)$ in the one-dimensional case}

\label{sec:kernel1D}

\begin{lem}[The case $\sigma_0 \neq \sigma_2$]\label{lem:55}
Suppose that $\sigma_0 \neq \sigma_2$, then we cannot have both $s_0>0$ and $s_2>0$. Consequently, we have:
\begin{equation}\label{eq:L = min}
L(x, \sigma_0, \sigma_2) = \min (L(x, \sigma_0, \sigma_0), L(x, \sigma_2, \sigma_2))\, .
\end{equation}
\end{lem}

\begin{proof}
We notice that the option $s_2=0$ is equivalent to compute $L$ in the case $\sigma_0 = \sigma_2$ by identification of $s_0$ and $s_0+s_2$ in the two minimization problems. Hence, \eqref{eq:L = min} is a direct consequence of the reduction to $s_0 = 0$ or $s_2 = 0$. Next, we turn to the justification of this point. The basic idea is that it is always advantageous to spend more time with the fastest velocity. 

Assume by contradiction that both $s_0$ and $s_2$ are positive. Let $y = s_0 \sigma_0 + s_2\sigma_2 =  x - s_1\sigma_1$, and assume that $y\geq 0$ without loss of generality. Assume also that $\sigma_0$ is the largest velocity: $\sigma_0>\sigma_2$. It must be nonnegative, otherwise $y$ would be negative. Then, we have the alternative decomposition: $y = \left (s_0 + \frac{\sigma_2}{\sigma_0} s_2\right ) \sigma_0$ which is admissible with a better score, because $s_0 + \frac{\sigma_2}{\sigma_0} s_2 < s_0 + s_2\leq   t-s_1$, and $s_0 + \frac{\sigma_2}{\sigma_0} s_2$ is nonnegative because $y$ and $\sigma_0$ have the same sign. Hence, it is better  to consider $s_2 = 0$. 
\end{proof}

\begin{lem}[The case $\sigma_0 = \sigma_2$]
If the two velocities coincide, then $L(x,\sigma_0)$ is given by the values in Figure \ref{fig:fund}. 
\end{lem}

\begin{proof}

%
Recall that $L(x,\sigma_0)$ is given by the following minimization problem:
$$L(x,\sigma_0) = \min_{\footnotesize \begin{array}{c}s_0 \sigma_0 + s_1 \sigma_1 = x\\
s_0,s_1\geq 0\\
s_0 + s_1 \leq t \end{array}} \left ( \frac{|\sigma_1|^2}{2} +  s_0 + s_1 \right ).$$
We distinguish between several cases:

\medskip 
\noindent {\bf \# Case 1: $s_1 = 0$.}
\medskip

This is possible only if $s_0 = \frac{x}{\sigma_0} \in [0,t]$.  The value of the minimum is $L(x,\sigma_0) = \frac{x}{\sigma_0}$. 

\medskip 
\noindent {\bf \# Case 2: $s_0 = 0$.}
\medskip

Here, we have $\sigma_1 = x/s_1$. Thus the value of the minimum is
\begin{equation}
\min_{s \in [0,t] } \left(  \dfrac{\vert x \vert^2}{2 s ^2} + s  \right) = \begin{cases}
\dfrac32 |x|^{2/3}\, , & \text{ if } |x|\leq t^{3/2}\medskip\\
\dfrac{\vert x \vert^2}{2 t^2} + t\, , & \text{ if } |x|\geq t^{3/2} 
\end{cases}\label{eq:candidate 3}\,.
\end{equation}

\medskip 
\noindent {\bf \# Case 3: $s_0 + s_1 = t$ with both $s_0>0$ and $s_1>0$.}
\medskip

The problem is equivalent to minimize
$
\frac12\left |\frac{x - t \sigma_0}{s_1} + \sigma_0 \right |^2  +  t .
$
There is a critical interior point at $s_1^* = t - \frac{x}{\sigma_0}$ only if $\frac{x}{\sigma_0}\in (0,t)$, as in Case \#1. However the resulting value, simply $t$, is worse.

\medskip 
\noindent {\bf \# Case 4: $(s_0,s_1)$ is an interior point of the triangle.}
\medskip

The problem is equivalent to find a critical point of $\frac{|x - s_0 \sigma_0 |^2}{2 s_1^2} +  s_0 + s_1$. 
The first order condition is
\begin{equation*}
\begin{cases}
s_1^2 = \sigma_0 (x - s_0 \sigma_0) \medskip \\
s_1^3 = \vert x - s_0 \sigma_0 \vert^2
\end{cases} \quad \Longleftrightarrow \quad
\begin{cases}
s_0 = \dfrac{x}{\sigma_0} - \vert \sigma_0 \vert^2\medskip \\
s_1 = \vert \sigma_0 \vert^2,
\end{cases}
\end{equation*}
Therefore, under the conditions to be an interior point, in particular,
\begin{equation*}
0< \vert \sigma_0 \vert^2 < t, \qquad 0 <  \frac{x}{\sigma_0} < t,
\end{equation*}
the candidate for the minimum value, $\frac{\vert \sigma_0 \vert^2}{2} + \frac{x}{\sigma_0}$.  Again, this is worse than  in Case \#1.

As a conclusion, three possible candidates remain after the distinction of cases, namely
\begin{equation}
L(x,\sigma_0) = \begin{cases}
\dfrac{x}{\sigma_0} \, , & \text{ if } 0\leq \dfrac{x}{\sigma_0} \leq t \medskip\\ 
\dfrac32 |x|^{2/3}\, , & \text{ if } |x|\leq t^{3/2} \medskip\\
\dfrac{\vert x \vert^2}{2 t^2} + t\, , & \text{ if } |x|\geq t^{3/2}.\\
\end{cases}\label{eq:candidates all}
\end{equation}
The final score of the minimization procedure is depicted in Figure \ref{fig:fund}.
\end{proof}


\subsection{Partial computation of $L(x,\sigma_0,\sigma_2)$ in the two-dimensional case}

\begin{lem}
It is possible to choose either $s_0=0$, or $s_1=0$, or $s_2=0$ in the minimization problem \eqref{eq:minimization}.
\end{lem}
\begin{proof}
On the one hand, the case where $\sigma_0$ and $\sigma_2$ are colinear can be dealt with a one-dimensional argument as in the proof of Lemma \ref{lem:55}. 

On the other hand, assume that $\spn(\sigma_0,\sigma_2)= \R^2$ and that the minimum point satisfies $s_0>0$, $s_1>0$ and $s_2>0$. 

The minimization problem \eqref{eq:minimization} admits the following system of first order conditions:
\begin{equation*}
\begin{cases}
\sigma_1 = \mu s_1\\
1 = \mu\cdot \sigma_0 + \lambda_0 - \eta\\
1 = \mu\cdot \sigma_1 + \lambda_1 - \eta\\
1 = \mu\cdot \sigma_2 + \lambda_2 - \eta
\end{cases}
\end{equation*}
where the multipliers $\lambda_i$ and $\eta$ are non-negative, and vanish if the associated constraint (resp. $s_i\geq 0$ and $s_0 + s_1 + s_2 \leq t$) is not saturated. 

Hence, if $s_i>0$ for all $i=0,1,2$,  we necessarily have:  $\sigma_1\cdot \sigma_0 = |\sigma_1|^2 = \sigma_1\cdot \sigma_2 $. As a consequence, we have $\sigma_1 = (1 - \alpha)\sigma_0 + \alpha \sigma_2$, where $\alpha = \frac{\sigma_0\cdot (\sigma_0 - \sigma_2)}{|\sigma_0 - \sigma_2|^2}$, and alternatively,
\begin{equation*}
\begin{cases}
\sigma_0 = \dfrac{1}{1-\alpha} \sigma_1 + \left ( 1 -  \dfrac{1}{1-\alpha} \right ) \sigma_2 \medskip \\
\sigma_2 = \left ( 1 -  \dfrac{1}{\alpha} \right ) \sigma_2 + \dfrac{1}{\alpha} \sigma_1 
\end{cases}
\end{equation*}
Then, depending on the value of $\alpha$, either $\alpha \in [0,1]$ or $\alpha \in (-\infty, 0)$, or $\alpha\in (1,\infty)$, we can rearrange $x = (s_0 + (1-\alpha)s_1) \sigma_0 + (\alpha s_1 + s_2)\sigma_2$ (or the other way around), without changing the minimizing value. It remains to notice than at least one of the rearrangement is admissible (the latter one if $\alpha\in [0,1]$ for instance). 
\end{proof}

The case $s_1 = 0$ is only admissible if $x$ belongs to the triangle $\{ s_0 \sigma_0 + s_2 \sigma_2\; |\; s_0,s_2\geq 0\, , \; s_0 + s_2 \leq t \}$. The resulting value must be compared with the other options.

The case $s_1>0$ is investigated below. An explicit formulation as in the one-dimensional case seems beyond tractable computations. However, we provide some elements that indicate the complexity of the problem. 

From now on, we assume without of generality that $s_2 = 0$, and compute $L(x,\sigma_0)$ accordingly.
The minimization problem \eqref{eq:minimization} can be reformulated as follows:
\begin{equation}\label{eq:min 2D}
\min_{s_0,s_1\geq 0\, , \; s_0 + s_1 \leq t} \dfrac{|x - s_0 \sigma_0|^2}{2s_1^2} + s_0 + s_1 = \min_{s_0,s_1\geq 0\, , \; s_0 + s_1 \leq t} \dfrac{|x|^2 - 2 s_0 |x||\sigma_0|\cos \theta + s_0^2 |\sigma_0|^2 }{2 s_1^2} + s_0 + s_1\, ,
\end{equation}
where $\cos \theta = \frac{x\cdot \sigma_0}{|x||\sigma_0|}$.
It is not a convex function with respect to $(s_0,s_1)$, so we have to discuss the properties of its critical points further.

\medskip 
\noindent {\bf \# Step 1: Interior critical points.}
\medskip

Interior critical points are subject to the following equations:
\begin{equation}\label{eq:1st order cond}
\begin{cases}
s_0 |\sigma_0|^2 - |x||\sigma_0|\cos \theta + s_1^2 = 0\medskip \\
- \left( |x|^2 - 2 s_0 |x||\sigma_0|\cos \theta + s_0^2 |\sigma_0|^2 \right )+ s_1^3 = 0
\end{cases}
\end{equation}
The equation for $s_1$ is the following quartic one:
\begin{equation*}
Q(s_1) =    s_1^3   - \dfrac{s_1^4}{  |\sigma_0|^2} =   |x|^2 (\sin \theta)^2 \, .
\end{equation*}
The maximum of $Q(s_1)$ over $s_1\in \R_+$ is attained at $\bar s = \frac34 |\sigma_0|^2$. Therefore, the equation admits two positive roots if and only if
\begin{equation}\label{eq:66}
\dfrac{|\sigma_0|^3}{|x|} > \dfrac{16}{\sqrt{27}}|\sin \theta|\, . 
\end{equation}
The largest one does not correspond to a minimizer as the Hessian at the critical point is not positive: 
\begin{equation*}
\mathrm{Hess} = \begin{pmatrix}
\frac{|x|^2}{s_1^2} & \frac{2}{s_1} \medskip\\
\frac{2}{s_1} & \frac{3}{ s_1} \end{pmatrix} \, , \quad \det(\mathrm{Hess}) = \frac{4}{s_1^3} ( \bar s - s_1)\, .  
\end{equation*}
However, the smallest one corresponds to a local minimizer with a positive-definite Hessian. We select this one in the sequel as a possible candidate. Further conditions must be checked.

In order to meet the constraint $s_0 > 0$, the condition $s_1^2 <  |x||\sigma_0|\cos \theta$ must be fulfilled \eqref{eq:1st order cond}. It is so if and only if $\cos \theta >0$, and one of the two following conditions is verified:   
\begin{align}
&\begin{cases}
\text{either}& Q\left (\left ( |x||\sigma_0|\cos \theta\right )^{1/2}\right ) > |x|^2 (\sin \theta)^2 \\
\text{or}& Q'\left (\left ( |x||\sigma_0|\cos \theta\right )^{1/2}\right ) < 0\, .
\end{cases}\medskip\nonumber\\
\label{eq:cond s0}
\Leftrightarrow&\begin{cases}
\text{either}& \dfrac{|\sigma_0|^{3/2}(\cos\theta)^{3/2}}{|x|^{1/2}}  - (\cos \theta)^2 >    (\sin \theta)^2\; \Leftrightarrow\;\dfrac{|\sigma_0|^{3}}{|x|} > (\cos\theta)^{-3}\\
\text{or}& (\bar s)^2 < |x||\sigma_0| (\cos\theta) \; \Leftrightarrow\; \dfrac{|\sigma_0|^{3}}{|x|} < \dfrac{16}{9} (\cos \theta) 
\end{cases}
\end{align}
Note that the second condition is not compatible with \eqref{eq:66} unless $|\tan \theta| < \frac{1}{\sqrt{3}}$, {\em i.e.} $|\theta|< \frac{\pi}{6}$. In the latter case, either one the two conditions in \eqref{eq:cond s0} is fulfilled because $\cos \theta > \frac{\sqrt{3}}{2}$.
Note also that the first condition in \eqref{eq:cond s0} naturally implies \eqref{eq:66} because $|\cos \theta|^{-3} > \frac{16}{\sqrt{27}}|\sin \theta|$ for all $\theta$. Finally, we can summarize the necessary conditions for finding a local minimizer in the interior $(s_0>0,s_1>0,s_0+s_1<t)$ as follows:
  \begin{equation*}
  \begin{cases}
  \text{either}& \theta \in \left (-\dfrac\pi 6,\dfrac\pi 6\right )\mod 2\pi\,,\quad \text{and} \quad \dfrac{|\sigma_0|^3}{|x|} >\dfrac{16}{\sqrt{27}}|\sin \theta| \\
  \text{or}& \theta \in \left (- \dfrac\pi 2,-\dfrac\pi 6\right ) \cup \left (\dfrac\pi 6, \dfrac\pi2\right )\mod 2\pi\,,  \quad \text{and} \quad \dfrac{|\sigma_0|^3}{|x|} > (\cos\theta)^{-3} 
  \end{cases}
  \end{equation*}
The last condition to check, that is $s_0 + s_1 < t$ seems not computationnally tractable, but this is only a condition on the range of $t$ which does not appear in the previous computations. 

To complete the picture, we examine the behavior of the minimization problem at the boundary $\{s_0 = 0\}$. Indeed, recall that the boundary $\{s_1 = 0\}$ plays a particular role due to the constraint on the decomposition $x = s_0 \sigma_0 + s_2\sigma_2$. Moreover, the boundary  
 $\{s_0 + s_1 = t\}$ would provide a time-dependent condition that could complicate the picture. 

\medskip 
\noindent {\bf \# Step 2: Behavior near $\{s_0 = 0\}$.}
\medskip

This case coincides with \eqref{eq:candidate 3}. We assume that $|x|< t^{3/2}$ so that there is a critical time $s_1 = |x|^{2/3}\in (0,t)$. The derivative in the transverse direction is 
\begin{equation*}
\dfrac{\partial }{\partial s_0}\bigg|_{s_0 = 0} \left ( \dfrac{|x - s_0 \sigma_0|^2}{2s_1^2} + s_0 + s_1\right ) = - \dfrac{|x||\sigma_0|\cos\theta}{s_1^2} + 1 = - \dfrac{|\sigma_0|\cos\theta}{|x|^{1/3}} + 1\, .
\end{equation*}
Hence, the condition $|\sigma_0|^3/|x| < (\cos\theta)^{-3}$ guarantees that it is a local minimizer.

Therefore, we have a clear dichotomy in the case $\theta \in \left (- \frac\pi 2,-\frac\pi 6\right ) \cup \left (\frac\pi 6, \frac\pi2\right )\mod 2\pi$ (and $t$ is large enough): either $|\sigma_0|^3/|x| < (\cos\theta)^{-3}$, then there is a local minimizer at $\{s_0 = 0\}$, but none in the interior. There is a transition for $|\sigma_0|^3/|x| > (\cos\theta)^{-3}$ as the local minimizer moves to the interior of the triangle.

The picture is more intricate when $\theta \in \left (-\frac\pi 6,\frac\pi 6\right )\mod 2\pi$ since there could be a pair of local minimizers, in the interior and at the boundary, precisely when 
\begin{equation*}
\dfrac{16}{\sqrt{27}}|\sin \theta| <  \dfrac{|\sigma_0|^3}{|x|}  < (\cos\theta)^{-3}\, . 
\end{equation*} 
In fact, it can be proven that this corresponds exactly to the case where there is a second interior critical point which is a saddle point. There is no clear order between the two minimizers: both can be global minima depending on the parameters (numerical check). 

The previous conditions are summarized in Figure \ref{fig:fund 2D} for large time.

\begin{figure}
\begin{center}
\includegraphics[width = .6\linewidth]{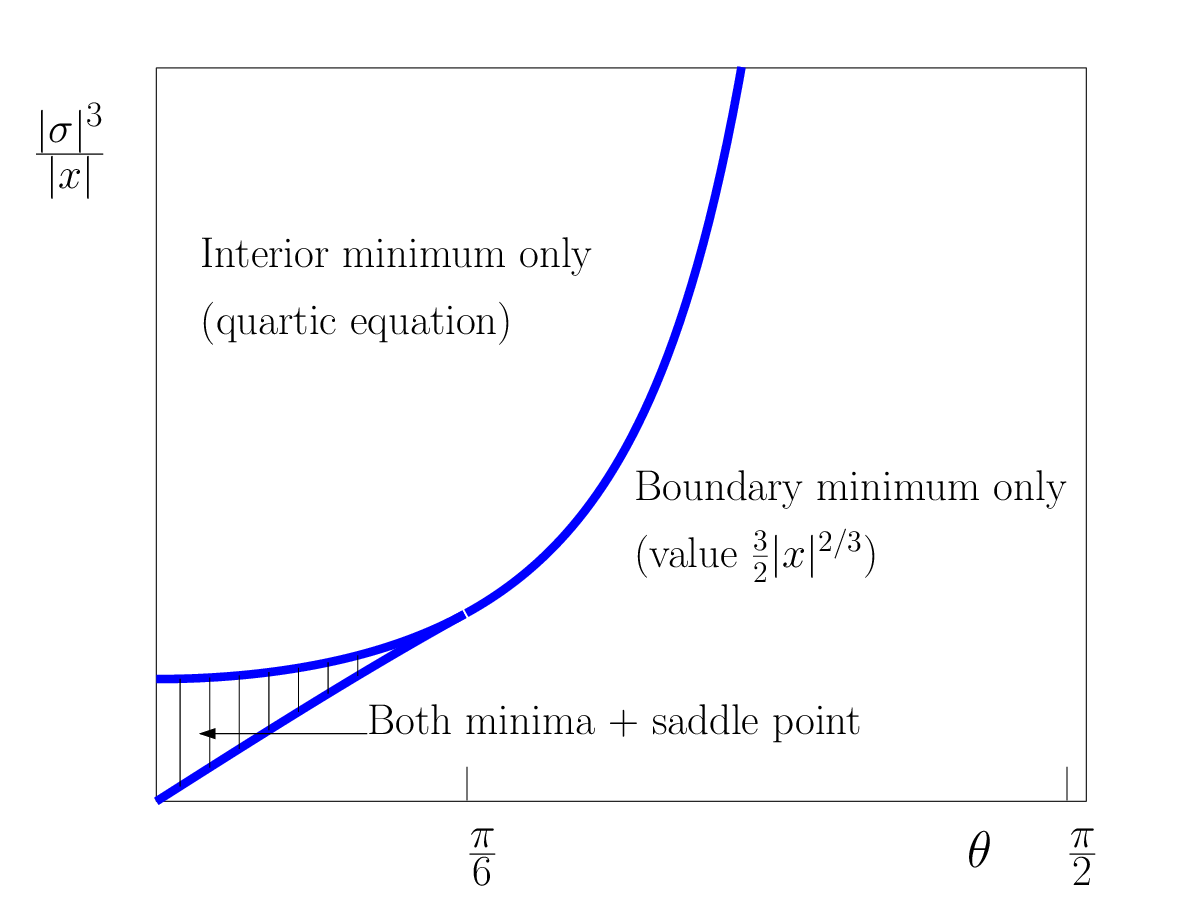}
\end{center}
\caption{Partial information about $L(x,\sigma)$ in the case of large time $t$ (so that the minimal value is independent of the time variable). Here, $\theta$ denotes the angle between $x$ and $\sigma$. There exists a zone for grazing angle (below $\pi/6$) and relatively small speed $|\sigma|\lesssim |x|^{1/3}$, but not too small (shaded area) such that two local minima co-exist in the minimizing problem \eqref{eq:min 2D}, separated by a saddle point. The picture is symmetric with respect to $\theta = 0$.}\label{fig:fund 2D}
\end{figure}

\section{Rate of acceleration in kinetic reaction-transport equations}\label{sec:acc}

We consider the following kinetic reaction-transport problem \cite{hadeler_reaction_1999,
schwetlick_travelling_2000,
cuesta_traveling_2012,
bouin_propagation_2015}, in one dimension of space $n = 1$:
\begin{equation}\label{eq:kinKPPR}
\begin{cases}
\eps \left( \partial_t f^\eps(t,x,v) + v \partial_x f^\eps(t,x,v) \right)=   M_\eps(v) \rho^\eps(t,x) - f^\eps(t,x,v)   + r \rho^\eps(t,x) \left( M_\eps(v) -  f^\eps(t,x,v) \right)\medskip\\
f^\eps(0,x,v) = \ind{\{x\leq 0\}}M_\eps(v)
\end{cases}
\end{equation} 

It is obtained from the problem \eqref{eq:kinreac} with $\eps = 1$ in the  scaling regime: $\left(\frac{ t}{\eps},\frac{ x}{\eps^{3/2}},\frac{ v}{\eps^{1/2}} \right)$. 
 

We build on the previous results in the present work to investigate the limit of $u^\eps = -\eps\log f^\eps$, following the lines of the celebrated works by Freidlin \cite{freidlin_geometric_1986}, Evans and Souganidis \cite{evans_pde_1989}, and Barles, Evans and Souganidis \cite{barles_wavefront_1990}. 

Let $\Upsilon$ be the rate of expansion of the spatial density of population:
\begin{equation}\label{eq:Upsilon}
\Upsilon =  \frac{\left( (2/3) r \right)^{3/2}}{1+r}\, .
\end{equation}


\begin{thm}[Weak propagation result]\label{th:propagation}
We have the following characterization of the accelerating front in terms of the spatial density:
\begin{equation}\label{eq:prop rasult}
\begin{cases}
\displaystyle \lim_{\eps\to 0} \rho^\eps(t,x) = 0 & \text{locally uniformly on $\mathrm{Int}\:\{(t,x): x> \Upsilon t^{3/2}\}$},\medskip\\
\displaystyle \lim_{\eps\to 0} \rho^\eps(t,x) = 1 & \text{locally uniformly on $\mathrm{Int}\:\{(t,x): x< \Upsilon t^{3/2}\}$}.
\end{cases}
\end{equation}
\end{thm}

The rest of this section is devoted to the proof of this theorem.

The immediate consequence of the non-linear reaction term is the following comparison result \cite{cuesta_traveling_2012}:
\begin{equation}\label{eq:estf0}
\forall (t,x,v) \in \R_+\times \R \quad 0\leq f^\eps(t,x,v)\leq  M_\eps(v).
\end{equation}
This readily implies the following estimate on the putative limiting function $u$:
\begin{equation*}
\forall (t,x,v) \in \R_+ \times \R  \times \R  \quad u(t,x,v) \geq \frac{\vert v \vert^2}{2} \, . 
\end{equation*}

Indeed, the latter obstacle condition  supplements the definition of the limiting problem satisfied by the function $u=\lim u^\eps$. To make it clear, we update the pair of Definitions \ref{def:subsol intro} and \ref{def:supersol intro} in the following way: $u$ is a viscosity solution of the limiting obstacle problem if  the two following sets of criteria are satisfied:
\begin{defi}[Sub-solution]\label{def:subsol}
Let $\underline{u}_0$ be a continuous function, and $T>0$. An upper semi-continuous function $\underline{u}$ is a \textbf{viscosity sub-solution} of  the limiting obstacle problem on $(0,T) \times \R^{2}$ with initial data $\underline{u}_0$ if:\medskip \\
\noindent{\em (i)} $\underline{u}(0+,\cdot,\cdot) \leq  \underline{u}_0 $.\medskip \\
\noindent{\em (ii)} It satisfies the constraint
$ \underline{u}(t,x,v) -  \minp \underline{u}(t,x) - \frac{\vert v\vert^2}{2} \leq 0$.
\medskip\\
{\noindent{\em (iii)}}
For all pair of test functions $(\phi,\psi) \in \mathcal{C}^1\left( (0,T) \times \R^{2} \right) \times \mathcal{C}^1\left( (0,T) \times \R \right) $, if $(t_0,x_0,v_0)$ is such that both $\underline{u}(\cdot,\cdot,v_0) - \phi(\cdot,\cdot,v_0)$ and $  \minp \underline{u}  - \psi$ have a local maximum at $(t_0,x_0)$ with $t_0>0$, then
\begin{equation}\label{eq:S1}
\begin{cases}
 \partial_t \phi(t_0,x_0,v_0) + v_0 \cdot \nabla_x \phi(t_0,x_0,v_0) - 1 \leq 0 & \displaystyle \text{if}\quad  \underline{u}(t_0,x_0,v_0)> \frac{|v_0|^2}{2}, \medskip \\
\partial_t \psi(t_0,x_0) \leq -r & \text{if} \quad \minp \underline{u}(t_0,x_0)>0. 
\end{cases}
\end{equation}
\end{defi}

\begin{defi}[Super-solution]\label{def:supersol}
Let $\overline{u}_0$ be a continuous function, and $T>0$. A lower semi-continuous function $\overline{u}$ is a \textbf{viscosity super-solution} of the the limiting obstacle problem on $(0,T) \times \R^{2}$  with initial data $\overline{u}_0$ if:
\medskip\\
{\noindent{\em (i)}} 
$\overline{u}(0+,\cdot,\cdot) \geq \overline{u}_0$.\medskip \\
{\noindent{\em (ii)}} 
$\overline{u}(t,x,v) \geq \frac{\vert v \vert^2}{2}$\medskip
\\
{\noindent{\em (iii)}}
For all pair of test functions $(\phi,\psi) \in \mathcal{C}^1\left( (0,T) \times \R^{2} \right) \times \mathcal{C}^1\left( (0,T) \times \R \right) $, if $(t_0,x_0,v_0)$ is such that both $\overline{u}(\cdot,\cdot,v_0) - \phi(\cdot,\cdot,v_0)$ and $  \minp \overline{u}   - \psi $ have a local minimum  at $(t_0,x_0)$ with $t_0>0$, then 
\begin{equation*}
\begin{cases}
\displaystyle\partial_t \phi(t_0,x_0,v_0)  + v_0 \cdot \nabla_x \phi(t_0,x_0,v_0) - 1 \geq 0 & \displaystyle \text{if} \quad \overline{u}(t_0,x_0,v_0) - \minp  \overline{u}(t_0,x_0) - \dfrac{\vert v_0\vert^2}{2} < 0, \medskip \\
\partial_t  \psi(t_0,x_0) \geq -r  &  \text{if} \quad \aminp \overline{u} (t_0,x_0) = \left\lbrace 0 \right\rbrace .
\end{cases}
\end{equation*}
\end{defi}


In contrast with the linear case $r=0$ \eqref{eq:limit}, we do not provide here  an informal description of the limiting problem, to avoid any possible ambiguity. Nonetheless, the pair of Definitions \ref{def:subsol} and \ref{def:supersol} completely defines the limiting problem, as it comes with a comparison principle as in the linear case $r=0$. 

For the purpose of dealing with spatially unbounded solutions, as for the fundamental solution \eqref{eq:candidates intro}, we need some extension of the previous assumption \eqref{eq:v2 plus borne} to quadratic growth in the $x$ variable, see Appendix  \ref{sec:unbded}.

\begin{thm}[Comparison principle]\label{theo:compobs}
Let $\underline{u}$ be a viscosity sub-solution in the sense of Definition \ref{def:subsol} and $\overline{u}$ be a viscosity super-solution in the sense of Definition \ref{def:supersol}  with continuous initial data    $\underline{u}_0 \leq \overline{u}_0$. Assume that $\underline{u}$ and $\overline{u}$ verify the growth condition \eqref{eq:growth condition},
\begin{equation} \label{eq:growth condition sec 6}
\frac1A |v|^2 - A|x|^2 - A \leq \underline{u}(t,x,v), \overline{u}(t,x,v) \leq A |v|^2 + A|x|^2 + A\, ,
\end{equation}
for some $A>1$.
Then $\underline{u} \leq \overline{u}$ on $(0,T) \times \R^{2}$.
\end{thm}

The proof in Section \ref{sec:Comp} -- extended in Appendix \ref{sec:unbded} to quadratic growth in the spatial variable -- goes the same, except that the test functions associated with the sub-solution can only be used under additional conditions on $\underline{u}$ and $\minp \underline{u}$ \eqref{eq:S1}. 
Nevertheless, we immediately get the correct comparison if, either $\underline{u}(t_0,x_0,v_0)\leq  |v_0|^2/2$ (because $\overline{u}(t_0,x_0,v_0)\geq  |v_0|^2/2$, see Definition \ref{def:supersol}(ii)), or  $\minp \underline{u}(t_0,x_0)\leq 0$ (because $\underline{u}(t_0,x_0,v_0) \leq \minp  \underline{u}(t_0,x_0) + |v_0|^2 /2  \leq |v_0|^2/2 \leq \overline{u}(t_0,x_0,v_0)$). The technical details are left to the reader. 

\bigskip 

Similarly, our convergence result still holds, if the initial data $u_0$ has at most quadratic growth in the spatial variable. As it is not the case for the specific initial data associated with \eqref{eq:kinKPPR}, namely $u_0 = \mathbf{0}_{y\leq 0} + \frac{|v|^2}{2}$, the following result is stated for a different class of initial data. We shall see later on that some suitable approximation enables to recover our singular (infinite valued) initial data, thanks to a precise knowledge of the variational representation formula.

\begin{thm}[Convergence]\label{HJlimit:obs}
Assume that  $\u_0$ is continuous and satisfies (\ref{eq:growth condition sec 6}). 
Let $\u^\eps = -\eps \log f^\eps$ be the logarithmic transformation of the solution of \eqref{eq:kinKPPR}, with the initial data $\u^\eps(0,\cdot) = u_0$. Then,  $\u^\eps$ converges locally uniformly towards $u$  as $\eps\to 0$, where $u$ is the unique viscosity solution of the limiting obstacle problem in the sense of Definitions \ref{def:subsol} and \ref{def:supersol} with initial data $\min\left(u_0,\minp u_{0} + |v|^{2}/2\right)$.
\end{thm}

We need to introduce the  representation formula of the limiting obstacle problem. For that purpose, given $t>0$, we define $\Gamma_0^t$ the set of piecewise linear curves on $[0,t]$, parametrized by $y\in \mathbb{R}$ the starting point, and $\sigma \in \Sigma_0^t$ the piecewise constant derivative. We subsequently define $\Theta_0^t$ the set of stopping times associated with curves $\gamma\in \Gamma_0^t$. A mapping  $\vartheta: \Gamma_0^t\to [0,t]$ is a stopping time if for all $\gamma, \hat{\gamma}\in \Gamma_0^t\times \Gamma_0^t$, such that $\gamma (s) = \hat{\gamma}(s)$ for all $s\in [\tau,t]$ and $\vartheta [\gamma]\geq \tau$, then $\vartheta [\gamma]= \vartheta [\hat{\gamma}]$.
Then, the candidate for the viscosity solution $u$ of the obstacle problem is the following solution of a two-players game (see the proof below, Step 4, for a discussion)
\begin{multline}\label{eq:obstacle variational}
\mathcal U(t,x,v) = \sup_{\vartheta\in \Theta_0^t} \quad \inf_{(y,\sigma)\in \mathbb{R}\times \Sigma_0^t} \left \{  \mA_{  \vartheta[\gamma]}^{t}[\sigma] + \frac{|\sigma(\vartheta[\gamma])|^2}{2}\ind{\vartheta[\gamma]>0}  + \left ( \mathbf{0}_{y\leq 0} +\frac{|\sigma_0|^2}{2}\right )\ind{\vartheta[\gamma]=0}   \right.\\ 
\left .\; \bigg| \; \gamma(s) = y + \int_0^s \sigma(s')\, ds'  \, , \; \gamma(t) = x  \,, \; \sigma_N = v \right \}\,,  
\end{multline}
where $\mathbf{0}_{y\leq 0} $ is the convex indicator function, taking value $0$ on $\{y\leq 0\}$ and $+\infty$ elsewhere.

As compared to the case $r=0$, the new action is defined in this section as follows, 
\begin{align}
\mA_s^t[\sigma] &= \frac12 \sum_{i=1}^N |\sigma_i|^2 + \sum_{i=0}^{N} (t_{i+1} - t_{i}) {\bf 1}_{ \sigma_i\neq 0} - r \sum_{i=0}^{N} (t_{i+1} - t_{i}) {\bf 1}_{ \sigma_i =  0}\,\label{eq:new action} 
\\
& =   \frac12 \sum_{i=1}^N |\sigma_i|^2 + (1+r)\sum_{i=0}^{N} (t_{i+1} - t_{i}) {\bf 1}_{ \sigma_i\neq 0} - r t.
\label{eq:new action bis} 
\end{align}
Alternatively speaking, any run  induces a running cost of $1$ per unit of time, except if it occurs at speed zero, in which case there is a (negative) running cost of $-r$ per unit of time.

\begin{thm} \label{th:nonlin}
Let $f^\eps$ be the solution of \eqref{eq:kinKPPR}, and $\u^\eps = -\eps \log f^\eps$. Then, we have $\mU(t,x,v) = \lim_{\eps\to0} u^\eps(t,x,v)$ for all $t>0$ and $(x,v)\in \R^2$. In addition, we have:
\begin{equation}\label{eq:mU}
\mU(t,x,0) = \max\left ( 0, \mu(t,x)\right)\, ,
\end{equation}
where 
\begin{equation}\label{eq:mumu}
\mu(t,x) = \begin{cases}
\dfrac32 |(1+r)x|^{2/3} - rt  & \text{ if } 0\leq x\leq (1+r)^{1/2}t^{3/2},\medskip\\
\dfrac{\vert x \vert^2}{2 t^2} + t  & \text{ if } x\geq (1+r)^{1/2}t^{3/2}. 
\end{cases}
\end{equation}
\end{thm}

\begin{rmq}[The Freidlin condition is not satisfied]\label{rm:rem}
The formula \eqref{eq:mU} essentially says that the value of the obstacle problem at $(t,x,0)$ is obtained by solving the problem without obstacle (resulting from the asymptotic analysis of \eqref{eq:kinKPPR} without the non-linear contribution $-r\rho^\eps f^\eps$), and truncating it above zero. This holds true for the classical Fisher-KPP equation as well as for other problems, see \cite{freidlin_geometric_1986,
evans_pde_1989}. A sufficient condition is usually referred to as Freidlin's N condition \cite{evans_pde_1989}. It happens that this condition is violated in our case for large $v$. Moreover,  the expression of $\mU(t,x,v)$ is not given by the truncation of the problem without obstacle, for large $v$ (details not shown for the sake of conciseness). However, our conclusion about the propagation result \eqref{eq:prop rasult} is not impacted since the value at $0$ coincides with the minimum value $\minp \mU(t,x)$ which carries enough information to detect the emergence of $\rho^\eps(t,x)$ for vanishing $\eps$. 
\end{rmq}

\begin{proof}
We follow the main lines of \cite{evans_pde_1989}. The main difficulty comes from the singularity at $t = 0$, where it is expected that $\mU(t,x,v)\to +\infty$ as $t\to 0$,  when $x>0$, because of the particular initial data which takes infinite value on $\mathbb{R}_+^*\times\mathbb{R}$. However, the variational formulation \eqref{eq:obstacle variational} enables to circumvent any such singularity at $t=0$.

\medskip

\noindent{\bf \# Step 1:  Quadratic bounds on $u^\eps$ for positive time.} 

\medskip

We begin by establishing rough estimates on $u^\eps$ which are compatible with \eqref{eq:growth condition sec 6}. 

\begin{prop}\label{prop:uepsestimate}
There exists $\eps_0>0$ such that for all $\eps < \eps_0$, for all $(t,x,v) \in \R_+^* \times \R \times \R$, one has the  estimate
\begin{equation*}
\frac{v^2}{2} + \eps \log\left ((2\pi \eps)^{1/2}\right ) \leq u^\eps(t,x,v) \leq \begin{cases} \displaystyle 
\frac{v^2}{2} + t - \eps \log\left (\frac t4\right ) 
& \text{ if $ x + (t/2)|v| \leq 0$},\medskip\\
\displaystyle \frac{v^2}{2} + t + \frac12 \left ( \left ( \frac{2x}{t} + |v|\right )_+ + 1 \right )^2 - \eps \log\left (\frac t2\right )   
& \text{ if $x + (t/2)|v| \geq 0$}.
\end{cases}
\end{equation*} 
\end{prop}

\begin{proof}
The lower bound is a direct consequence of the maximum principle \eqref{eq:estf0}. As a side effect, the  contribution $r\rho^\eps( M_\eps - f^\eps)$ is nonnegative, so that we can restrict to the linear problem with $r=0$ in order to obtain an upper bound for $u^\eps$ (a lower bound for $f^\eps$).

We proceed in two steps.  
Omitting the non-local term $M_\eps \rho^\eps$, we get that $f^\eps$ is bounded below by the damped free transport equation,
$f^\eps(t,x,v) \geq f^\eps(0,x-tv,v) e^{-t/\eps}$.
Hence, we have
\begin{equation*}
\rho^\eps(t,x) \geq  e^{-t/\eps} \frac{1}{(2\pi\eps)^{1/2}} \int_{\R} \ind{\{x\leq tw\}}   \exp\left(-\frac{w^2}{2\eps} \right) dw\\
\end{equation*}
Plugging this bound into the Duhamel formula, we deduce
\begin{align*}
f^\eps(t,x,v) &\geq M_\eps(v) \frac{1}{(2\pi\eps)^{1/2}} \int_0^t \left(  e^{-(t-s)/\eps} \int_{\R} \ind{\{x - sv \leq (t-s)w\}}
  \exp\left(-\frac{w^2}{2\eps} \right) dw \right)e^{-s/\eps} ds\nonumber\\
&= M_\eps(v) \frac{1}{(2\pi\eps)^{1/2}} e^{-t/\eps}  \int_0^t \int_{ \frac{x-sv}{t-s}  }^{+\infty} \exp\left(-\frac{w^2}{2\eps} \right) dw \, ds
\end{align*}
We restrict to intermediate times $s\in (0,t/2)$. We have $x - sv \leq x + (t/2)|v|$. Suppose the latter is nonpositive, then we have:
\begin{equation*}
f^\eps(t,x,v) \geq M_\eps(v)  e^{-t/\eps}  \int_0^{t/2} \frac12 \, ds =  \frac t4 M_\eps(v)  e^{-t/\eps} \, .
\end{equation*}
Suppose on the contrary that $ x + (t/2)|v|>0$, then we have 
$\frac{x - sv}{t-s} \leq \frac{x + (t/2)|v|}{t/2}$. We denote the latter expression by $W$, so that 
\begin{equation*}
f^\eps(t,x,v) \geq \frac t2 M_\eps(v)  e^{-t/\eps} \frac{1}{(2\pi\eps)^{1/2}}     \exp\left(-\frac{(W+1)^2}{2\eps} \right ) 
\end{equation*} 
We can even omit the (large) prefactor $(2\pi\eps)^{-1/2}$ when $\eps< \eps_0$ is small enough. Taking the logarithm on both sides concludes the proof.
\end{proof}
%
%

\medskip

\noindent{\bf \# Step 2:  Definition of appropriate sub-and super-solutions.} 

\medskip

Since we cannot handle arbitrary unbounded initial data currently, we wish to introduce two auxiliary problems approximating the true problem on $u^\eps$, and bearing quadratic growth at most -- in agreement with \eqref{eq:growth condition sec 6} and the analysis in Appendix \ref{sec:unbded}. 

Let $A>0$ and $\alpha>0$. We define an upper bound for $f^\eps$ (a lower bound for $u^\eps$) as follows: let $\hat{f}^\eps$ be the solution of the reaction-transport equation \eqref{eq:kinKPPR} associated with the initial data  
\begin{equation*}
\hat{f}^\eps(0,x,v) = \exp\left ( - \frac{\hat{u}_0(x,v)}{\eps} \right )\, , \quad \hat{u}_0(x,v) = A \chi (x/\alpha) + \frac{v^2}{2}  +   \eps_0 \log\left ((2\pi \eps_0)^{1/2}\right ), 
\end{equation*}
where $\chi (x) = 0$ if $x\leq 0$, $\chi (x) = 1$ if $x\geq 1$ and $\chi$ is affine between these two values. We clearly have $\hat{f}^\eps(0,x,v) \geq \ind{\{x\leq 0\}}M_\eps(v)$ by construction when $\eps < \eps_0$ is small enough. 

Let $\tau\in (0,1)$. We define a lower bound for $f^\eps$ (an upper bound for $u^\eps$) as follows: let $\check{f}^\eps$ be the solution of the reaction-transport equation \eqref{eq:kinKPPR} associated with an initial data 
\begin{equation*}
\check{f}^\eps(0,x,v) = \exp\left ( - \frac{\check{u}_0(x,v)}{\eps} \right ),
\end{equation*}
where $\check{u}_0$ is a continuous approximation (from above) of 
\begin{equation*}
\begin{cases} \displaystyle 
\frac{v^2}{2} + \tau - \eps_0 \log\left (\frac \tau4\right ) 
& \text{ if $ x + (\tau/2)|v| \leq 0$},\medskip\\
\displaystyle \frac{v^2}{2} + \tau + \frac12 \left ( \left ( \frac{2x}{\tau} + |v|\right )_+ + 1 \right )^2  - \eps_0 \log\left (\frac \tau2\right )  
& \text{ if $x + (\tau/2)|v| \geq 0$}.
\end{cases}
\end{equation*}
We have $\hat{f}^\eps(0,x,v) \leq f^\eps(\tau,x,v)$ by Proposition \ref{prop:uepsestimate}. Therefore, the comparison principle implies $\hat{f}^\eps(t,x,v) \leq f^\eps(t + \tau,x,v)$ for all time $t>0$. 

In summary, for all $t>\tau$, we have:
\begin{equation}\label{eq:comparison u hat check}
\hat{u}^\eps(t,x,v) \leq u^\eps(t,x,v) \leq \check{u}^\eps (t - \tau,x,v)\, .
\end{equation}


Both problems are well-prepared in the sense that the initial data $\hat{u}_0$ and $\check{u}_0$ do not depend on $\eps$. Moreover, both $\hat{u}_0$ and $\check{u}_0$ converge to the singular initial data  $\mathbf{0}_{y\leq 0} + \frac{v^2}{2}$ when $A\to +\infty$, and $\eps_0\to 0$, then $\tau\to 0$. 
Moreover, the corresponding solutions verify the quadratic bounds in \eqref{eq:growth condition sec 6} which are required for comparing viscosity sub- and super-solutions. 

%
%

\medskip

\noindent{\bf \# Step 3: Characterization of the limits $\hat{u}$ and $\check{u}$.} 

\medskip

Following the ideas of optimal control, as in \cite{evans_pde_1989}, we can formulate the solutions of the two problems $\hat{u}^\eps$ and $\check{u}^\eps$ as $\eps\to 0$. Indeed, it can be established that  the limit are solutions of the following variational problems. 


\begin{prop}[Solution of the variational problem with at most quadratic initial data] \label{prop:varobs}
Consider any initial data $u_{0}$ which is continuous and satisfies \eqref{eq:growth condition sec 6}. Let \begin{multline}
\label{eq:obstacle variational bis 0}
U(t,x,v) = 
  \sup_{\vartheta\in \Theta_0^t} \quad \inf_{(y,\sigma)\in \mathbb{R}\times \Sigma_0^t}\left \{  \mA_{  \vartheta[\gamma]}^{t}[\sigma] +  \frac{|\sigma(\vartheta[\gamma])|^2}{2}\ind{\vartheta[\gamma]> 0} +   u_{0}(y,\sigma_0)\ind{\vartheta[\gamma]=  0} 
   \right .\\
  \left .\; \bigg| \; \gamma(s) = y + \int_0^s \sigma(s')\, ds'   \, , \; \gamma(t) = x  \,, \; \sigma_N = v \right \}\,,
\end{multline}
Then $U$ is the unique viscosity solution, in the sense of Definitions \ref{def:subsol} and \ref{def:supersol}, of the limiting obstacle problem with   initial data $\min\left(u_0,\minp u_{0} + |v|^{2}/2\right)$. 
\end{prop}

\begin{proof}
By choosing the trivial stopping time $\vartheta\equiv t$, we immediately have 
\begin{equation}\label{eq:contraint Uhat}
U(t,x,v) \geq \frac{v^2}{2}\,.
\end{equation}

The rest of the  proof relies on the dynamic programming principle, as in Section \ref{sec:variational}, but for stopping times taking values in $[s,t]$: 
\begin{multline}
\label{eq:obstacle variational bis}
U  (t,x,v) = 
  \sup_{\vartheta} \inf_{(y,\sigma)} \left \{  \mA_{  \vartheta[\gamma]}^{t}[\sigma] +  \frac{|\sigma(\vartheta[\gamma])|^2}{2}\ind{\vartheta[\gamma]> s} +   U (s,y,\sigma_0)\ind{\vartheta[\gamma]=  s} 
   \right .\\
  \left .\; \bigg| \;  \gamma(s') = y + \int_s^{s'} \sigma(s'')\, ds''   \, , \; \gamma(t) = x  \,, \; \sigma_N = v \right \}\, .
\end{multline}

Let us define $U^*$ the upper semi-continuous envelope of $U$. We first check  (ii) in the definition of the sub-solution, that is, such that the constraint $U^*\leq \minp U^* + {v^2}/{2}$ is satisfied.  Let $\tau>0$, and $s = t - \tau$.
Let $\vartheta^*$ be a nearly optimal stopping time (we consider it to be optimal for the sake of conciseness). 
Let $\gamma$ such that $\sigma_0 = w$ on $(s,(s+t)/2)$ and $\sigma_1 = v $ on $[(s+t)/2,t]$. We distinguish between two cases. If $\vartheta^*[\gamma] \geq (s+t)/2$, then we have 
\begin{equation*}
U(t,x,v) \leq  \tau + \dfrac{v^2}{2}\,,
\end{equation*}
and since $U\geq 0$, we deduce that 
\begin{equation*}
U(t,x,v) \leq \tau + \minp U^*(t,x) + \dfrac{v^2}{2}\,.
\end{equation*}
On the contrary, if $\vartheta^*[\gamma] < (s+t)/2$, then using \eqref{eq:contraint Uhat}, we find that 
\begin{align*}
U(t,x,v) 
&\leq  \tau + \dfrac{v^2}{2} + \left ( \dfrac{w^2}{2} \ind{\vartheta[\gamma]> s} +   U (t - \tau,y,w)\ind{\vartheta[\gamma]=  s} \right )\\
& \leq \tau + \dfrac{v^2}{2} + U (t - \tau,y,w)\left ( \ind{\vartheta[\gamma]> s} +  \ind{\vartheta[\gamma]=  s}\right )  \\
&  \leq \tau + U (t - \tau,y,w) + \dfrac{v^2}{2}.
\end{align*}
By letting $\tau\to 0$, we find that $U^*$ verifies the parabolic constraint. 

Let now $(\phi, \psi)$ be a pair of test functions for $(U^*, \minp U^*)$  as in Definition \ref{def:subsol}(iii). 
Assume first that $U^{*}(t_{0},x_{0},v_{0}) >|v_0|^{2}/2$. 
We may choose the constant configuration $\sigma \equiv v$ in \eqref{eq:obstacle variational bis}, so as to obtain 
\begin{align*}
U(t ,x ,v) & \leq  \sup_{\vartheta\in \Theta}  \left( \tau- \vartheta[\gamma] +  \frac{|v|^2}{2}\ind{\vartheta[\gamma]> s} +   U (t-\tau, x-\tau v ,v)\ind{\vartheta[\gamma]=  s } \right) \\
& \leq  \max \left\{ \tau +  \frac{|v|^2}{2},\tau+   U^{*} (t-\tau, x-\tau v ,v) \right\}  
\end{align*}
in the neighbourhood of $(t_0,x_0,v_0)$. Taking the upper semicontinuous envelope, we obtain that 
\begin{equation*}
U^*(t_0,x_0,v_0)  \leq \tau +  \max \left\{    \frac{|v_{0}|^2}{2},   U  (t_{0} - \tau, x_{0} - \tau v_0 ,v_{0}) \right\}\,,
\end{equation*}
The second term is larger than the first one when $\tau$ is small enough since $U^{*}(t_{0},x_{0},v_{0}) >|v_0|^{2}/2$. 
Therefore we find:
\begin{equation*}
\phi(t_0,x_0,v_0) - \phi(t_0-\tau,x_0 - \tau v_0,v_0)
\leq 
U^*(t_0,x_0,v_0) - U^*(t_0-\tau,x_0 - \tau v_0,v_0) \leq \tau \end{equation*}
by definition of $\phi$. Dividing by $\tau$ and letting $\tau \to 0$ in the previous inequalities, 
we eventually obtain $\partial_t \phi + v_0\cdot \nabla_x \phi - 1 \leq 0$ at $(t_0,x_0,v_0)$. 

Next, as $U^*\leq \minp U^* + v^2/2$, we know that the minimum of $U^*$ is attained at $v = 0$. We have
\begin{equation}\label{eq:721}
 \psi(t_0,x_0) - \psi(t_0-\tau,x_0 ) \leq U^*(t_0,x_0,0) - U^*(t_0-\tau,x_0 ,0) . 
\end{equation}
Choose the step function $\sigma$, such that $\sigma  = 0$ on $(t-\tau,t-\delta) $, and $\sigma = v$ on $[t-\delta,t]$ with arbitrarily small $\delta$. We get
\begin{equation*}
U(t,x,v)  \leq  \sup_{\vartheta }  \left( \delta + \frac{v^2}{2} \ind{\vartheta[\gamma]\geq t-\delta} + \frac{v^2}{2} \ind{\vartheta[\gamma]\in (s,t-\delta)}  +U(t-\tau,x  ,0) \ind{\vartheta[\gamma]=  t -\tau}\right)\, ,  
\end{equation*}
in the neighbourhood of $(t_0,x_0,0)$. 
Taking the upper semicontinuous envelope, we obtain in particular that 
\begin{equation*}
U^*(t_0,x_0,0)  \leq  U^*(t_0-\tau,x_0  ,0)  \, ,  
\end{equation*}
Together with \eqref{eq:721}, we conclude that $\partial_t \psi(t_0,x_0) \leq 0$. 

This shows that $U^{*}$ is a subsolution. 

\bigskip

We continue with the criteria for the viscosity supersolution. Let $U_*$ be the lower semicontinuous envelope of $U$. Let $(\phi, \psi)$ be a pair of test functions for $(U_*, \minp U_*)$  as in Definition \ref{def:supersol}(ii). Let $(t^n, x^n, v^n)$ be a minimizing sequence in the neighbourhood of $(t_0,x_0,v_0)$ such that $U(t^n, x^n, v^n)$ converges to $U_*(t_0,x_0,v_0)$. Let $\tau > 0$, consider the trivial stopping time $\vartheta^n \equiv t^n - \tau$
and let $\sigma^n$ be a nearly optimal trajectory for \eqref{eq:obstacle variational bis}: 
\begin{equation*}
\frac1 n+ U(t^n, x^n, v^n) \geq  U(t^n-\tau , y^n, \sigma_0^n) + \mA_{t^n-\tau}^{t^n}[\sigma^n]\, ,
\end{equation*}
where $x^{n}=y^{n}+\int_{t^{n}-\tau}^{t^{n}}\sigma^{n}$. Assume that $U_*(t_0,x_0,v_0) <  \minp U_*(t_0, x_0)  +\frac12 |v_0|^2$. Let us recall, from the proof of Theorem \ref{th:hopf lax intro} in Section \ref{sec:variational}, that the optimal trajectory is constant beyond a certain range $n$, and accumulates around $v_0\neq 0$. Therefore, the discount factor $-r$  is not seen in \eqref{eq:new action}. Then, we can apply exactly the same arguments as in the proof  of Theorem \ref{th:hopf lax intro} to show that $\partial_t \phi(t_0,x_0,v_0) + v_0\cdot \nabla_x \phi(t_0,x_0,v_0) - 1 \geq 0$.

If $\aminp U_*(t_0,x_0) = \left\lbrace 0 \right\rbrace$, then the proof of $\partial_t \psi(t_0,x_0)\geq -r$ goes exactly as in that of Theorem \ref{th:hopf lax intro}, using the inequality obtained in \eqref{eq:obstacle variational bis} with the trivial stopping time $\vartheta^n \equiv t^n - \tau$, except that the modified action function provides the required $\partial_t \psi(t_0,x_0)\geq -r$ instead of $0$. We recall the main computation, omitting much of the details:
\begin{align*}
& \frac1n +  U(t^n, x^n, v^n) - \minp U_*(t_0, x_0)\\ 
& \geq \psi(t^n-\tau,y^n) -  \psi(t_0,x_0)  + \mA_{t^n-\tau}^{t^n}[\sigma^n] \\
&\geq \psi(t^n,x^n) -  \psi(t_0,x_0)  -\tau \partial_t \psi(t^n,x^n) - |x^n - y^n| \left | \nabla_x\psi(t^n,x^n) \right | + \mA_{t^n-\tau}^{t^n}[\sigma^n] \\
&\hspace{280pt}  - \tau\omega(\tau) - |x^n - y^n|\omega(x^n - y^n)  .
\end{align*}
We have \eqref{eq:new action}:
\begin{equation*}
\mA_{t^n-\tau}^{t^n}[\sigma^n] \geq \dfrac12 \max_{i=1..N} |\sigma_i^n|^2 - r\tau \,,  
\end{equation*}
and, just as previously \eqref{eq:estimate displ}--\eqref{eq:Young}
\begin{align*}
|x^n - y^n| \left | \nabla_x\psi(t^n,x^n) \right | 
&\leq 
\tau \max \left ( |\sigma_0^n|, \max_{i=1..N} |\sigma_i^n| \right ) \left | \nabla_x\psi(t^n,x^n) \right |\\
& \leq \tau \left ( |\sigma_0^n| + \max_{i=1..N} |\sigma_i^n| \right ) \left | \nabla_x\psi(t^n,x^n) \right | \\
&\leq  
 \tau   |\sigma_0^n|  \left | \nabla_x\psi(t^n,x^n) \right | + \dfrac12 \max_{i=1..N} |\sigma_i^n|^2 + \frac{\tau^2}2  \left | \nabla_x\psi(t^n,x^n) \right |^2  \,.
\end{align*}
We deduce as previously (but the $-r\tau$ additional contribution) that
\begin{align*}
& \frac1n + U(t^n, x^n, v^n)  - \minp U_*(t_0, x_0) \\
&\geq \psi(t^n,x^n) - \psi(t_0,x_0)  - \tau \partial_t \psi(t^n,x^n) -  \tau |\sigma_0^n|  \left | \nabla_x\psi(t^n,x^n) \right |  - \frac{\tau^2}2 \left\vert \nabla_x\psi(t^n,x^n) \right\vert^2 - r\tau  \\
&\hspace{280pt}  - \tau \omega(\tau)  - |x^n - y^n|\omega(x^n - y^n)\, .
\end{align*}
We conclude that $\partial_t \psi(t_0,x_0)\geq -r$ by letting $n\to +\infty$, then $\tau \to 0$ as in Section \ref{sec:variational}, where it was established that $|\sigma_0^n| $ gets arbitrarily small as $\tau \to 0$.  This concludes the proof of Proposition~\ref{prop:varobs}.
\end{proof}


We define $\hat{U}$ and $\check{U}$ from the representation formula \eqref{eq:obstacle variational bis 0} associated with the initial data $\hat{u}_{0}$ and $\check{u}_{0}$. 
These functions satisfy the hypotheses of Proposition \ref{prop:varobs} and Theorem \ref{HJlimit:obs}. Hence, we have $\check{u}^\eps \to \check{U}$ and 
 $\hat{u}^\eps \to \hat{U}$ as $\eps\to 0$.

\medskip

\noindent{\bf \# Step 4:  Identification of the limiting value $\mU(t,x,v)$.} 

\medskip

Recalling the comparison   $\hat{u}^\eps(t,x,v) \leq u^\eps(t,x,v)\leq \check{u}^\eps(t-\tau,x,v)$ \eqref{eq:comparison u hat check}, we deduce that 
\begin{equation*}
\begin{cases}
&\displaystyle\liminf_{\eps\to 0} u^\eps(t,x,v) \geq  \hat{U}(t,x,v)\medskip\\
&\displaystyle\limsup_{\eps\to 0} u^\eps(t,x,v) \leq  \check{U}(t-\tau,x,v)
\end{cases}
\end{equation*}
Based on the representation formula, one can easily prove that  $\hat{U}\to \mU$ as $A\to +\infty$ and $\alpha \to 0$, and  $\check{U}(\cdot-\tau,\cdot,\cdot)\to \mU$ as $\eps_0\to 0$, then $\tau\to 0$. We thus deduce  that 
\begin{equation*}
\lim_{\eps\to 0} u^\eps(t,x,v) = \mU(t,x,v)\, .
\end{equation*}

\bigskip

We aim to compute the value of $\mU(t,x,v)$ \eqref{eq:obstacle variational}.  This belongs to the class of two-players games \cite{evans_differential_1984,
evans_pde_1989}. Alice plays the stopping time strategy $\vartheta$. She wishes to maximize her reward. Bob plays the piecewise linear curve $\gamma$ parametrized by $(y,\sigma)$ on $(0, t]$.
He wishes to minimize his reward. 

We restrict to the final velocity  $\sigma_N = v = 0$ as the spreading of the population relies upon the value of $\minp \mU = \mU(t,x,0)$ only, see Remark \ref{rm:rem}. 

We denote by $\mu(t,x)$ the minimizing value in the absence of stopping time (alternatively speaking, Alice plays the trivial strategy $\vartheta\equiv 0$):
\begin{equation}\label{eq:mu 7}
\mu(t,x) = \inf_{(y,\sigma)} \left \{  \mA_{0}^{t}[\sigma]   + \left ( \mathbf{0}_{y\leq 0} +\frac{|\sigma_0|^2}{2}\right )   \right.\\ 
\left .\; \bigg| \; \gamma(s) = y + \int_0^s \sigma(s')\, ds'  \, , \; \gamma(t) = x  \,, \; \sigma_N = 0 \right \}\,,  
\end{equation}
It is a consequence of Section \ref{sec:kernel} that $\mu(t,x)$ is given as in \eqref{eq:mumu}.
Indeed, the problem with positive $r>0$ without saturation (hence, without obstacle) is equivalent to the problem $r=0$ under an appropriate change of unknown. In fact, the actions \eqref{eq:action-intro} and \eqref{eq:new action bis} are in one-to-one correspondance  by the change of time $t_0 = (1+r) t_r$, and accordingly the change of space $x_0 = (1+r)x_r$ (velocity is unchanged), up to the additive contribution $rt$.  Back to the minimal action starting from the origin $y=0$ (clearly the best choice here), that is \eqref{eq:candidate 3}, we recover \eqref{eq:mumu} by \eqref{eq:mu 7}:
\begin{equation*}
\mu(t,x) = \begin{cases}
\dfrac32 |(1+r)x|^{2/3} - rt  & \text{ if } 0\leq x\leq (1+r)^{1/2}t^{3/2},\medskip\\
\dfrac{\vert x \vert^2}{2 t^2} + t  & \text{ if } x\geq (1+r)^{1/2}t^{3/2}. 
\end{cases}
\end{equation*}

We claim that the best strategy of Alice depends on the sign of $\mu$.

\medskip 
\noindent {\bf $\bullet$ Case 1: $\mu(t,x)\geq 0$. Alice's best strategy is $\vartheta\equiv 0$.} Suppose there is an alternative strategy $\vartheta_0$ which provides her with a better reward. 
Bob can always play the optimal curve $\gamma_*$ in the absence of stopping time. Recall from Section \ref{sec:kernel1D} that there is one intermediate time $s_*\in [0,t]$ which realizes the minimum of $\frac{x^2}{2s^2} + s - r(t-s)$, and the minimal value is precisely $\mu(t,x)$. Then $s_* = \min(t,((1+r)x^2)^{1/3})$.  
Suppose that $s_*<t$, then
\begin{equation*}
\gamma_*(s) = \begin{cases}
s \dfrac{x}{s_*} & \text{if}\quad s <s_*\medskip\\
x & \text{if}\quad s \geq s_*
\end{cases}
\end{equation*}
When considering the stopping time $\vartheta_0[\gamma_*]$ for this specific trajectory, we have the following alternative: if $\vartheta_0\geq s_*$ then the reward is $-r(t-\vartheta_0) \leq 0 \leq \mu(t,x)$, whereas if $\vartheta_0< s_*$ then the reward is $(s_* - \vartheta_0) - r(t-s_*) + \frac{x^2}{2(s^*)^2} = \frac{x^2}{2(s_*)^2} + s_* - r(t-s_*) - \vartheta_0 = \mu(t,x) - \vartheta_0\leq \mu(t,x) $. Hence, the stopping time $\vartheta_0$ induces a better reward for Bob.  

Suppose on the contrary that $s_* = t$. Then, $\gamma_*(s) = s \frac{x}{t}$. Consequently, the reward is 
$ (t - \vartheta_0) + \frac{x^2}{2t^2} = \mu(t,x) - \vartheta_0$, which is again better for Bob since $\vartheta_0 \geq 0$.


\medskip 
\noindent {\bf $\bullet$ Case 2: $\mu(t,x)\leq 0$. Alice's best strategy is $\vartheta\equiv t$.} 

This results obviously in a zero reward.  Suppose there is an alternative strategy $\vartheta_0$ which provides her with a positive reward. Then, the same reasoning as in Case 1 can be followed, with Bob's same strategy. If $\vartheta_0\geq s_*$, then the reward is non-positive. If $\vartheta_0< s_*$, then the reward is simply $\mu(t,x) - \vartheta_0$, which is worse for Alice. 

As a conclusion, we have $\mU(t,x,0) = \max(0,\mu(t,x))$. 

This concludes the proof of Theorem \ref{th:nonlin}. 
\end{proof}

We conclude this article by giving the proof of the propagation result.
\begin{proof}[Proof of Theorem \ref{th:propagation}]
Clearly, we have $\{(t,x): x> \Upsilon t^{3/2}\} =\{(t,x): \mu(t,x)>0\} $ by definition of the constant $\Upsilon$ \eqref{eq:Upsilon}.
Let $K$ be a compact subset of $\mathrm{Int}\:\{\mu>0\}$. There exists a constant $\delta>0$ such that $\mu>\delta$ on $K$.
Recall that we have used the method of half-relaxed limits, so that convergence of $u^\eps \to \mU$ is locally uniform, see {\em e.g.} \cite[Lemma 6.2]{achdou_hamilton-jacobi_2013}. Since $\mu$ is the partial minimum $\mu(t,x) = \minp \mU(t,x) = \mU(t,x,0)$ by definition,  we can find $\eps_0>0$ such that $u^\eps(t,x,v)>\delta/2$ for $(t,x,v)\in K\times \mB(0,R)$ for some large $R$, and $\eps<\eps_0$.
We can deal alternatively with the large velocities outside $\mB(0,R)$, since $f^\eps \leq M_\eps$, as follows:
\begin{equation*}
\rho_\eps(t,x) \leq \int_{|v|\leq R} f^\eps(t,x,v)\, dv + \int_{|v|>R} f^\eps(t,x,v)\, dv \leq 2R \exp\left (-\frac{\delta}{2\eps}\right ) +\int_{|v|>R} M_\eps(v)\, dv \, .
 \end{equation*} 
We deduce that $\rho^\eps(t,x)\to 0$ as $\eps\to 0$.

In order to prove the opposite, we adapt the technique of \cite[Section 4]{evans_pde_1989} in the context of Section \ref{sec:Conv} (sub-solution step). Let $K$ be a compact subset of $\mathrm{Int}\:\{\mu<0\}$, and $(t_0,x_0)\in K$.  
We define the test function $\phi(t,x,v) = \frac12(t-t_0)^2 + \frac12(x - x_0)^2 + \left (\frac{1+\delta}2\right ) v^2$ for $\delta>0$ arbitrary.  Clearly, $\mU - \phi$ attains a local strict maximum at $(t_0,x_0,0)$, which is global with respect to velocity, since $\mU(t,x,v) =  v^2/2$ in the neighbourhood of $(t_0,x_0)$. Hence, there exists $(t_\eps,x_\eps,v_\eps)\to (t_0,x_0,0)$ such that $u^\eps - \phi$ admits a local maximum at  $(t_\eps,x_\eps,v_\eps)$ with respect to $(t,x)$, which is also global with respect to $v$. Plugging the first order condition in the equation for $u^\eps$, namely,
\begin{multline*}
\partial_t \u^{\eps}(t,x,v) + v  \partial_x  \u^{\eps}(t,x,v) - 1 \\= - \frac{1+r}{(2 \pi\eps)^{1/2} }
\int_{\R} \exp \left( \frac{\u^{\eps}(t,x,v) - \u^{\eps}(t,x,v') - \vert v\vert^2/2 }{\e} \right) dv' + r \rho^\eps(t,x)\, ,
\end{multline*}
we find
\begin{align*}
&\partial_t \phi(t_\eps,x_\eps,v_\eps) + v_\eps \partial \phi_\eps(t_\eps,x_\eps,v_\eps) - 1 \\
&\qquad\qquad\qquad \leq - \frac{1+r}{(2 \pi\eps)^{1/2} }
\int_{\R} \exp \left( \frac{\phi(t_\eps,x_\eps,v_\eps) - \phi(t_\eps,x_\eps,v') - \vert v_\eps\vert^2/2 }{\e} \right) dv' + r \rho^\eps(t_\eps,x_\eps)\\
&\qquad\qquad\qquad  = - \frac{1+r}{(2 \pi\eps)^{1/2} }
\int_{\R} \exp \left(  -\left (\frac{1+\delta}{2\e}\right ) |v'|^2 + \frac\delta{2\e} \vert v_\eps\vert^2   \right) dv' + r \rho^\eps(t_\eps,x_\eps)
\end{align*}
We obtain that
\begin{equation*}
(t_\eps - t_0) + v_\eps (x_\eps - x_0) - 1 \leq - \frac{1+r}{(1+\delta)^{1/2}}  + r \rho^\eps(t_\eps,x_\eps)\, .
\end{equation*}
On the other hand, since $u^\eps(t_\eps,x_\eps,v) - u^\eps(t_0,x_0,v) \geq \phi(t_\eps,x_\eps,v) - \phi(t_0,x_0,v)\geq 0$ by the definition of the local maximum (global with respect to $v$), we have $\rho^\eps (t_\eps,x_\eps)\leq \rho^\eps(t_0,x_0)$. 
By letting $\eps \to 0$, we get eventually that 
\begin{equation*}
\liminf_{\eps\to 0}\rho^\e(t_0,x_0) \geq \frac1r \left (\frac{1+r}{(1+\delta)^{1/2}} - 1 \right )\, .
\end{equation*}
As $\delta$ is arbitrary small, we deduce that $\liminf \rho^\eps(t_0,x_0) \geq 1$.  
\end{proof}

\appendix

\section{Extension to quadratic spatial growth}

\label{sec:unbded}

Being given the shape of the function $L$, see \eqref{eq:candidates all} in the one-dimensional case, which admits quadratic growth, both in space and velocity, at any positive time,  it is natural to investigate the well-posedness of \eqref{eq:limit} under more general growth conditions than \eqref{eq:v2 plus borne}. This is the purpose of the next result. This extension is also required in Section \ref{sec:acc} to cope with unbounded initial data $u_0$, which are very natural for the analysis of propagation phenomena.

\begin{thm}[Comparison principle]\label{theo:comp2}
Let $\underline{u}$ (resp. $\overline{u}$) be a viscosity sub-solution (resp. super-solution) of \eqref{eq:limit} on $(0,T)\times\R^{2n}$ with  continuous  initial data $\underline{u}_0 \leq \overline{u}_0$. Assume that  there exists  some $A>1$ such that for all $(t,x,v)$,
\begin{equation} \label{eq:growth condition}
\frac1A |v|^2 - A|x|^2 - A \leq \underline{u}(t,x,v), \overline{u}(t,x,v) \leq A |v|^2 + A|x|^2 + A\, .
\end{equation}
Then $\underline{u} \leq \overline{u}$ on $(0,T) \times \R^{2n}$.
\end{thm}

An immediate consequence of \eqref{eq:growth condition} is that any partial minimum with respect to velocity,
say attained at $v_0$, satisfies
\begin{equation} |v_0|^2 \leq 2A^2\left(  |x|^2 +  1 \right)\,
  . \label{eq:localization minimum x} 
  \end{equation}

\begin{proof}
We do not reproduce all the details of the proof of the comparison principle, but only indicate the
major changes that have to be made in comparison with Section \ref{sec:Comp}. 
The main discrepancy
concerns the localization of minima with respect to the velocity
variable. Indeed, they are no more confined uniformly with respect
to $x$ as in Section \ref{sec:Comp}, see \eqref{eq:localization minimum x}. We have to adapt the penalization terms in  $\wtildechi$ accordingly. 
As a side effect, a short time condition $T<T_0$ is required. However, this does not affect the conclusion since we can iterate the comparison principle on time intervals of length $T_0/2$. Let us define $T_0$ as follows, together with an auxiliary parameter $B$:
\begin{equation*} 
T_0 = \dfrac{1}{4A}\, , \quad B = \frac{4A^2}{1 - 4A^2 T_0^2} = \frac{16}{3}A^2\, .
\end{equation*}
Let $T<T_0$. Let $\kappa<1$, $\eps>0$, $\alpha>0$, $R>0$. Let $\delta>0$, $\gamma>0$, $B>0$ and $\Lambda>0$ to be suitably chosen below. We cook up the functions with twice the number of variables (except velocity) as in Section  \ref{sec:Comp}:
\begin{multline*}
\wtildechi(t,x,s,y,v) = \kappa\underline{b}(t,x,v) - \overline{b}(s,y,v) -   \frac\delta4 e^{ \gamma t}\left (  |x|^4 + \vert y \vert^4 \right) - \alpha\left( \dfrac1{T- t} + \dfrac1{T- s}\right) 
\\  - \frac1{2\eps} \left(|t -s|^2 + |x-y|^2\right) - \Lambda\left(  |v|^2 -
  B|x - tv|^2  - R^2\right)_+ \, .
\end{multline*}
\begin{multline*}
\overlinechi(t,x,s,y) = \kappa \underline{m}(t,x) - \overline{m}(s,y) -   \frac\delta4 e^{\gamma t}\left (  |x|^4 + \vert y \vert^4 \right) -  \alpha \left( \dfrac1{T- t} + \dfrac1{T- s}\right)
\\   - \frac1{2\eps} \left(|t -s|^2 + |x-y|^2\right)\, .
\end{multline*}
The quadratic penalty term in \eqref{eq:chiapp} has been turned into a quartic
one in order to ensure the existence of a minimum with respect to
space variable. 
The exponential prefactor is inspired from \cite[page 72]{achdou_hamilton-jacobi_2013}.
Also, the penalization with respect to velocity has
been extended in order to take into account the nonuniform velocity
confinement. The particular dependency on $x - tv$ is chosen in order to be transparent when applying the free transport operator $\partial_t + v\cdot \nabla_x$. 

Finally, we define the  maximum values:
\begin{equation*} 
\omega = \max_{((0,T) \times \R^n)^2}\overlinechi \,, \quad \Omega = \max_{((0,T) \times \R^n)^2 \times \R^n} \wtildechi \,.
\end{equation*}

To establish that $\omega\leq \Omega$, as in Lemma \ref{lem:Oomega}, it is enough to check that the
penalty term  $\left(  |v|^2 -  B|x - tv|^2  - R\right)_+$
vanishes at the minima of   $\overline{u}$ with
respect to velocity. This is indeed a consequence of \eqref{eq:localization
  minimum x} and the short time condition. 
 
\begin{lem}\label{lem:minconf A} 
Let $T<T_0$. Assume that $R \geq 2A \left(1 +  T_0^2 B\right)^{1/2}$.
For all $(t,x) \in (0,T) \times \R^{n}$, the velocity penalization vanishes on $\aminp  \overline{u} (t,x)$, that is
\begin{equation*}
 \forall v_0\in \aminp \overline{u} (t,x) \quad \left(|v_0|^2 - B|x-tv_0|^2 - R^2\right)_+ = 0.
\end{equation*}
\end{lem}

\begin{proof}
  Let $v_0$ be a
minimum, we have:
\begin{align*}
|v_0|^2 - B|x-tv_0|^2 - R^2 &\leq \left(1 - t^2 B\right) |v_0|^2
\tg{+} 2 B t x\cdot v_0  - B |x|^2  - R^2 \\
& \leq \left(  1+  t^2 B\right) |v_0|^2
-\frac12 B |x|^2  - R^2\\
& \leq 2A^2\left(1 +  T_0^2 B\right) \left( |x|^2 + 1 \right) -\frac12 B |x|^2  - R^2 \, ,
\end{align*} 
where the last inequality comes from \eqref{eq:localization minimum x}.
Hence, we get
\begin{equation*}
|v_0|^2 - B|x-tv_0|^2 - R^2  \leq \frac12\left ( 4 A^2 
+ \left ( 4 A^2 T_0^2 - 1 \right ) B \right )|x|^2 + 2A^2\left (1 +  T_0^2 B\right)  
 - R^2\, .
\end{equation*}
The right-hand-side is certainly nonpositive thanks to our choice of $B$ and $R$.
\end{proof}

The estimates on $(\tx,\ty,\tv)$ read as follows:
\begin{equation*}
\begin{cases}
&\delta \max\left (  |\tx|^4,|\ty|^4\right ) \leq 16A \left ( |\tv|^2 +  \max\left (  |\tx|^2,|\ty|^2\right ) + 1 \right )\\
&\Lambda \left(  |\tv|^2 -  B|\tx - \tt\tv|^2  - R\right) \leq 16A \left ( |\tv|^2 +  \max\left (  |\tx|^2,|\ty|^2\right ) + 1 \right )\,, 
\end{cases}
\end{equation*}
Expanding the left-hand-side, we find 
\begin{align*}
\Lambda \left ( 1 - 2 B T_0^2\right ) |\tv|^2  - 16 A |\tv|^2 &\leq 2 \Lambda B |\tx|^2 + \Lambda R^2 +  16A \left (   \max\left (  |\tx|^2,|\ty|^2\right ) + 1 \right )\\
\left (\frac13\Lambda - 16 A\right )  |\tv|^2  & \leq 
 16 \Lambda A^2  |\tx|^2 + \Lambda R^2 +  16A \left (   \max\left (  |\tx|^2,|\ty|^2\right ) + 1 \right )
\end{align*}
It remains to choose $\Lambda = 60 A$, so  as to get
\begin{equation}\label{eq:tv estimate}
|\tv|^2  \leq C \left ( R^2 +   \max\left (  |\tx|^2,|\ty|^2\right ) + 1 \right )\, ,
\end{equation} 
for some constant $C$ depending only on $A$. Next, we deduce some useful compactness estimate: 
\begin{equation*}
\delta  \max\left (  |\tx|^4,|\ty|^4\right )  \leq C \left ( R^2 +  \max\left (  |\tx|^2,|\ty|^2\right ) + 1  \right )\, .
\end{equation*}

We are ready to perform the comparison argument. If $(0,0)$ is an accumulation point of $(\tt,\ts)$ as $\eps \to 0$, then the same argument as in Section \ref{sec:Comp} can be reproduced. 

Otherwise we distinguish between two cases:

\medskip
\noindent{\bf \# Case 1:  $\overline{b}(\ts,\ty,\tilde v) < \underset{w \in \R^n}{\min}\left(\overline{b}(\ts,\ty,w) + \frac{\vert w \vert^2}{2} \right) =    \minp \overline{u}(\ts,\ty)$.} 
\medskip

In this case, $\partial_t \underline{b} + v\cdot \nabla_x \underline{b} - 1 \leq 0$ and $\partial_t \overline{b} + v\cdot \nabla_x \overline{b} - 1 \geq 0$ in the viscosity sense.
We first use the test function
\begin{multline*}
\phi_2(s,y,v) = \kappa\underline{b}(\tt,\tx,v) -  \frac\delta4 e^{\gamma \tt } \left( |\tx|^4 + \vert y \vert^4 \right) - \alpha\left( \dfrac1{T- \tt} + \dfrac1{T- s}\right) \\ - \frac1{2\eps} \left(|\tt-s|^2 + |\tx-y|^2\right)  - \Lambda\left(  |v|^2 -
  B|\tx - \tt v|^2  - R^2\right)_+,
\end{multline*}
associated to the supersolution $\overline{b}$ at the point $(\ts,\ty,\tilde v)$. By using  Definition  \ref{def:supersol intro} of a super-solution, this yields
\begin{equation}
-\dfrac\alpha{(T - \ts)^2}  + \frac1\eps (\tt -\ts) + \tilde  v \cdot \left( - \delta e^{\gamma \tt } |\ty|^2 \ty - \frac1\eps ( \ty- \tx)\right) - 1 \geq 0\, . \label{eq:chain rule 1.1 sec 6}
\end{equation}
On the other hand, using the test function
\begin{multline*}
\phi_1(t,x,v) = \overline{b}(\ts,\ty,v) + \frac\delta4 e^{\gamma t} \left( |x|^4 + \vert \ty \vert^4 \right) + \alpha \left( \dfrac1{T- t} + \dfrac1{T- \ts} \right) \\  + \frac1{2\eps} \left(|t-\ts|^2 + |x-\ty|^2\right) + \Lambda\left(  |v|^2 -
  B|x - t v|^2  - R^2\right)_+\, ,
\end{multline*}
associated to the subsolution $\kappa \underline{b}$ at the point $(\tt,\tx,\tilde v)$, we obtain
\begin{equation}
\dfrac\alpha{(T - \tt)^2} + \frac{\delta \gamma}4 e^{\gamma \tt } \left( |\tx|^4 + \vert \ty \vert^4 \right) + \frac1\eps (\tt - \ts) +\tilde  v \cdot \left( \delta e^{\gamma \tt } |\tx|^2 \tx +  \frac1\eps ( \tx - \ty) \right) - \kappa \leq 0\,,
\label{eq:chain rule 2.1 sec 6}
\end{equation}
by using Definition  \ref{def:subsol intro} of a sub-solution. By substracting \eqref{eq:chain rule 2.1 sec 6} to \eqref{eq:chain rule 1.1 sec 6}, we obtain
\begin{align*}
& \alpha \left( \dfrac1{(T - \tt)^2} + \dfrac1{(T - \ts)^2} \right) + \frac{\delta \gamma}4 e^{\gamma \tt }  \left( |\tx|^4 + \vert \ty \vert^4 \right) + \delta e^{\gamma \tt }\tilde  v \cdot \left( |\tx|^2 \tx  + |\ty|^2 \ty  \right) + (1-\kappa) \leq 0,\\
& \frac{\delta\gamma}4 e^{\gamma \tt } \left (  |\tx|^4 + \vert \ty \vert^4 - \frac{8|\tv|}\gamma  \max\left ( |\tx|^3,|\ty|^3 \right )
\right )   + 1 - \kappa \leq 0\, .
\end{align*}
Using  the linear bound of $|\tv|$ on $ |\tx|,|\ty|$ \eqref{eq:tv estimate}, we can choose $\gamma$ large enough depending on $A$ and $R$ to get a contradiction.

\medskip

\noindent{\bf \# Case 2:  $\overline{b}(\ts,\ty,\tilde v) \geq \underset{w \in \R^n}{\min}\left(\overline{b}(\ts,\ty,w) + \frac{\vert w \vert^2}{2} \right)  = \minp \overline{u}(\ts,\ty)$.} 

\medskip

The arguments of Section \ref{sec:Comp} can be reproduced, with the help of Lemma \ref{lem:minconf A} (no penalty on $\aminp \overline{u}(\ts,\ty) $), so as to get $\aminp \overline{u}(\ts,\ty) = \{0\}$. We conclude as before that $(\tt,\tx,\ts,\ty)$ realizes the maximum value of $\overlinechi$, allowing to use the viscosity conditions on $\minp u$. The last step of Section \ref{sec:Comp} is modified as follows:
the first test function is again $\psi_1(t,x) = \phi_1(t,x,0)$, and the subsolution criterion is:
\begin{equation*}
\dfrac\alpha{(T - \tt)^2} + \frac{\delta \gamma}4 e^{\gamma \tt } \left( |\tx|^4 + \vert \ty \vert^4 \right) + \frac1\eps (\tt - \ts) \leq 0\, ,
\end{equation*} 
whereas the super-solution viscosity criterion is as before:
\begin{equation*}
-\dfrac\alpha{(T - \ts)^2}  + \frac1\eps (\tt -\ts) \geq 0\, .
\end{equation*}
This is a contradiction as $\alpha>0$.
\end{proof}

\printbibliography

\end{document}